\documentclass[12pt]{amsart}

\RequirePackage{amsthm,amsmath}
\RequirePackage[numbers]{natbib}
\RequirePackage[colorlinks,citecolor=blue,urlcolor=blue]{hyperref}

\usepackage{graphicx}
\graphicspath{{./figures/}}
\usepackage{amsfonts}
\usepackage{amssymb}
\usepackage{amsbsy}
\usepackage{epsfig}
\usepackage{natbib, mathrsfs}
\usepackage{verbatim}
\usepackage[latin1]{inputenc}
\usepackage{mhequ}
\usepackage[ruled]{algorithm2e}
\usepackage{algorithmic}
\usepackage{tikz}

\makeatletter
\newcommand{\RemoveAlgoNumber}{\renewcommand{\fnum@algocf}{\AlCapSty{\AlCapFnt\algorithmcfname}}}
\newcommand{\RevertAlgoNumber}{\algocf@resetfnum}
\makeatother

\numberwithin{equation}{section}

\def \be{\begin{equs}}
\def \ee{\end{equs}}
\def \P{\mathbb{P}}
\def \E{\mathbb{E}}
\def \wep{{\sl w.e.p.}}

\def \Vol{\mathrm{Vol}}

\newcommand \bad[2]{\mathcal{B}(#1,#2)}
\def \wit{\mathcal{W}}

\def \err{\mathfrak{D}}
\def \eurt{\sigma_{\mathrm{eurt}}}
\def \alg{\mathfrak{F}}

\def \csketch{5}
\def \cbad{}

\def \cat{2}
\def \cafo{}
\def \caft{2}
\def \cafi{4}
\def \cafti{3}

\def \cub{1}
\def \cfin{5}
\def \czeta{4}
\def \cS{3}

\newcommand \GA[1]{G_{#1}^{(2)}}
\newcommand \WA[1]{w_{#1}^{(2)}}

\newtheorem{theorem}{Theorem}[section]
\newtheorem{lemma}[theorem]{Lemma}
\newtheorem{remarks}[theorem]{Remarks}
\newtheorem{corollary}[theorem]{Corollary}
\newtheorem{prop}[theorem]{Proposition}
\newtheorem{assumptions}[theorem]{Assumptions}
\newtheorem{assumption}[theorem]{Assumption}

\theoremstyle{plain}
\newtheorem{thm}{Theorem}
\newtheorem*{thm-non}{Theorem}

\theoremstyle{definition}
\newtheorem{defn}[theorem]{Definition}
\newtheorem{remark}[theorem]{Remark}

\begin{document}

\bibliographystyle{plain}

%
%
\title{Reconstruction of Line-Embeddings of Graphons}
%
%
%
%
%
%
%
%
%
%
%
%
\author{Jeannette Janssen$^{\flat}$}
\thanks{$^{\flat}$ janssen@mathstat.dal.ca,
Department of Mathematics and Statistics
Dalhousie University, Chase Building, Halifax
NS B3H 4J1, Canada}

\author{Aaron Smith$^{\sharp}$}
\thanks{$^{\sharp}$smith.aaron.matthew@gmail.com,
   Department of Mathematics and Statistics,
University of Ottawa, 585 King Edward Avenue, Ottawa
ON K1N 7N5, Canada}
%
%
%
%
%
\begin{abstract}
Consider a random graph process with $n$ vertices corresponding to points $v_{i} \stackrel{i.i.d.}{\sim} \mathrm{Unif}[0,1]$ embedded randomly in the interval, and where edges are inserted between $v_{i}, v_{j}$ independently with probability given by the \textit{graphon} $w(v_{i},v_{j}) \in [0,1]$.  Following \cite{chuangpishit2015linear}, we call a graphon $w$ \textit{diagonally increasing} if, for each $x$, $w(x,y)$ decreases as $y$ moves away from $x$. We call a permutation $\sigma \in S_{n}$ an \textit{ordering} of these vertices if $v_{\sigma(i)} < v_{\sigma(j)}$ for all $i < j$, and ask: how can we accurately estimate $\sigma$ from an observed graph? {We present a randomized algorithm with output $\hat{\sigma}$ that, for a large class of graphons, achieves error $\max_{1 \leq i \leq n} | \sigma(i) - \hat{\sigma}(i)| = \tilde{O}(\sqrt{n})$ with high probability; we also show that this is the best-possible convergence rate for a large class of algorithms and proof strategies. Under an additional assumption that is satisfied by some popular graphon models, we break this ``barrier" at $\sqrt{n}$ and obtain the vastly better rate $\tilde{O}(n^{\epsilon})$ for any $\epsilon > 0$. These improved seriation bounds can be combined with previous work to give more efficient and accurate algorithms for related tasks, including \textit{estimating diagonally increasing graphons}  \cite{gao2015rate,ghandehari2020sequences} and \textit{testing whether a graphon is diagonally increasing}  \cite{chuangpishit2015linear}.}
\end{abstract}
\maketitle

\section{Introduction}

In this paper, we propose and analyze new algorithms for estimating \textit{latent vertex labellings} given an \textit{observed random graph.} We first discuss a motivating simpler problem, often called the \textit{seriation problem} (this dates back at least to the 1899 paper \cite{petries99prehistoric}; see also \textit{e.g.} \cite{Liiv2010}, \cite{laurent2016similarityfirst}). The basic seriation problem considers a \emph{Robinsonian similarity matrix}, i.e.~a symmetric $n$ by $n$ matrix $A = [a_{i,j}]_{1 \leq i,j \leq n}$ with the  property: there exists a permutation $\sigma \in S_{n}$ so that every row of the permuted matrix $A_{\sigma} = [a_{\sigma(i), \sigma(j)}]_{1 \leq i,j \leq n}$  is unimodal with the maximum occurring at the diagonal. It then asks: \textbf{how can we find such a permutation $\sigma$?}

This problem turns out to have an elegant and computationally-tractable solution. For any embedding $\phi \, : \, \{1,2,\ldots,n\} \mapsto [0,1]$ of the indices of $A$ into the line segment $[0,1]$, we define the \textit{induced} permutation $\sigma \in S_{n}$ by the equation
\be \label{EqInducedPerm}
\sigma(i) = \sigma(i,\phi) \equiv | \{ j \in [1:n] \, : \, \phi(j) \leq \phi(i) \} |,
\ee
breaking ties arbitrarily.  It turns out that, under mild conditions, the correct permutation $\sigma$ is of the form
\be
\sigma(i) = \sigma(i,\hat{\phi}),
\ee
where $\hat{\phi}$ is an eigenvector of a matrix related to $A$ \cite{atkins1998spectral}.

This seriation problem occurs in a number of contexts, where we have some collection of objects $\{1,2,\ldots,n\}$ which we would like to order, (\textit{e.g.} ordering types of artifacts by their ages), and some measure $A = [a_{ij}]$ as to whether a pair of objects are similar (\textit{e.g.} their co-occurrence in tombs). A special case is graph seriation, where the aim is to determine whether a permutation of the vertices exists so that the corresponding adjacency matrix has the Robinsonian property, and to find such a permutation.

Of course, real data is noisy. Even if the \textit{expected value} of the similarity matrix is Robinsonian, the \textit{observed} similarity matrix may not be and a ``perfect" ordering may not exist. Our aim is to solve a very generic ``noisy" version of the graph seriation problem. 
Following  \cite{chuangpishit2015linear}, we introduce noise by studying a natural generalization of the above example based on some models from Bayesian nonparametrics. Recall that a \textit{graphon} is just a symmetric measurable function $w \, : \, [0,1]^{2} \mapsto [0,1]$ (these were introduced in \cite{lovasz2006limits}; see \cite{lovasz2012book} for a broader survey).
Under suitable measurability conditions, a graphon defines an algorithm for sampling a random graph of any size $n$:
\begin{enumerate}
\item Sample i.i.d. sequences $\{U_{i}\}_{i=1}^{n}$, $\{U_{ij}\}_{1 \leq i < j \leq n} \sim \mathrm{Unif}([0,1])$.
\item Define the edge set by
\be \label{EqGraphonDef}
\{(i,j) \in E \} \Longleftrightarrow \{ U_{i,j} < w(U_{i},U_{j})\}.
\ee
\end{enumerate}
We write $G \sim w$ if $G$ is a random graph obtained from graphon $w$ in this way. In \cite{chuangpishit2015linear}, the authors say that a graphon $w$ \textit{is diagonally increasing} if it satisfies
\be \label{IneqEmbBasic}
w(x,y) &\leq w(z,y) \\
w(x,z) &\geq w(x,y)
\ee
for all $x<z<y$. Note that this implies that $w(x,\cdot )$ is unimodal with the maximum occurring at $x$, so this definition matches that of the Robinsonian property for  matrices. 

We think of the following simple and well-studied family of graphons as our prototypical examples, and will use this family as a running example in the text:
\be \label{EqSimpleGraphon}
w(x,y) = \begin{cases}
p \qquad |x-y| < d \\
q \qquad |x-y| \geq d.
\end{cases}
\ee
These graphons satisfy \eqref{IneqEmbBasic} as long as $0 \leq q \leq p \leq 1$ (though reconstructing an ordering is clearly impossible if $p=q$ or $d \in \{0,1\}$). Very similar random graphs have been extensively studied (see \textit{e.g.} \cite{RePEc:inm:oropre:v:68:y:2020:i:1:p:53-70, 7924316, DingSmallWorld,rocha2018recovering}).

Special cases of this noisy seriation problem (with points sometimes embedded on a circle instead of an interval) appear in many areas. In some instances, an approximate ordering with an error rate similar to that presented in this paper has been obtained. However, generally such results only apply to a narrow class of random graphs, and the restriction of the problem to this specific class allows for the use of highly specialized methods. A few examples beyond archaeology include genomics \cite{RePEc:inm:oropre:v:68:y:2020:i:1:p:53-70, Putnam_2016}, ranking \cite{fogel2016spectralranking}, data visualization \cite{doi:10.1057/palgrave.ivs.9500042}, and ecology and sociology \cite{Liiv2010}.

Graphs sampled from a diagonally increasing graphon have a natural line-embedded structure, defined by the latent variables $(U_{1},\ldots,U_{n})$. Namely, vertex $i$ will have a higher probability of linking to vertex $j$ if $U_j$ is closer to $U_i$. Consequently, the permutation that reveals the ``almost" Robinsonian property of the adjacency matrix will be the permutation $\sigma(\cdot, (U_{1},\ldots,U_{n}))$ induced by the values $U_i$. 
We call this permutation the \emph{line-embedded permutation.}
This raises the main question of our paper: \textbf{how accurately can we recover the line-embedded permutation  based on an observed graph $G$?}

\begin{remark} \label{RemReverse}
The embedding $\phi=(U_{1},\ldots,U_{n})$ and its ``reverse" $1-\phi=(1-U_{1},\ldots,1-U_{n})$ are both equally valid line-embeddings for a sampled graph $G$, and it is impossible to distinguish between the line-embedded permutation $\sigma(\cdot,\phi)$ and its reverse $\sigma^{rev} =  \sigma(\cdot,1-\phi)$ based on an observed graph. 
In informal discussion we will often implicitly mod out by this reverse operation and talk about ``the unique" line-embedded permutation without further comment.
\end{remark}

A natural idea is to apply a spectral method as used in \cite{atkins1998spectral}  for the noiseless matrix seriation problem. A related approach was studied in the recent work \cite{rocha2018recovering} in the special case that $w$ is of the form \eqref{EqSimpleGraphon} 
and $d=0.5$, $q=0$. Theorem 3 of this paper says that, with high probability, there exists a set $I \subset \{1,2,\ldots,n\}$ of size $|I| = n(1-o(1))$ so that
\be \label{IneqRochaSum}
\max_{i \in I} | \hat{\sigma}(i) - \sigma(i, (U_{1},\ldots,U_{n})) | = \tilde{O}(n^{0.5}).
\ee
 However, the authors indicate that their methods cannot be easily extended to other graphons or, indeed, other parameters of \eqref{EqSimpleGraphon}. It was also not clear if this bound could be improved by other methods.

Our paper introduces and analyzes two new algorithms for the noisy graph seriation problem, which approximate the line-embedded permutation for graphs sampled from a large general class of diagonally increasing graphons. The paper has two main results, which bound the error in the approximate permutation returned by our algorithms. In Theorem \ref{ThmReconMainRes1}, we extend the result from  \eqref{IneqRochaSum} to a very general class of graphons, and obtain an error term of
\be 
\max_{1 \leq i \leq n} | \hat{\sigma}(i) - \sigma(i, (U_{1},\ldots,U_{n})) | = \tilde{O}(n^{0.5}).
\ee

In Theorem \ref{ThmReconMainRes2}, we greatly improve this bound for a slightly smaller class of graphons, showing
\be \label{IneqOtherThm}
\max_{1 \leq i \leq n} | \hat{\sigma}(i) - \sigma(i, (U_{1},\ldots,U_{n})) | = \tilde{O}(n^{\epsilon})
\ee
for any fixed $\epsilon > 0$. This smaller class still contains the most popular statistical models of graphons, including those of the form \eqref{EqSimpleGraphon} in the special case $q=0$.

\begin{remark} [Errors in Positions and Orderings]

We pause to explain why this result may be surprising. In practice, many latent-position models are analyzed by algorithms that \textit{first} estimate all latent positions and \textit{then} plug this estimator into a formula for a quantity of interest (in seriation this quantity of interest is the ordering, but see also \cite{little2018analysis} for applications of the same approach to other problems).  This approach seems sensible under the condition that the first step is not much less accurate than the second step.

However, this condition turns out to \textit{fail} in the context of seriation. In a fairly strong sense that we make precise in Section \ref{SecLargeError}, it is not possible to reconstruct the latent positions $U_{1},\ldots,U_{n}$ themselves with an error that is similar to the error on the positions given in \eqref{IneqOtherThm}. We believe that such a discrepancy  between the achievable error in reconstructing \textit{ordering} and the achievable error in reconstructing \textit{latent positions} is important for practical algorithm development in the area, as it suggests that beginning with an embedding step could result in a statistically inefficient algorithm.

See  Section \ref{SecLargeError} for further discussion and a more precise version of this heuristic.
\end{remark}

\subsection{High-Level Algorithm Descriptions}

We sketch our main algorithms and give heuristics for why they work. Our first algorithm (Algorithm \ref{AlgMerge}) proceeds as follows:

\begin{enumerate} 
\item We begin by computing a new graph  $\GA{\alpha}$ from the observed graph $G$, as follows: 
\begin{enumerate} 
\item We square the adjacency matrix; its entries represent the number of common neighbours of pairs of vertices. Since the number of common neighbours are the sum of many independent Bernouilli trials, the entries of this squared matrix concentrate around their expected values.
\item We convert the square matrix into a binary matrix, and thus a graph, by thresholding at some level roughly $\alpha n$. We show the thresholded matrix is ``almost" Robinsonian, in the sense that any violations occur in the very narrow bad region where the expected value of the squared matrix is very close to the threshold value $\alpha n$.
\end{enumerate}
\item We then take many small random subgraphs of $\GA{\alpha}$ and attempt to order them using a deterministic algorithm. Since $\GA{\alpha}$ is ``almost" Robinsonian, most of these subgraphs will be exactly Robinsonian. 
\item We then align our orderings of these small subgraphs. (Recall from Remark \ref{RemReverse} that individual graphs can't distinguish between $\sigma_{true}$ and $\sigma_{eurt}$, so our subgraph orderings will generally not all be aligned to the same one.) This turns out to be straightforward as long as there are enough small subgraphs to guarantee substantial overlap.
\item Finally, we merge all of our subgraph orderings into a large ordering, essentially by a voting procedure. All of the steps up to this point preserve quite a bit of randomness, and so we again obtain concentration bounds for the result. 

\end{enumerate} 

Sections  \ref{SecMainEstHeuristic} and \ref{Sec:ProofSketch} give a more detailed sketch of our analysis of  Algorithm \ref{AlgMerge}, giving precise error bounds in place of informal phrases such as ``almost." We believe it is possible for the reader to have a nearly-complete understanding of the algorithm by reading these two short sections; the remainder of Section \ref{SecAlgAnalysis} is concerned with filling in unsurprising details or ruling out failure modes that ``obviously" should not occur.

The second algorithm (Algorithm \ref{AlgIterationHolder}), which achieves much smaller error bounds, proceeds in a sequence of stages. 
\begin{enumerate}
\item The process  starts with a coarse ordering achieved using Algorithm \ref{AlgMerge} on a small random subgraph. 
\item In each subsequent stage, we use the approximate ordering on a small graph to obtain an ordering on a slightly larger graph with a slightly smaller error. These iterative steps use the graph itself (not its square). The additional condition of Theorem \ref{ThmReconMainRes2} guarantees that for every pair of vertices $i,j$, there is an interval $I_{i,j}\subset [0,1]$ so that any vertex $k$ with $U_k\in I_{i,j}$ may link to $i$, but cannot link to $j$. Vertices of the previous iteration with $U_k$-values estimated to be in this interval provide a signal of the true ordering of $i$ and $j$.
\item In the final stage we return the permutation based on a final ordering on the full graph.
\end{enumerate}

\subsection{Basic Notation}
 We say that a sequence of events $\mathcal{A} = \{\mathcal{A}^{(n)}\}$ indexed by $n \in \mathbb{N}$ hold \textit{with extreme probability} if
\be
\P[\mathcal{A}^{(n)}] \geq 1 - n^{-c_{1} \log(n)^{c_{2}}}
\ee
for some $c_{1},c_{2} > 0$ and all $n > N_{0}$ sufficiently large. We use the shorthand \textit{$\mathcal{A}$ holds w.e.p.}, and note that for any fixed $C < \infty$ the intersection of $n^{C}$ events that occur w.e.p. also occurs w.e.p.  This property motivates our definition. In particular, we will often show that a collection of events $\mathcal{A}_{i,j}$ hold w.e.p. for all $1 \leq i,j \leq n$, then cite this to conclude that $\cap_{i,j} \mathcal{A}_{i,j}$ occurs w.e.p.

We will measure the error of a given ordering by measuring the maximum size of a `local' neighbourhood around a vertex where errors can occur:

\begin{defn} [Ordering Error] \label{DefErrorOrder}
Let $\{ U_i\}_{i=1}^{n}$ be a collection of distinct points in $[0,1]$. 
Let $\sigma_{\mathrm{true}}$ be the  permutation induced by $(U_1,\dots ,U_n)$ and
let  $\eurt $ be its reverse. 
Say that a permutation $\sigma\in S_{n}$ is \textit{correct} for this latent collection of points if $\sigma =\sigma_{\mathrm{true}}$ or $\sigma =\eurt$, i.e.~it satisfies either
\be
\{ \sigma(i)<\sigma(j)\} \leftrightarrow \{U_i<U_j\}
\ee
for all $i,j$, or satisfies
\be
\{ \sigma(i)>\sigma(j)\} \leftrightarrow \{U_i<U_j\}
\ee
for all $i,j$. Say that a permutation $\sigma$ has error less than $\err$ if there exists a correct ordering $\sigma_{\mathrm{correct}}$ so that, for all $1 \leq i,j \leq n$ satisfying $\sigma_{\mathrm{correct}}(i) > \sigma_{\mathrm{correct}}(j) + \err$, we also have $\sigma(i) > \sigma(j)$.
\end{defn}

\subsection{A Simple Assumption}

In this section, we give a simple-to-state sufficient condition for our first main result, Theorem \ref{ThmReconMainRes1}. This sufficient condition is satisfied for many graphons (including those of the form \eqref{EqSimpleGraphon}), but is far from tight; we relax these conditions in the statement of Theorem \ref{ThmMaxOrd} in Section \ref{SectionSlightlyStronger}.


We begin by restricting our attention to the simple class of \textit{uniformly embedded} graphons:

\begin{defn} [Uniformly embedded graphons]\label{DefBoundedIncrease}
A graphon $w$ is {\sl uniformly embedded} if there exists a function $f[0,1]\rightarrow [0,1]$ so that, for all $x,y$,
\[
w(x,y)=f(|x-y|).
\]
In this context, $f$ is called the {\sl link probability function} of $w$.
\end{defn}

We then make the following technical assumption on the link function:

\begin{assumption} \label{AssumptionsSimpleWeakIdentifiabilityAssumptions}

Let $w$ be a uniformly embedded graphon with link probability function $f:[0,1]\rightarrow [0,1]$. In addition, $f$ is decreasing, and there exist constants $0<d <0.5$ and $0 \leq c<f(0)$ so that $f(z)=c$ for all $z\geq d$. Finally, there exists $\alpha \in (0,1)$ so that $f$ satisfies
\be\label{eq:alpha}
\inf_{s\in [0,d]}\int_0^1 f(z)f(|s-z|) dz  > \alpha > \int_0^1 f\left( \left|z-\frac{1-d'}{2}\right|\right) f\left(\left|z-\frac{1+d'}{2}\right|\right) dz,
\ee
where $d'=\min\{ 0.5,2d\}$.
\end{assumption}

\begin{remark}
Denote by $w^{(2)}$ the usual ``square" of a graphon (see Definition \ref{DefGraphPower}). We note that the left hand side of Equation \eqref{eq:alpha} equals $\inf_{s\in [0,d]}w^{(2)}(0,s)$, and the right hand side equals $w^{(2)}(\frac{1-d'}2,\frac{1+d'}{2})$. Also, since $w(x,y)\geq c$ for all $x,y$, the lower bound on $\alpha $ implies that $\alpha >c^2$.

Condition \eqref{eq:alpha} is needed to guarantee that 
the thresholded squared graph used in our algorithm is close to diagonally increasing.  It is not generally true that the square of a diagonally increasing graphon is diagonally increasing. See Figure \ref{fig:example} for an illustration of a simple counterexample of type \eqref{EqSimpleGraphon}. However, even if the squared graphon itself is not diagonally increasing, we can often assure that the thresholded squared graphon is diagonally increasing for some well-chosen $\alpha$. See the appendix for details on the example in Figure \ref{fig:example}, including the choice of $\alpha$. 

We note that this obstacle is real (not merely an artifact of our proof technique), and other approaches to seriation have similar obstacles. For example, spectral seriation requires the Fiedler eigenvector of a matrix associated with $A$ related to the graph to be monotone. However, the Fiedler eigenfunction associated with a graphon satisfying \eqref{IneqEmbBasic} is \textit{not} always monotone - \textit{some} additional conditions really are necessary.
\end{remark}

\begin{figure}[h]
\label{fig:example}
\begin{center}
\begin{tabular}{cc}
\begin{tikzpicture}[scale=0.8]
\draw[step=1cm,gray!50,very thin] (-1,-1) grid (6,6);
\draw[thick,->] (0,0) -- (5.5,0) node[anchor=north west] {$y$};
\draw[thick,->] (0,0) -- (0,5.5) node[anchor=south east] {$w(0,y)$};
\node at (5,0) [anchor=north ] { 1};
\node at (1.3,0) [anchor=north ] {$d$};
\node at (0,0) [anchor=north east] {0};
\node at (0,1) [anchor=east] {};
\node at (0,3.7) [anchor=east] {};
\node at (0,5) [anchor=east] {};
\draw[thick] (1.3,1) -- (5,1);
\draw[thick] (0,4) -- (1.3,4);
\draw[thick, dashed] (1.3,4) -- (1.3,1);
\end{tikzpicture}
&
\begin{tikzpicture}[scale=0.8]
\draw[step=1cm,gray!50,very thin] (-1,-1) grid (6,6);
\draw[thick,->] (0,0) -- (5.5,0) node[anchor=north west] {$y$};
\draw[thick,->] (0,0) -- (0,5.5) node[anchor=south east] {$w^{(2)}(0,y)$};
\node at (5,0) [anchor=north ] {1};
\node at (1.3,0) [anchor=north ] {$d$};
\node at (2.6,0) [anchor=north ] {$2d$};
\node at (0,0) [anchor=north east] {0};
\node at (0,2.5) [anchor=east] {};
\node at (0,3.7) [anchor=east] {{\color{red} $\alpha$}};
\node at (0,5) [anchor=east] {};
\draw[thick] (0,4.3) -- (1.3,4.8) -- (2.6,3.2) -- (4,3.2) -- (5,1.5);
\draw[very thick, red, dashed] (2.2,3.7) -- (2.2,0);
\draw[very thick, red] (0,3.7)--(2.2,3.7);
\draw[very thick, red] (2.2,0)--(5,0);
\end{tikzpicture}
\end{tabular}
\end{center}
\caption{This shows a sketch of $w(0,\cdot ) $ and $w^{(2)}(0,\cdot)$ for a graphon of type \eqref{EqSimpleGraphon} that satisfies Assumption \ref{AssumptionsSimpleWeakIdentifiabilityAssumptions}. Note that $w^{(2)}$ violates the diagonally increasing condition, but $w^{(2)}_{\alpha}$ (shown in red) does not. }
\end{figure}
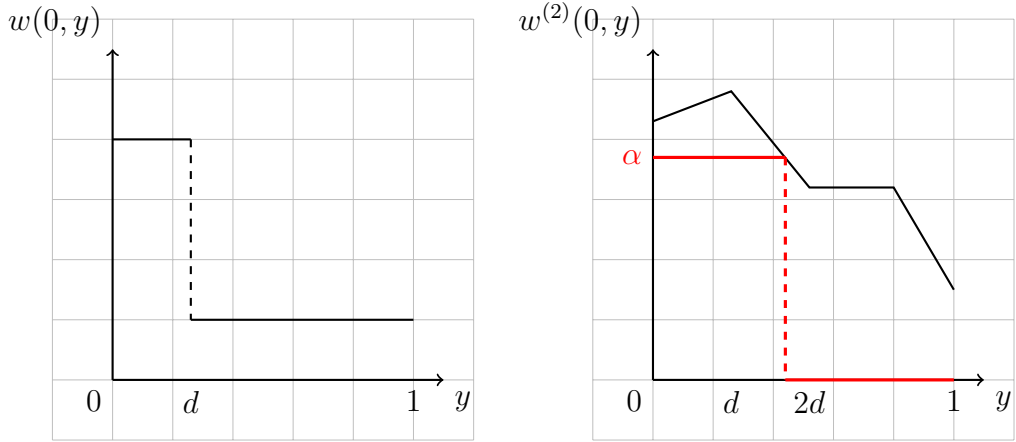

\begin{remark}\label{rem:examples}
Simple examples of classes of uniformly embeddded graphons for which Assumption \ref{AssumptionsSimpleWeakIdentifiabilityAssumptions} holds are:
\begin{itemize}
\item[{\sl (i)}] Any uniformly embedded graphon for which there exists a value $0<d<0.25$ so that $f(x)=0$ for all $x>d$,
\item[{\sl (ii)}] Graphons of type \eqref{EqSimpleGraphon}, where $0.25\leq d<0.5$, and  $d<1-p(1-3q)/(2p^2+2q^2)$.
\item[{\sl(iii)}] Uniformly embedded graphons where $f$ decreases linearly between $0$ and $d$ from $f(0)=p$ to $f(d)=c<f(0)$, and $f(x)=c$ for all $x\geq d$, and $d\in (0.25,0.5)$. 
\item[{\sl (iv)}] Uniformly embedded graphons where $f$ is concave from $0$ to $d<0.5$ and constant from $d$ to $1$, and in addition,
\[
\int_0^1 f(z)^2 dz   > \int_0^1 f(|z-\frac{1}{4}|) f(|z-\frac{3}{4}|) dz.
\]
\end{itemize}

We also mention one simple graphon type for which the more general Assumption \ref{AssumptionGoodGraphon} holds but which is not uniformly embedded:
\begin{itemize}
\item[{\sl (v)}] Graphons of the form
\[
w(x,y)=\left\{ \begin{array}{cl} p&\text{if } \ell (x)\leq y\leq r(x)\\
0 & \text{otherwise,}\\
\end{array}\right.
\]
where $\ell (x)$ and $r(x)$ are differentiable boundary functions that are each other's inverse, have slope bounded away from zero on the appropriate domain, and satisfy 
\begin{align*}
0<\inf_x (r(x) - x) &\leq \sup_x (r(x) - x)<0.25 \\
\end{align*}

\end{itemize}
Note that graphs sampled from graphons of type ({\sl v}) can be interpreted as graphs where points are sampled from $[0,1]$ according to a non-uniform distribution, while links are formed according to a simple graphon of type \ref{EqSimpleGraphon}. If this non-uniform distribution has invertible CDF $F$, then the associated graphon may be written as $w(x,y) = w_{unif}(F(x), F(y))$, where $w_{unif}$ is a graphon of type \ref{EqSimpleGraphon}. 

\end{remark}

We include these simple examples primarily as illustration. We suspect that most real-world data will be much more complicated and will not resemble graphs generated by \textit{any} graphon with a simple formula - that is, most datasets will look quite weird and irregular in some way. We also note that most papers on ``noisy" or ``statistical" seriation have studied small parametric families of random graphs, but we allow a much larger (nonparametric) class of graphons.  See Section \ref{SSSSeriation} for further discussion. 

\subsection{Main Results}

In the next section, we will present our main new algorithm (Algorithm \ref{AlgMerge}). The first of our main results is the following bound on its error:

\begin{thm} [Reconstruction for General Graphons] \label{ThmReconMainRes1}
Fix a graphon $w$ and constant $\alpha$ that satisfy Assumption \ref{AssumptionsSimpleWeakIdentifiabilityAssumptions}.
Let $G_{n} \sim w$ be a graph of size $n \in \mathbb{N}$. Then, when Algorithm \ref{AlgMerge} is executed with parameters as in Equation \eqref{def:GoodParameters} on input  graph $G_{n}$ and value $\alpha$,
the output is a permutation ``$\sigma$" on $\{1,2,\ldots,n\}$ with error $\err$ that satisfies
\be
\err \leq  \sqrt{n} \log(n)^{\cfin}
\ee
w.e.p.
\end{thm}

This matches (up to logarithmic terms) the rate found in \cite{rocha2018recovering} for the special case of graphons of the form \eqref{EqSimpleGraphon}. However, it turns out that a slightly different algorithm can give much better rates for graphons of that form, and indeed any graphons satisfying the additional assumption: 

\begin{assumption} [Sharp Boundaries] \label{AssumptionSharpBoundary}
Say that a graphon $w$ has \textit{sharp boundaries} if:

\begin{enumerate}
\item There exists some $\delta > 0$ so that
\be
w(x,y) \in \{0 \} \cup (\delta,1]
\ee
for all $x,y \in [0,1]$.
\item There exists some $B > 0$ so that for all $x,y \in [0,1]$,
\be
\mathrm{Vol}(\{z \in [0,1] \, : \, w(x,z) = 0 \neq w(y,z) \text{ or } w(y,z) = 0 \neq w(x,z) \}) \geq B |x-y|.
\ee
\end{enumerate}
\end{assumption}

Note that this additional assumption holds for \textit{some} examples of each type given in Remark \ref{rem:examples}. In particular, it holds for type (ii) when $q=0$ and for type (iii) when $c=0$. We don't know if this condition is close to optimal; see the discussion following Algorithm \ref{AlgIterationHolder} for an explanation of where it is used and discussion in Section \ref{SSSecEstGraph} for an explanation of a similar condition that turns out to be required for a closely-related problem.

We have:
\begin{thm} [Reconstruction For Graphons with Sharp Boundaries] \label{ThmReconMainRes2}

Fix a graphon $w$ and constant $\alpha$ that satisfy Assumptions \ref{AssumptionsSimpleWeakIdentifiabilityAssumptions} and \ref{AssumptionSharpBoundary}, and let $G_{n}\sim w$ be a graph of size $n$.  Fix $\epsilon >0$. When  Algorithm \ref{AlgIterationHolder} is executed with parameters given by \eqref{DefParameters} on input $G_n$ and value $\alpha$,
the output is a permutation ``$\sigma$" on $\{1,2,\ldots,n\}$ with error
\be
\err \leq  n^{\epsilon}
\ee
w.e.p.

\end{thm}

Algorithm \ref{AlgIterationHolder} consists of two parts: first a coarse ordering of a small subset of the vertices is obtained using the method from Theorem \ref{ThmReconMainRes1}, then a refinement algorithm (Algorithm \ref{AlgRefine}) is used to obtain an ordering of all vertices with the desired error level. These parts are independent, and the refinement algorithm could be used to improve and extend \emph{any} coarse ordering of a subset of the vertices. 

To be more precise, we will use the definition:

\begin{defn} \label{AssumptionEff}
Fix a function $\alg$ which takes a finite graph as input and gives an ordering of the vertices as output. Fix a diagonally-increasing graphon $w$, and let $G_{1},G_{2},\ldots \sim w$ be a sequence of graphs on $1,2,\ldots$ vertices. Say that the function $\alg$ is \textit{efficient} for  $w$ if $\alg(G_{n})$ has error less than $\sqrt{n} \log(n)^{5}$ w.e.p.
\end{defn}

For shorthand, we write ``$\alg$-Algorithm \ref{AlgIterationHolder}" for Algorithm \ref{AlgIterationHolder}, with step 2 replaced by the assignment:
\be 
\sigma_{1} = \alg(G_{1}).
\ee 

\begin{corollary} \label{CorExperts}
Fix a graphon $w$ that satisfies Assumptions  \ref{AssumptionSharpBoundary}, let $\alg$ be \textit{efficient} for $w$ in the sense of Definition \ref{AssumptionEff}, and let $G_{n}\sim w$ be a graph of size $n$.  Fix $\epsilon >0$. When  $\alg$-Algorithm \ref{AlgIterationHolder} is executed with parameters given by \eqref{DefParameters} on input $G_n$ and value $\alpha$,
the output is a permutation ``$\sigma$" on $\{1,2,\ldots,n\}$ with error
\be
\err \leq  n^{\epsilon}
\ee
w.e.p.

\end{corollary}

We can think of two situations in which Corollary \ref{CorExperts} might be useful:

\begin{enumerate}
\item Naturally, other mathematical work may provide algorithms that are \textit{efficient} in the sense of  Assumption \ref{AssumptionEff} (see \textit{e.g.} \cite{rocha2018recovering,natik2020consistency}). 
\item In some situations we may use \textit{expert knowledge}. For example, one might send a small random sample to experts who can order them by hand (likely based on features not captured by the graph). The large remainder of the data could be ordered automatically as in Algorithm \ref{AlgIterationHolder} after Step 2. It is difficult to model expert knowledge precisely, but we know that in practice it is often very useful to be able to incorporate strong (but expensive) evidence about a (small) subsample. 
\end{enumerate}

\begin{remark} [Optimistic Conjecture on Optimal Reconstruction]
We conjecture that the best possible reconstruction error is $O(\sqrt{n} \log(n)^{C})$ in the situation of Theorem \ref{ThmReconMainRes1} and $O(\log(n)^{C'})$ in  the situation of Theorem \ref{ThmReconMainRes2}. However, the arguments in the present paper seem to have no hope of finding the optimal constants $C, C'$.

It is natural to ask when it is possible to attain a ``super-small" error that is much less than $O(\sqrt{n} \log(n)^{C})$. We conjecture that this can occur in many situations where Assumption \ref{AssumptionSharpBoundary} fails, but also that recovering such a super-small error would require modifications to our main refinement algorithm. We would be interested in knowing if it is possible to come close to covering all situations in which a ``super-small" error occurs by using a single algorithm, but make no conjectures on this question.

\end{remark}

\begin{remark} [Dependence on Parameters]

We note that Theorems \ref{ThmReconMainRes1} and \ref{ThmReconMainRes2} have assumptions with constants $c,d,\delta,B$ that don't appear explicitly in the theorem statements. Recall that both theorems say that a certain inequality occurs \wep - that is, it occurs for all $n$ larger than some random integer $N$. The distribution of this random integer $N$ depends on the constants $c,d,\delta,B$. In general, $N$ may become very large if these constants are close to the edge of their ranges (e.g. if $\delta$ is very close to 0). This dependence is essentially unavoidable. For example, when $c$ is very close to $f(0)$, the graphon is very close to the constant graphon $w_{const}(x,y) \equiv c$. It is obviously not possible to reconstruct the ordering of vertices given observations from a constant graphon, and it is straightforward to check that one requires at least $n \gtrsim (f(0)-c)^{-0.5}$ to reliably distinguish between a graphon $w$ and the constant graphon $w_{const}(x,y) \equiv c$. 
\end{remark} 

\subsection{Running Time}

It is straightforward to check that both algorithms have running times that are polynomial in $n$ with parameters as stated. We have not made a serious effort to optimize the parameters for running time, and in several places we make choices in order to make our (already-long) proofs slightly shorter. With that in mind:

\begin{enumerate}

\item  Algorithm \ref{AlgMerge} has running time that is $O(n^{4} \mathrm{polylog}(n))$, dominated by  Algorithm \ref{AlgGlobal}. See Remark \ref{RemSpeedingUpSlow} for a quick discussion on how this may be improved to $O(n^{3} \mathrm{polylog}(n))$.

\item For the parameters appearing in Theorem \ref{ThmReconMainRes2}, Algorithm \ref{AlgIterationHolder} has running time that is $O(n^{2} \mathrm{polylog}(n))$. Furthermore, the running time is dominated by steps 3--6; these can be substantially parallelized. More generally, this $O(n^{2} \mathrm{polylog}(n))$ running time holds as long as the parameter $p_{1}$ appearing in Algorithm \ref{AlgIterationHolder} is less than $n^{-0.5}$.


\end{enumerate}

\subsection{Main Contributions and Related Work}

We highlight what we view as the main contributions of this paper, in the context of the enormous literature on latent position models:

\begin{enumerate}
\item Many of the most-similar previous papers studied parametric families of graphons, often with few parameters (see \textit{e.g.} \cite{7924316},\cite{ rocha2018recovering}). Often, graphs drawn from the family can be seen as slightly perturbed versions of graphs with a highly symmetric structure. The current paper studies a large nonparametric family, and importantly the algorithm does not require the user to identify a specific subfamily. This is important from a practical point of view, since simple-to-describe parametric families such as \eqref{EqSimpleGraphon} are typically not realistic models and it is often difficult to identify a reasonable model.

One consequence of this generality is that the conditions for our main result are longer and more complicated. This may cause our conditions to appear more technical or restrictive than previous work, even though the opposite is true. This is largely a consequence of the fact that it  takes more space to write down a collection of inequalities for generic graphons than it does to write down a collection of inequalities for a small number of real-valued parameters.

\item Many previous works on latent position models proceed by first estimating latent positions, then plugging this estimate into the definition of a discrete object of interest (a permutation in the case of seriation, a community structure in the case of \textit{e.g.} spectral clustering \cite{little2018analysis}). It is clear that estimating the full embedding is not any easier than estimating the discrete structure, since the discrete structure is a deterministic function of the true embedding. In this paper, we show that there is in fact a large gap in the difficulty of these two problems -- estimating the latent positions is sometimes \textit{much harder} than estimating the discrete structure. We hope that this observation is useful for the design of future algorithms.

\end{enumerate}

In the rest of the section, we give a more detailed look at some specific parts of the literature on latent position models.

\subsubsection{Consequences for Efficiency of Downstream Tasks} \label{RemCons}
We note that the problem considered in this paper has the following form: there is some ``true" latent position $U_{1},\ldots,U_{n}$ of each vertex, and we wish to estimate some explicit and reasonably nice function of these latent positions (in this case, the induced permutation as in Equation \eqref{EqInducedPerm}). It is natural and very common to try to do this sort of task by the following two-step procedure:  \textit{first} finding an estimate  $\hat{U}_{1},\ldots,\hat{U}_{n}$ of the true latent positions $U_{1},\ldots,U_{n}$, and \textit{then} plug this estimate into the function (typical algorithms for the closely-related problem of spectral clustering are of this form; see \textit{e.g.} the famous overview \cite{von2007tutorial}).

The algorithms studied in this paper are \textit{not} of this two-step form, and this is not an accident. In Section \ref{SecLargeError}, we prove that the best estimate $\hat{U}_{1},\ldots,\hat{U}_{n}$ of $U_{1},\ldots,U_{n}$ will have error $\Omega(n^{-0.5} \log(n)^{-C''})$ for a broad class of models. This implies that \textit{there is no way to recover the error bound in Theorem \ref{ThmReconMainRes2} by estimating the latent positions $\hat{U}_{1},\ldots,\hat{U}_{n}$ and then propagating this error bound to an estimate $\hat{\sigma} = \hat{\sigma}(\hat{U}_{1},\ldots,\hat{U}_{n})$ of $\sigma_{\mathrm{true}}$ that depends only on these estimated latent positions}. This fact was somewhat surprising to us, as it is common to analyze the performance of algorithms on ``downstream tasks" such as clustering or ordering using this sort of two-step error bound  (see \textit{e.g.} the recent paper  \cite{little2018analysis}, which include such a  ``downstream analysis" for a closely-related embedding problem). Our general observation that two-step procedures may not be optimal (even where they are common in practice) is not new - see \textit{e.g.} \cite{moitra2010settling} and the references therein.

\subsubsection{Seriation} \label{SSSSeriation}

A graph for which the adjacency matrix is a Robinsonian similarity matrix is known as a \emph{unit interval graph}.
The graph seriation problem for unit interval graphs can be solved efficiently using a graph-theoretical algorithm.  In particular, in \cite{corneil2004}, the author gives a linear time algorithm that returns the line-embedded permutation. The algorithm uses three sweeps of \emph{LexBFS}, a special Breadth-First Search algorithm first proposed in \cite{Rose1976LexBFS}. Corneil's 3-sweep LexBFS algorithm is used as an important subroutine in our algorithm.
The matrix seriation problem can also be solved efficiently. either with a combinatorial algorithm (see \cite{laurent2016similarityfirst} for an $O(n\log n)$ algorithm) or with spectral methods, as proposed in \cite{atkins1998spectral}.

If the matrix is not exactly Robinsonian, the seriation problem becomes intractable. Previous work on seriation-with-error considers an  
 \emph{optimization} problem: find a Robinson matrix and a permutation so that the $\ell_p$ distance between the permuted matrix and the Robinson approximation is minimized. This problem is NP-hard in general (see \cite{barthelemy2001overlappingclustering} for the case  $p<\infty$ and  \cite{chepoi2009seriationhardness} for the case  $p=\infty$). In \cite{chepoi2011seriationapproximation} the authors propose a (deterministic) approximation algorithm for the $\ell_\infty$ case. If the matrix exhibits certain additional nice properties, then the seriation problem can be solved in polynomial time as a quadratic assignment problem, see  \cite{laurent2015ToeplitzRobinsonian,cela2015QAProbinson}. In \cite{fogel2013convex,lim2014convex,evangelopoulos2017non-convex}, noisy seriation is presented as a convex optimization problem over a polytope related to the set of permutations. An approximate solution is reached using a relaxation of this optimization problem. Finally, \cite{vuokko2010spectralordering,fortin2015signchanges} study related problems in the special case of $\{0,1\}$-valued matrices.

The article \cite{flammarion2019optimal}, like our work, considers a \textit{statistical} problem where we must recover the ordering associated with some generating process.  In  \cite{flammarion2019optimal}, the given matrix is assumed to be a permuted version of a Robinson matrix, plus an additive error matrix of Gaussian noise. This approach falls in the general categories of shape constrained regression and latent permutation learning.  In  \cite{fogel2016spectralranking}, the authors use spectral seriation to reconstruct a ranking from a set of pairwise comparisons.

The articles \cite{RePEc:inm:oropre:v:68:y:2020:i:1:p:53-70, 7924316, DingSmallWorld,rocha2018recovering} are most similar to the present work. All of them study essentially the same problem as the current paper - reconstructing an ordering based on an observed graph. The major difference is that they all consider graphons that lie in small-dimensional parametric families, all quite similar to our example \eqref{EqSimpleGraphon}. \footnote{The details are somewhat different. Our points are embedded uniformly, while some of theirs are placed at positions $U_{i} = \frac{i}{n}$; our graphs have $\Theta(n^{2})$ edges, while theirs might have fewer; etc. We suspect that these minor differences could be accounted for using the approach set forth in this paper.} In most cases, it is not clear how their algorithms or analyses could be extended to the more general setting considered in this paper. This is important in practice, since we do not expect many real-life datasets to be generated by any graphon with a simple formula such as \eqref{EqSimpleGraphon} - real graphs are often quite messy!

The article  \cite{natik2020consistency} appeared after this paper and studies essentially the same problem. Like this paper, it also considers a nonparametric family of graphons. Its main result is that the popular spectral clustering algorithm gives a consistent estimate of the ordering under certain conditions. The main difference between our papers is that the approach in \cite{natik2020consistency} works for a different class of graphons and achieves a slower rate of convergence.


\subsubsection{Testing Graphons} \label{SecTestingGraphons}

Our initial motivation for this problem was finding efficient tests as to whether a random graph was simulated from a graphon that could be embedded (either in a one-dimensional space as in here, or more generally). Naive approaches fall into the trap discussed in \ref{RemCons}. 

In \cite{chuangpishit2015linear} and \cite{ghandehari2020sequences}, the authors construct a test  of line-embeddability. Precisely, they define a function $\Gamma^{\ast}$ on the space of graphs so that, for a sequence of random graphs $G_{n} \sim w$ sampled from a graphon $w$, we have $\lim_{n \rightarrow \infty} \Gamma^{\ast}(G_{n})=0$ if and only if $w$ is a.e.~diagonally increasing. Furthermore, $\Gamma^{\ast}$ is obtained from a function $\Gamma$ that is a \emph{testable parameter} in the sense of \cite{lovasz2012book}, and thus it can be efficiently estimated from small samples. Unfortunately, a problem remained: \textit{computing} $\Gamma^{\ast}$ exactly requires the computation of a vertex ordering that is ``optimal," in some sense made precise in \cite{chuangpishit2015linear}. The algorithms from this paper provide an approximate solution to a very similar seriation problem, which seems to be useful in providing a rigorous and efficient method for estimating $\Gamma^*$; the details will be contained in future work.

\subsubsection{Estimating Graphons} \label{SSSecEstGraph}

Consider a graph sampled from a graphon $w$ as in Equation \eqref{EqGraphonDef} with driving randomness $\{U_{i}\}$, $\{U_{i,j}\}$. In the present paper, we focus on estimating the \textit{ordering} $\sigma$ for which $U_{\sigma(1)} < U_{\sigma(2)} < \ldots < U_{\sigma(n)}$. There is a larger literature on estimating the \textit{graphon} - the full $n$ by $n$ matrix $\{ w(U_{i},U_{j})\}_{i,j=1}^{n}$. The papers \cite{gao2015rate,chatterjee2015matrix,mao2018breaking,mao2018towards,ghandehari2020sequences} all define and analyze estimators that can be used to estimate graphons (some, especially \cite{chatterjee2015matrix}, also apply to much larger classes of random matrices). We summarize some main points on the relationship between these results and ours.

It is clear that good estimates of a graphon give \textit{some} information about its  ordering - if you had a perfect estimate of \textit{e.g.} a graphon of the form \eqref{EqSimpleGraphon}, it would be trivial to recover a perfect ordering. However, this relationship is not very robust, and breaks down quickly in the presence of even small error. Due to either the weak norm \cite{gao2015rate,chatterjee2015matrix,mao2018breaking,mao2018towards} or the slow convergence rate \cite{ghandehari2020sequences}, 
existing graphon estimates can't be used to obtain good order estimates, even with unlimited computational resources. 

Although our results are not equivalent, we suspect that our estimates of the \textit{ordering} may lead to improved \textit{graphon} estimates in some situations. The two papers with the best bounds on the convergence rates of their estimators, \cite{gao2015rate,ghandehari2020sequences}, both use estimators that seem to be computationally intractable. Furthermore, both are intractable for essentially the same reason: they require the computation of an ordering of the vertices that is ``optimal" in a sense that is very similar to our problem.\footnote{To be more precise: \cite{gao2015rate} asks that the vertices be partitioned into roughly $\sqrt{n}$ sets, and that these sets (rather than individual vertices) be optimally ordered. Ordering partitions is easier than ordering vertices, but in this case it doesn't seem to be \textit{significantly} easier. } Our improved bounds on the ordering problem may lead to computationally-tractable estimators that come close to matching the efficiency of \cite{gao2015rate,ghandehari2020sequences}.

Although there is no general equivalence between estimating graphons and orderings, it seems likely to us that \textit{sufficiently good} graphon estimates could be used to obtain \textit{consistent} order estimates, even when this ``translation" is not statistically efficient. Even inefficient estimates could be used as initial steps in Algorithm \ref{AlgIterationHolder}, as discussed in and immediately before Corollary \ref{CorExperts}.

As one final remark, these papers indicate that the  rate at which graphon estimation procedures can converge depends on the ``smoothness" of the underlying graphon. This suggests that Assumption \ref{AssumptionSharpBoundary} may not be entirely an artifact of our proof technique.

\subsection{Paper Guide}

In Section \ref{SectionAlgorithms}, we state the main algorithms studied in this paper. In Section \ref{SecAlgAnalysis}, we prove a version of Theorem \ref{ThmReconMainRes1} with weaker conditions; we defer a proof that Theorem \ref{ThmReconMainRes1} actually follows from this result until Appendix \ref{SubsecWeakStrongAssumptionsCheck}. In Section \ref{SecErrorRooting} we discuss why it is possible to beat the $\tilde{O}(\sqrt{n})$ reconstruction barrier and prove Theorem \ref{ThmReconMainRes2}.

In the remainder of the paper, we discuss some auxillary results that are helpful in applying and understanding our main results. In Section \ref{SecSeekingAlpha}, we give simple guidelines for finding a ``good" thresholding value $\alpha$.  In Section \ref{SecLargeError}, we show that it is generally not possible to beat the $\tilde{O}(\sqrt{n})$ reconstruction barrier by using an ``induced" permutation of the form \eqref{EqInducedPerm}.

\section{Algorithms for General Graphons}\label{SectionAlgorithms}

In this section, we state the main algorithms studied in this paper, interspersed with comments on what is ``typically" happening at various important stages.

\subsection{Some Further Notation}

Much of our work will be related to taking powers of graphs and graphons, then thresholding the result:

\begin{defn}[Powers of Graphs and Graphons] \label{DefGraphPower}
For a graph $G$, denote by $A = A(G)$ its adjacency matrix. We will use $A^{(2)}$ to denote the square of the adjacency matrix. Entry $A^{(2)}[i,j]$ denotes the number of common neighbours of vertices $i$ and $j$, i.e.~the number of vertices adjacent to both $i$ and $j$.

Similarly, define the ``product" $w_{1} \star w_{2}$ of two graphons $w_{1},w_{2}$ by
\be
(w_{1} \star w_{2})(x,y) = \int_{u} w_{1}(x,u) w_{2}(u,y) du
\ee
and define $w^{(2)} = w \star w$.
\end{defn}

For any graphon $w$ and $\alpha\in [0,1]$, define the thresholded graphon $w_{\alpha}$ by
\be
w_{\alpha}(u,v) &= \textbf{1}_{w(u,v) \geq \alpha}.
\ee

For parameter $0 < \alpha <1$ and graph $G = (V,E)$, define the ``threshold-square" graph $\GA{\alpha} = (V,E_{\alpha})$ by
\be \label{EqThreshDef}
\{(u,v) \in E_{\alpha} \} \longleftrightarrow \{ A^{(2)}[u,v] > \alpha \, (n-2)\}.
\ee

Similarly, define the thresholded graphon $\WA{\alpha}$ by
\be
\WA{\alpha}(u,v) &= \textbf{1}_{w^{(2)}(u,v) > \alpha}.
\ee

Although our goal is to estimate a permutation $\sigma$, most of our calculations will be for pairwise comparison functions $F \, : \, V^{2} \mapsto \{-1,0,1\}$. The idea is that a comparison $F$ corresponds closely to a permutation $\gamma$ if $F$ is close to the following function $F_{\gamma}$:

\be \label{EqDefOrderingMap}
F_\gamma(u,v) = \left\{ \begin{array}{ll}1 &\text{if } \gamma(u) < \gamma(v)\\
-1 &\text{if }\gamma(u) > \gamma(v).\end{array}\right.
\ee
To go from a comparison function to a permutation, define:
\be
\gamma_{F}(i) &= \sum_{j \in V} F_{\gamma}(i,j) \\
\sigma_{F}(i) &= |\{ j \in V \, : \, \gamma_{F}(j) \leq \gamma_{F}(i) \}|
\ee
when the values of $\gamma_{F}$ are distinct; when there are ties, break them arbitrarily and use the above formula. We note that, for any permutation $\sigma$, we have $\sigma_{F_{\sigma}} = \sigma$.

We can use these formulae to move between permutations and comparison functions, and we extend notation in the obvious way. To make one frequently-used example explicit: we say that a comparison function $F \, : \, V^{2} \mapsto \{-1,0,1\}$ has error less than $\err$ if there exists a correct permutation $\sigma_{\mathrm{true}}$ so that, for all $1 \leq i,j \leq n$ satisfying $|\sigma_{\mathrm{true}}(i) - \sigma_{\mathrm{true}}(j)|> \err$,
\be
F(u,v) = 1 - 2 \textbf{1}_{\sigma_{\mathrm{true}}(i) < \sigma_{\mathrm{true}}(j)}.
\ee

We will also often define permutations on some subset $S \subset V$. We say that a permutation $\sigma$ on $S$ agrees with a permutation $\eta$ on $V$ if $F_{\sigma}(i,j) = F_{\eta}(i,j)$ for all $i,j \in S$. Abusing notation slightly, we will also say that $\sigma$ agrees with $\eta$ on some $D \subset S^{2}$ if $F_{\sigma}(i,j) = F_{\eta}(i,j)$ for all $(i,j) \in D$ (typically, $D$ will be pairs that are sufficiently far apart in the order $\eta$).

Finally, for a vertex set $V \subset \mathbb{N}$, denote by
\be
\tilde{V} = \{ (u,v) \in V^{2} \, : \, u < v \}
\ee
the list of ordered pairs of vertices. {We note that a function $\tilde{F} \, : \, \tilde{V} \mapsto \{-1,0,1\}$ can be extended to a unique antisymmetric function $F \, : \, V^{2} \mapsto \{-1,0,1\}$; we use this extension without comment.}

\subsection{Algorithms for Theorem \ref{ThmReconMainRes1}}

Our main algorithm, Alg.~\ref{AlgMerge}, uses Alg.\ \ref{AlgSketch} to find an initial coarse comparison $F$ based on $\GA{\alpha}$, and then ``fills in" the remaining comparisons by taking larger subsamples and merging them by a voting procedure (Alg.\ \ref{AlgUseSketch}). This will give a comparison $F'$ that agrees with the line-embedded permutation $\sigma_{true}$ or its reversal $\eurt$ for all pairs sufficiently far apart, which implies that $\sigma_{F'}$ is close to either $\sigma_{\mathrm{true}}$ or $\eurt$.

\setcounter{algocf}{0}
\RevertAlgoNumber

\smallskip
\begin{algorithm}[H]
\caption{Main Estimation Algorithm \label{AlgMerge}}
\flushleft
\textbf{parameters:} Sample size $m \in \mathbb{N}$, running time $t$, threshold $\zeta $, and truncation level $\alpha$.\\
 \textbf{input:}  Graph $G = (V,E)$ of size $|V| = n \geq m$.\\
 \textbf{output:}  A  full order  $\sigma$ on $V$.\\
\begin{algorithmic}[1]
\STATE Compute $\GA{\alpha}$ according to Equation \eqref{EqThreshDef}. \\
\STATE Run Alg.\ \ref{AlgSketch} with parameters $t$, $m$, and $\zeta$.  
Set $F =$ \textit{SparseSketch}$(\GA{\alpha})$. \\
\COMMENT{With high probability, $F$ has error $\lesssim \log(n)^{-1}$.}
\STATE Run Alg.~\ref{AlgUseSketch}. 
Set $F' =$ \textit{LocalRefinement}$(\GA{\alpha},F)$.
\STATE Return $\sigma_{F'}$.

\end{algorithmic}
\end{algorithm}
\smallskip

Most of the work is done in the function \textit{SparseSketch} (Alg. ~\ref{AlgSketch}). The most important step of Algorithm  ~\ref{AlgSketch} is the repeated call to \textit{OrderedSubsample} (Alg.~\ref{AlgSubsamp}), which samples small subgraphs of $\GA{\alpha}$ and then estimates orders on these small subgraphs.

The motivation for repeated subsampling is that, with high probability, the vast majority of these subgraphs typically have a special property that $\GA{\alpha}$ does not have: they will be \textit{unit interval graphs} (see Section \ref{Sec:UnitIntervalGraphs}). This special property allows us to use an existing efficient algorithm for finding the line-embedded permutation, referred to here as LexBFS, that only applies to unit interval graphs. 

Having used this efficient algorithm for our subgraphs, it is then necessary to stitch together their estimated orderings. The function \textit{GlobalOrder} (Alg.~\ref{AlgGlobal}) is needed because we expect some of the small samples to be ordered according to the true line-embedded permutation, and others according to its total reversal;  they must be aligned to  either all agree with the true line-embedded permutation, or all agree with its total reversal. 

The number of iterations is chosen such that each pair of vertices appears together in many of the samples. After aligning all the samples, we then simply count how often each pair is ordered `up' or `down', and decide the final ordering by majority vote. 

\smallskip
\begin{algorithm}[H]
\caption{\textit{SparseSketch} \label{AlgSketch}}
\flushleft
\textbf{parameters:} Stopping time $t$, size $m \in \mathbb{N}$, and alignment threshold $\zeta$.\\ 
 \textbf{input:}  ``Square-thresholded" graph $H = (V,E)$ of size $|V| = n$.\\
 \textbf{output:}  A comparison function $F$ on $V$. \\
\begin{algorithmic}[1]
\STATE Initialize the list of pairwise comparisons and number of comparisons $\mathcal{C}\, : \, \tilde{V} \mapsto \mathbb{Z}$ by $\mathcal{C} \equiv 0$. Also set $i=0$.
\WHILE{$i < t$}
\STATE Call Alg.~\ref{AlgSubsamp} with parameter $m$ and input $H$. If algorithm succeeds, set $(\sigma, S)=$\textit{OrderedSubsample}$(H)$ and $Z=1$; otherwise $Z=0$.
\IF{$Z=1$}
\STATE Set $i=i+1$.
\STATE Set $(\sigma^{(i)}, S^{(i)}) = (\sigma, S)$.
\ENDIF
\ENDWHILE
\STATE Run Alg. \ref{AlgGlobal} with parameter $\zeta$. Set $a=$\textit{GlobalOrder}$((\sigma^{(1)}, S^{(1)}),\ldots,(\sigma^{(t)}, S^{(t)}))$. For $1 \leq j \leq t$, set:
\be
\sigma^{(j)}  &= \left\{ \begin{array}{ll}
\sigma^{(j)}, & a(j) = 1 \\
m +1 - \sigma^{(j)}, & a(j) = -1.
\end{array}\right.
\ee
\FOR{$j = 1,2,\ldots,t$}
\STATE For vertices $u,v \in S^{(j)}$, update $\mathcal{C}$ by
\be
\mathcal{C}(u,v) &= \left\{ \begin{array}{ll}
\mathcal{C}(u,v) +  1, &  \sigma^{(j)}(u)<\sigma^{(j)}(v) \\
\mathcal{C}(u,v) -  1, & \sigma^{(j)}(u)>\sigma^{(j)}(v)
\end{array}\right.\\
\ee
\ENDFOR
\STATE Return the function $F \, : \, \tilde{V} \mapsto \{-1,0,1\}$ given by
\be
F(u,v)  &= \left\{ \begin{array}{rl}
-1, &  \mathcal{C}(u,v) <  0, \\
1, & \mathcal{C}(u,v) > 0, \\
0 & \text{otherwise.}\\
\end{array}\right.
\ee

\end{algorithmic}
\end{algorithm}

\textit{OrderedSubsample} (Alg. ~\ref{AlgSubsamp}) uses a graph algorithm to order the vertices of a small random subgraph of $G$:

\begin{algorithm}[H]
\caption{\textit{OrderedSubsample} \label{AlgSubsamp}}
\flushleft
\textbf{parameter:} Size $m \in \mathbb{N}$.\\ 
 \textbf{input:}  ``Square-thresholded" graph $H = (V,E)$ of size $|V| = n$.\\
 \textbf{output:}  A set $S \subset V$ of size $m$, and a permutation  $\sigma$ on $S$.\\
\begin{algorithmic}[1]
\STATE Sample $S \sim \mathrm{Unif}(\{ T \subset V \, : \, |T| = m \})$. \\
\STATE Apply the {\sl LexBFS} algorithm to the sampled graph $H(S)$.
\STATE If $H(S)$ is a connected proper interval graph, say the algorithm succeeds and return $S$ and a line-embedded permutation $\sigma$ of $S$.

\end{algorithmic}
\end{algorithm}

The {\sl LexBFS} algorithm, which quickly recognizes and orders unit interval graphs, was first described in \cite{corneil2004}. The version of the algorithm in that paper doesn't detect if a graph is disconnected, but (since the algorithm is based on {\sl BFS}) it can be easily modified to do so.  More information about unit interval graphs and this algorithm is given below in Section \ref{Sec:UnitIntervalGraphs}.

\textit{GlobalOrder} (Alg. ~\ref{AlgGlobal}) is used to globally align the estimated orderings for individual subgraphs:

\begin{algorithm}[H]
\caption{\textit{GlobalOrder} \label{AlgGlobal}}
\flushleft
\textbf{parameter:} Threshold $\zeta >0$.\\
\textbf{input:}  List of pairs $(\sigma^{(j)}, S^{(j)})_{j=1}^t$ of permutations $\sigma^{(i)}$ on sets $S^{(i)} \subset [n]$ of common size $|S^{(i)}| = m$. \\
 \textbf{output:}  A global alignment $a \, : \, \{1,2,\ldots,t\} \mapsto \{-1,+1\}$.\\
\begin{algorithmic}[1]
\STATE Initialize the number of pairwise comparisons $\mathcal{N} \, : \, \tilde{V} \mapsto \{0,1,\ldots\}$ by $\mathcal{N} \equiv 0$.
\FOR{$j \in \{1,2,\ldots,t\}$}
 \STATE Set
\be
\mathcal{L}^{(j)} &= \{ u \in S^{(j)} \, : \, \sigma^{(j)}(u) < \zeta \} \\
\mathcal{R}^{(j)} &= \{ u \in S^{(j)} \, : \, \sigma^{(j)}(u) > m - \zeta \}. \\
\ee
\ENDFOR
\STATE Initialize our estimated pairwise alignment $H \, : \, \{1,\ldots,i\}^{2} \mapsto \{-1,0,+1\}$ by $H \equiv 0$.
\FOR {$j,k \in \{1,2,\ldots,t\}$}
\STATE Set
\be
H(j,k) &= \begin{cases}
1, \, \qquad (\mathcal{L}^{(j)} \cap \mathcal{L}^{(k)}) \cup (\mathcal{R}^{(j)} \cap \mathcal{R}^{(k)}) \neq \emptyset, \, \text{otherwise}\\
-1, \qquad (\mathcal{L}^{(j)} \cap \mathcal{R}^{(k)}) \cup (\mathcal{R}^{(j)} \cap \mathcal{L}^{(k)}) \neq \emptyset, \, \text{otherwise}\\
0.
\end{cases}
\ee
\ENDFOR
\STATE Find a function $a \, : \, \{1,2,\ldots,t\} \mapsto \{-1,1\}$ so that, for all $i,j$ with $H(i,j)\in \{ -1,1\}$, $H(i,j)=a(i)a(j)$. \\
\COMMENT{We will show that \wep\ such a function always exists.}

\STATE Return $a$.
\end{algorithmic}
\end{algorithm}

The sets $\mathcal{L}^{(j)}$ and $\mathcal{R}^{(j)}$ capture the vertices in sample $j$ with highest and lowest $U$-values, respectively. Sample size is chosen so that there is sufficient overlap between these sets between samples. If  $\mathcal{L}^{(i)}$ overlaps with  $\mathcal{L}^{(j)}$, then this is an indication that the samples $i$ and $j$ are aligned, and $H(i,j)$ is set to 1, whereas if  $\mathcal{L}^{(i)}$ overlaps with  $\mathcal{R}^{(j)}$, then this indicates that the samples have opposite alignment, and $H(i,j)$ is set to $-1$.

When it exists, the function $a$ can be found very quickly by a greedy search (step 9). In practical implementations, it may be better to choose an algorithm that is more robust to small errors (e.g. solve the usual convex relaxation of the problem, then round the solution to the closest value in $\{-1,1\}$).

\begin{remark} \label{RemSpeedingUpSlow}

The cost of Algorithm \ref{AlgGlobal} is dominated by the cost $t^{2}$ of constructing the $t$ by $t$ matrix $H$. We note that the choice of number of samples $t$ and their sizes $m$ are driven by two concerns: the total number of samples $mt$ must be at least $n^{2} \mathrm{polylog}(n)$ to ensure that each pair of vertices appears together in many samples, and the size of the sample $m$ must be small enough that the subsampled graphs have certain good properties with high probability (primarily, they must be Robinsonian). 

We expect that in fact subgraphs up to size $m' = \frac{\sqrt{n}}{\mathrm{polylog(n)}}$ would have these good properties. If true, this would allow us to replace $t$ by $t' = O(n^{1.5} \mathrm{polylog}(n))$, reducing the computational cost of the algorithm by a factor of $n$. 
\end{remark}

Finally, we describe \textit{LocalRefinement} (Alg.~\ref{AlgUseSketch}), which refines the initial sketch provided by Algorithm \ref{AlgSketch}.  We use the notation $N_G(u)=\{ w\in V(G)\,:\, (u,w)\in E(G)\}$ to denote the {\sl neighbourhood} of $u$ in $G$, omitting the subscript $G$ when it is clear from context:

\begin{algorithm}[H]
\caption{\textit{LocalRefinement} \label{AlgUseSketch}}
\flushleft
\textbf{input:}  ``Square-thresholded" graph $H = (V,E)$ of size $|V| = n$, and a comparison $F$.\\
 \textbf{output:}  A comparison $F'$.\\
\begin{algorithmic}[1]

\STATE Initialize $F' \, : \, \tilde{V} \mapsto \{-1,0,1\}$ by $F' \equiv 0$.

\FORALL{$(u,v) \in  \tilde{V}$}
\STATE Set
\be
D_{u,v} &=|\{ x\in N(u)\setminus N(v)\,:\, F(x,v)=1\}|-|\{ x\in N(u)\setminus N(v)\,:\, F(x,v)=-1\}| \\
D_{v,u} &=|\{ x\in N(v)\setminus N(u)\,:\, F(x,u)=1\}|-|\{ x\in N(v)\setminus N(u)\,:\, F(x,u)=-1\}|
\ee
\IF{$|D_{u,v}|>|D_{v,u}|$}
\STATE Set
\be
F'(u,v) &= \begin{cases}
1, \, \qquad D_{u,v} > 0\\
-1, \qquad D_{u,v} \leq 0\\
\end{cases}
\ee
\ELSE
\STATE Set
\be
F'(u,v) &= \begin{cases}
1, \, \qquad D_{v,u} < 0\\
-1, \qquad D_{v,u} \geq 0.\\
\end{cases}
\ee
\ENDIF
\ENDFOR
\STATE  Return $F'$.
\end{algorithmic}
\end{algorithm}

Algorithm \ref{AlgUseSketch} is based on the following observations. 
Recall that Algorithm \ref{AlgUseSketch} takes as input a ``rough'' ordering from some comparison $F$. Now consider $i,j\in V$ so that $U_i<U_j$. Then any vertex $k$ with $U_k>U_j$ is closer to $U_j$ than to $U_i$,  so more likely to be a neighbour of $j$ than of $i$. Thus, we expect that $j$ will have more neighbours than $i$ that are higher up in the ordering than $j$. Similarly, we expect that $j$ will have fewer neighbours than $i$  that are lower down in the ordering than $i$.  Therefore, we expect to be able to order $i,j$ by counting the number of elements of the  neighbourhoods $N(i)\setminus N(j)$ and $N(j)\setminus N(i)$ that are higher than $i,j$ according to the rough ordering $F$.

\subsection{Algorithms for Theorem \ref{ThmReconMainRes2}}

Algorithm \ref{AlgMerge} described in the previous subsection gives an ordering with error approximately $\sqrt{n}$. Here we describe the iterative refinement algorithm, which sequentially reduces this error.

Algorithm \ref{AlgIterationHolder} creates a coupled sequence of random graphs $G_1 \subset G_{2} \subset \ldots  \subset G_k=G$. Algorithm \ref{AlgMerge} is used with input $G_1$ to obtain an initial ordering of the vertices in $G_1$. Then, in each successive step,  Algorithm \ref{AlgRefine} is called to replace the order $F_{i}$ on $G_{i}$ by the more-accurate order $F_{i+1}$ on the larger graph $G_{i+1}$. The iterative step is based on $G$ itself, not $\GA{\alpha}$.

\begin{algorithm}[H]
\caption{Iterative-Improvement Estimate \label{AlgIterationHolder}}
\flushleft
\textbf{parameter:} Sequence of rates $0 < p_{1} < p_{2} < \ldots < p_{k}=1 \in [0,1]$, errors $d_1>d_2>\dots >d_k$,
truncation level $\alpha$.\\
 \textbf{input:}  Graph $G = (V,E)$ of size $|V| = n$.\\
 \textbf{output:}  A permutation  $\sigma$ on $V$.\\
\begin{algorithmic}[1]
\STATE Let $\{B_{i}\}_{i=1}^{n} \stackrel{i.i.d.}{\sim} \mathrm{Unif}[0,1]$.
\STATE Run Algorithm \ref{AlgMerge} on the induced subgraph $G_{1}$ of $G$ with vertex set $V_{1} = \{j \in V \, : \, B_{j} \leq p_{1}\}$, $n_1=|V_1|$, and parameters $\alpha$ and $m= m(|V_{1}|), t = t(|V_{1}|) \zeta = \zeta(|V_{1}|)$ as in  Equation \eqref{def:GoodParameters}. Let $\sigma_{1}$ be the returned ordering of $V_1$.

\FOR{$i \in \{1,2,\ldots,k-1\}$}
\STATE Let $G_{i+1}$ be the induced subgraph of $G$ with vertex set $V_{i+1} = \{j \in V \, : \, B_{j} \leq p_{i+1}\}$.
\STATE Run Algorithm \ref{AlgRefine} with parameters $C_1=\lceil p_id_in\log (n)^4\rceil$, $C_2=\lfloor \sqrt{d_ip_in}\log n^{6}\rfloor$, $C_3=\lfloor \sqrt{d_ip_i n}\log (n)^{2}\rfloor$. Let $\sigma_{i+1} =\textit{Refine}(G_{i+1},V_i, \sigma_i)$.
\ENDFOR
\STATE Return $\sigma_{k}$.
\end{algorithmic}
\end{algorithm}

All of the work  in Algorithm \ref{AlgIterationHolder} is done in Algorithm \ref{AlgRefine}, the refinement algorithm that is called in the loop. It makes use of the following consequence of Assumption \ref{AssumptionSharpBoundary}: if $U_i<U_j$, then there is a region $R \subset [0,1]$ so that any vertex with $U$-value in $R$ can only be a neighbour of one of $i,j$. If $U_i$ and $U_j$ are sufficiently far apart, then $R$ is large and will contain many vertices; these will provide the signal that indicates the true ordering of $i$ and $j$. Of course, we don't know which vertices have $U$-values in $R$. But we know that the vertices that provide the signal must be at the extremes of the ordering in the neighbourhood of $i$ or $j$, so we can use the ordering from the previous iteration to find these vertices. The key idea is that, 
in each loop, we use our current estimate of $\sigma$ to improve our estimate of $R$, then use our improved estimate of $R$ to further improve our estimate of $\sigma$.

To make this precise, we introduce some definitions.  Given an ordering $\sigma$ of a set $V$,  a set $S\subset V$, and a parameter $c<|S|$, we define the sets of  the $c$ elements of $S$ with the highest and lowest rank according to $\sigma$:

\be
\label{Def:RL}
R(S,\sigma,c)&=\{ k\,:\,|\{ p\in S\,:\, \sigma(k)\leq \sigma (p)\}|\leq c\},\\
L(S,\sigma,c)&=\{ k\,:\,|\{ p\in S\,:\, \sigma(p)\leq \sigma (k)\}|\leq c\}.\\
\ee

Using these and an estimated ordering to approximate the sets $R,L$, we can then estimate a new ordering of $i,j$ by counting the number of neighbours $N(i), N(j)$ in these approximations and comparing the results. We note that this comparison, occurring in steps 4-25 of Algorithm \ref{AlgRefine}, is  more complicated than one might expect from this description. The additional steps allow us to ``prune out" and then ignore comparisons that we are not sufficiently sure of, preventing this uncertainty from ``infecting" other calculations.

\begin{algorithm}[h]
\caption{\textit{Refine} \label{AlgRefine}}
\flushleft
\textbf{parameters:} thresholds $C_1$, $C_2$, $C_3$.\\
 \textbf{input:}  Graph $G = (V,E)$, sets $V_1\subset V_2\subset V$, order $\sigma_1$ on $V_1$.\\
 \textbf{output:}  An order  $\sigma_2$ on $V_2$.
\begin{algorithmic}[1]
\FOR{$i<j$ in $V_2\setminus V_1$}
\STATE Set
\be
R = R(i,j)=R(V_1\cap(N(i)\cup N(j)),\sigma_1,C_1),\\
L = L(i,j)=L(V_1\cap(N(i)\cup N(j)),\sigma_1,C_1).
\ee
\STATE Initialize $F^{(2)}(i,j)=0$.
\IF{$|N(j)\cap R|> |N(i)\cap R| + C_2$ or $|N(i)\cap L|>|N(j)\cap L|+C_2$}
\STATE Set $F^{(2)}(i,j)=1$.
\ELSE\IF{$|N(i)\cap R|> |N(j)\cap R| + C_2$ or $|N(j)\cap L|>|N(i)\cap L|+C_2$}
\STATE Set $F^{(2)}(i,j)=-1$.
\ENDIF
\ENDIF
\ENDFOR
\STATE Let $\sigma_2' $  be a permutation of $V_2\setminus V_1$ that is compatible with $F^{(2)}$, i.e. $\sigma_2'=\sigma_{F^{(2)}}$ as in \eqref{EqDefOrderingMap}.
\FOR{$i$ in $V_2$}
\STATE Set
\be
t(i)&=\max\{ \sigma_2'(k)\,:\, k\in N(i)\cap (V_2\setminus V_1)\}\\
b(i)&=\min\{ \sigma_2'(k)\,:\, k\in N(i)\cap (V_2\setminus V_1)\}.
\ee
\ENDFOR
\FOR{$i\in V_1$, $j\in V_2$}
\STATE Initialize $F^{(2)}(i,j)=0$.
\IF{$t(j)-t(i)>C_3$ or $b(j)-b(i)>C_3$}
\STATE Set $F^{(2)}(i,j)=1$
\ELSE\IF{$t(i)-t(j)>C_3$ or $b(i)-b(j)>C_3$}
\STATE Set $F^{(2)}(i,j)=-1$
\ENDIF
\ENDIF
\ENDFOR
\STATE Return a permutation $\sigma_2 =\sigma_{F^{(2)}}$ of $V_2$ that extends $F^{(2)}$.

\end{algorithmic}
\end{algorithm}

\section{Proof of Strengthened Version of Theorem \ref{ThmReconMainRes1}} \label{SecAlgAnalysis}

In this section, we will prove a slightly stronger version of Theorem \ref{ThmReconMainRes1}. Since this section is fairly long and the proof does not follow the algorithm line-by-line, we give a quick guide to the subsections of this section in order. Sections 3.3 and 3.4 are particularly important, as they give the ``big-picture" explanation of why the algorithm works and what needs to be checked:

\begin{enumerate}
\item[3.1] We state a version of Theorem \ref{ThmReconMainRes1} that holds under weaker conditions, and state those (highly technical) conditions.
\item[3.2] We describe the probability space that will be used for the remainder of the proof.
\item[3.3] We give a detailed description of several key properties of ``typical" realizations of the random graph $G$, as well as key properties of several derived objects that appear in Algorithm \ref{AlgMerge}, and explain why those properties ensure that Algorithm \ref{AlgMerge} works. This section also includes technical lemmas showing that these key properties occur with high probability. 
\item[3.4] We give a sketch of the most important parts of our proof, and how it depends on the typical properties described in the previous section. 

\item[3.5] With the preliminaries out of the way, we analyze the output of Algorithm \ref{AlgSubsamp}, the sub-algorithm that  orders small subgraphs of $G$.
\item[3.6] We analyze the output of Algorithm \ref{AlgGlobal}, the sub-algorithm that aligns the orderings of the small subgraphs returned by Algorithm \ref{AlgSubsamp} (this alignment is necessary because even ``good" outputs of \ref{AlgSubsamp} will be randomly aligned with either $\sigma_{true}$ or $\sigma_{eurt}$).
\item[3.7] We analyze the output of Algorithm \ref{AlgSketch}, the sub-algorithm that calls Algorithms \ref{AlgSubsamp} and \ref{AlgGlobal} to obtain an estimated order of the full graph $G$, and give bounds on the error that hold \emph{whp} (though these bounds may still be quite large).
\item[3.8] We complete the analysis of the remaining sub-algorithm of Algorithm \ref{AlgMerge}, which refines the estimate obtained in Algorithm \ref{AlgSketch}.
This analysis shows that \emph{whp} the refined ordering has error bounded as stated in Theorem \ref{ThmMaxOrd}. This completes the proof of Theorem \ref{ThmMaxOrd}.
\end{enumerate}

\subsection{A Stronger Result}\label{SectionSlightlyStronger}

Rather than proving Theorem \ref{ThmReconMainRes1} directly, we prove a related result with  weaker (but more complicated) assumptions. This theorem applies to graphons in general, not just graphons with a uniform embedding. We verify that the stronger assumptions of Theorem \ref{ThmReconMainRes1} imply the more general assumptions presented in this section in Appendix \ref{SubsecWeakStrongAssumptionsCheck}.
While weakening the assumptions in this way will provide a stronger final result, our main motivation was to replace the assumptions of Theorem \ref{ThmReconMainRes1} (which are easy to read but hard to work with directly) with the weaker assumptions of Theorem \ref{ThmMaxOrd} (which are harder to read but easier to work with in our proof).

Before we discuss the weaker assumptions, we give an alternative characterization of  $\{ 0,1\}$-valued graphons.
It is not hard to verify (see also the discussion following Definition 2.1 of \cite{chuangpishit2017uniform}) that, if  $w$ is a $\{ 0,1\}$-valued graphon, then $w$ is diagonally increasing if and only if there exist two non-decreasing functions $\ell:[0,1]\rightarrow [0,1]$ and $r:[0,1]\rightarrow [0,1]$ which demarkate the regions in $[0,1]^2$ where $w$ assumes the value 1. The functions $\ell$ and $r$ will be called the \emph{boundaries} of $w$. Precisely, $w$ is  of the form
\be
w(x,y)=\left\{ \begin{array}{ll} 1&\text{if }y\in I(x)\\
0&\text{otherwise, }\end{array}\right.
\ee
where $(\ell (x),r(x)) \subset  I(x)\subset [\ell(x),r(x)]$. Thus, the boundaries define $w$ up to the values exactly on points of the form $(\ell(x),x)$ and $(x,r(x))$; of course this collection of points has measure 0, and thus does not influence the distribution of samples from $w$.

In this section, which leads up to the proof of Theorem \ref{ThmMaxOrd}, we will apply this characterization in particular to $\WA{\alpha}$, which is almost everywhere defined by the boundaries $r_\alpha$ and $\ell_\alpha$ defined as follows:
\be\label{Eq:Boundaries}
r_\alpha(x) &= \sup\{ y \in [0,1]\,:\, \WA{}(x,y) {\geq} \alpha\},\\
\ell_\alpha (x)&= \inf\{  y \in [0,1 ]\,:\, \WA{}(x,y) {\geq} \alpha\}.
\ee

We now state the definitions required for our weaker assumptions. For a set $S \subset \mathbb{R}^{d}$, denote by $\Vol(S)$ the $d$-dimensional volume $S$. Then define:

\begin{defn}[Goodness] \label{DefGoodness}
Fix $0 < \alpha, \delta  < 1$ and $0 < A < \infty$.
Say that $w$ is \textit{uniformly $(A,\delta)$-good at $\alpha$} if it satisfies
\be
\sup_{x \in [0,1]} \, \Vol (\{ y \in [0,1] \, : \, |\alpha - w(x,y)| \leq  \delta' \}) \leq A \delta'
\ee
for all $0 < \delta' < \delta$.

\end{defn}

In other words: if $w$ is uniformly $(A,\delta)$-good at $\alpha$, then the values of $w(x,\cdot)$ don't concentrate around $\alpha$. We will not apply this assumption to $w$ directly; we apply it to $w^{(2)}$.
Next, we use a simple notion of ``connectedness" for graphons:

\begin{defn} [Connectedness of Graphons] \label{DefGraphonComponents}
Fix $0 < \epsilon < 1$. We say that a graphon $w$ is $(\epsilon,\alpha)$-connected if
\be
\inf\{ w(u,v)\,:\, u, v \in [0,1],\, |u - v | \leq \epsilon \}  >\alpha.
\ee
\end{defn}

For  graphons $w$, to be $(\epsilon,\alpha)$-connected means that the region where $w$ achieves values greater than $\alpha$ contains a strip of width $\epsilon$ around the diagonal. We will see later that this implies that, asymptotically almost surely, large graphs sampled from the $\alpha$-thresholded version of such a graphon are connected.

The following ``separation" property implies that the neighbourhoods of far-apart vertices are easy to distinguish:

\begin{defn} [Separation of Graphons] \label{DefGraphonSeparation}
Fix $0 < \epsilon < \infty$. We say that a $\{ 0,1\}$-valued graphon $w$ is $\epsilon$-separated if, for all $x,y$,
\be
\mathrm{Vol}(\{ z  \in [0,1] \, : \, |w(x,z) - w(y,z)| = 1 \} ) \geq \epsilon |x-y|.
\ee

\end{defn}

Diagonally increasing graphs may have many indistinguishable rows - as an extreme example, the complete graph is Robinsonian! 
The following assumption rules out this sort of near-complete indistinguishability by requiring very different points to have disjoint neighbourhoods: 

\begin{defn} [Splitting graphons] \label{DefGraphonContainment}
Fix $0<\epsilon <1$. We say that  a $\{ 0,1\}$-valued graphon $w$ has an  $\epsilon$-split if for all $x \leq y - (1-\epsilon)$, 
\be
\{ z  \in [0,1] \, : \, w(x,z) = w(y,z) = 1 \} = \emptyset.
\ee
\end{defn}

In graphs sampled from a graphon with an $\epsilon$-split, vertices with $U_i$-values that are far apart cannot have any common neighbours. This guarantees that a certain ``line-embeddedness" is preserved in the graph and its samples, and a correct permutation can be extracted from the samples.

We combine these properties:

\begin{assumptions} [Weak Assumptions] \label{AssumptionGoodGraphon}
Let $w$ be a graphon, and $\alpha, \epsilon \in (0,1)$ and $A > 0$ be constants, so that:
\begin{enumerate}
\item $\WA{\alpha}$ is diagonally increasing, and
\item $w^{(2)}$ is uniformly $(A,\epsilon)$-good at $\alpha$, and
\item $\WA{}$ is $(\epsilon ,\alpha )$-connected, and
\item $\WA{\alpha}$ has an $\epsilon$-split, and
\item $\WA{\alpha}$ is $\epsilon $-separated.

\end{enumerate}
\end{assumptions}

Note that most of these assumptions are properties of $\WA{\alpha}$, not $\WA{}$. In particular, we do not need that $\WA{}$ is diagonally increasing, only that the thresholded $\{0,1\}$-valued graphon $\WA{\alpha} $ is. This is an important distinction: it is not true, in general, that the square of a diagonally increasing graphon is itself also diagonally increasing (see Appendix \ref{SubsecWeakStrongAssumptionsCheck}, Remark \ref{rem:counterexample} for a simple example). Even if $\WA{}$ is itself not diagonally increasing, there still may be a value of $\alpha$ for which $\WA{\alpha}$ is diagonally increasing.

Finally, we define the set of parameters for our main algorithm that gives the desired result for Theorem \ref{ThmReconMainRes1}.

\be \label{def:GoodParameters}
m&= m(n) \equiv \log(n)^\csketch,\\
t&= t(n) \equiv (n\log (n))^2,\\
\zeta &= \zeta(n) \equiv 4\log(n)^4\\
\ee

We are now ready to state the stronger theorem.

\begin{thm} [Ordering on Large Sample] \label{ThmMaxOrd}
Fix a graphon $w$ and constant $\alpha$ that satisfy Assumption \ref{AssumptionGoodGraphon}.
Let $G_{n} \sim w$ be a graph of size $n \in \mathbb{N}$. Then when Algorithm \ref{AlgMerge} is executed with parameters as in \eqref{def:GoodParameters} on input  graph $G_{n}$ and value $\alpha$, then
the output is an ordering ``$\sigma$" on $\{1,2,\ldots,n\}$ with error that satisfies
\be
\err \leq  \sqrt{n} \log(n)^{\cfin}
\ee
w.e.p.
\end{thm}

\subsection{Construction of Probability Spaces} \label{SubsecInitNot}

We set some notation that will be used throughout the remainder of Section \ref{SecAlgAnalysis}. Fix notation as in the statement of Theorem \ref{ThmMaxOrd}. Fix two i.i.d. sequences $\{U_{i}\}_{i \in \mathbb{N}}$, $\{U_{i,j}\}_{i < j \in \mathbb{N}}$ with uniform distribution on $[0,1]$. We will use these two sequences to couple all of the graphs described in this section as follows: for \textit{any} graph of size $n \in \mathbb{N}$ sampled from a graphon in this section, we will \textit{always} assume that the graph is sampled by using the elements of \textit{this fixed} pair of sequences with indices $1 \leq i, j \leq n$, in the representation \eqref{EqGraphonDef}. In particular, if $G, G' \sim w$ are graphs of sizes $n < n'$, then this coupling gives a natural embedding $G\subset G'$ of the graphs. For fixed $n$, we denote by $\sigma_{\mathrm{true}}$ the line-embedded permutation associated with the sequence $\{U_{i}\}_{i=1}^{n}$, and denote by $F_{\mathrm{true}}$ the associated comparison function. Permutation $\eurt$ is the total reversal of $\sigma_{\rm true}$.

Denote by $H_{ \alpha} \sim  \WA{ \alpha}$ a ``correct" thresholded graph, drawn from $ \WA{\alpha}$ (the thresholded square graphon as in Definition \ref{DefGraphPower}) using the same random variables $\{U_{i}\}$, $\{ U_{i,j}\}$ as $G$. (Since $\WA{\alpha}$ is a $\{ 0,1\}$-valued graphon, the variables $\{ U_{i,j}\}$ are irrelevant.)
Thus, for all $i,j$, 
\[
H_\alpha (i,j)=\WA{\alpha}(U_i,U_j).
\]
Recall the definition of the threshold-square graph $\GA{\alpha}$ from Equation \eqref{EqThreshDef}. For $S \subset \{1,2,\ldots,n\}$, denote by $G(S)$  (respectively $\GA{ \alpha}(S)$, $H_{ \alpha}(S)$) the induced subgraph of $G$ (respectively $\GA{ \alpha}$, $H_{ \alpha}$) on vertex set $S$.

We define $\mathcal{F}_{1}(n)$ (respectively $\mathcal{F}_{2}(n)$) to be the $\sigma$-algebra generated by the sequence $\{U_{i}\}_{i =1}^{n}$ (respectively $\{U_{i,j}\}_{1 \leq i < j \leq n}$). We define $\mathcal{F}_{k} \equiv \mathcal{F}_{k}(\infty)$ for $k \in \{1,2\}$. Finally, we set $\mathcal{F} = \mathcal{F}_{1} \cap \mathcal{F}_{2}$ to be the $\sigma$-algebra generated by all of these random variables.

\subsection{Proof Sketch and Main Estimates for Bad Sets} \label{SecMainEstHeuristic}

In this section, we will show that the graphs $G, G^{(2)}_\alpha$  resulting from the random draw $\{U_{i}\}$, $\{ U_{i,j}\}$ will have certain good properties with extreme probability, and sketch out how these properties are used to guarantee that the algorithm succeeds.

\subsubsection{Unit interval graphs and the {\sl LexBFS} Algorithm}\label{Sec:UnitIntervalGraphs}

The heart of Algorithm \ref{AlgSubsamp} is the 3-sweep LexBFS unit interval graph recognition algorithm of \cite{corneil2004}. We recall the basic properties of unit interval graphs and this algorithm that will be used. A graph $G$ is a unit interval graph if and only if there exists a total order of the vertices so that the adjacency matrix $A$  of $G$ with respect to that order is diagonally increasing (where we set the diagonal elements $A_{i,i} = 1$). We will refer to such orders as {\sl interval orders}. If the graph is connected, then the interval order is unique, up to its total reversal and to permutation of any duplicated neighbourhoods. More precisely, for a total order of the vertices represented by a permutation $\sigma$, let $A^\sigma$ be the matrix defined as:
\be
A^\sigma_{i,j}=A_{\sigma^{-1} (i),\sigma^{-1} (j)}.
\ee
Then $\sigma$ is an interval order if and only if $A^\sigma$ is diagonally increasing. If $G$ is connected, then $A$ is irreducible. In that case, if $\sigma$ and $\tau$ are both interval orders, then either $A^\tau =A^\sigma$ (which allows for the possibility that some rows are identical), or $A^\tau = A^{\sigma^{trev}}$, where $\sigma^{trev}$ is the total reversal of $\sigma$.

\subsubsection{Properties of Latent Variables} \label{SubsubsecLatProp}

Algorithm \ref{AlgMerge} is based on the assumption that the square thresholded graph $\GA{\alpha}$ has similar properties to $H_\alpha$, and that the vertices in $V$ are fairly uniformly spread over the interval $[0,1]$. This, together with our assumptions on $\WA{}$ and $\WA{\alpha}$ will then guarantee that small samples from $\GA{\alpha}$ have similar good properties. This will imply that, with high probability, the ``sketch" returned by Algorithm \ref{AlgSketch} has no incorrect comparisons. In particular, we want to show that the following properties  hold w.e.p.:

\begin{enumerate}
\item For vertices $i,j \in V$ with $|U_{i} - U_{j}|\gg \frac{1}{\sqrt{n}}$, we have $\GA{ \alpha}(i,j) = H_{ \alpha}(i,j)$. That is, on a coarse scale, $\GA{\alpha}$ looks like it was sampled from $w^{(2)}_\alpha$.
\item Any random uniform subsample $S \subset V$ that is independent of $\mathcal{F}$ and has $|S| \geq \log(n)^{\csketch}$ will induce a connected graph $\GA{  \alpha}(S)$.
\item For all vertices $i,j \in V$ with $|U_{i} - U_{j}| \gg  \frac{1}{\sqrt{n}}$, there are many vertices that are connected to $i$ but not $j$ in $\GA{\alpha}$ (and vice-versa).
\item Pointwise, the empirical CDF of the collection $\{ U_i \}_{i=1}^{n}$ is never too far from its expected value. Specifically, for all vertices $i, j \in V$ with $|U_{i} - U_{j}| \gg  \frac{1}{\sqrt{n}}$, we have $|\{ k \in V \, : \, U_{k} \in (U_{i},U_{j}) \}| \approx n \, |U_{i}-U_{j}|$.

\end{enumerate}

Note that these are all purely properties of $\{U_{i}\}$, $\{ U_{i,j}\}$ (even property {(2)}, which is about \textit{typical draws of $S$ after observing the $\sigma$-algebra $\mathcal{F}$}). When these conditions hold, we know that when $i,j$ appear in the same subsample $S$ of size $|S|= \log(n)^{\csketch} $, the following will all occur with probability quite a bit larger than $\frac{1}{2}$:

\begin{enumerate}
\item[{(\sl i)}] The neighbourhood of $i,j$ in $\GA{\alpha}(S)$ will agree with their ``correct" neighbourhood in $H_{  \alpha}(S)$.
\item[{(\sl ii)}] The subgraph $\GA{ \alpha}(S)$ will be connected.
\item[{(\sl iii)}] If $U_i$ and $U_j$ are significantly far apart, the neighbourhoods of $i,j$ can be distinguished in $\GA{ \alpha}(S)$, and so in particular it is possible to compare them locally.

\end{enumerate}

When \textit{these} three events all occur for a given $i,j\in S$, we will have that $\GA{ \alpha}(S)$ is a connected unit interval graph whose interval orderings agree with one of the correct line-embedded permutations of $w^{(2)}_\alpha(S)$ for all vertices that are sufficiently far apart.
The interval ordering can be retrieved with a standard, efficient graph-theoretic algorithm as discussed in Section \ref{Sec:UnitIntervalGraphs}; moreover, the algorithm will detect when $\GA{\alpha} (S)$ is not a connected unit interval graph. Thus, when ${\sl (i)-(iii)}$ occur, we will be able to reconstruct the correct ordering of sufficiently distant pairs $i,j\in S$.

A precise statement of Conditions (1)-(4) above will be given in Definition \ref{DefBadEvents}, and Conditions (i)-(iii) will be given in Definition \ref{BadEventsSub} in the following subsection. First we introduce the necessary notation. For fixed $n$ and $\alpha$, define the ``bad set"
\be \label{EqDefBadSet}
\bad{n}{\alpha} = \left\{ (x,y) \in [0,1]^{2} \, : \, |w^{(2)}(x,y) - \alpha| \leq \frac{\log(n)^{\cbad}}{\sqrt{n }} \right\}. 
\ee
Thus, $\bad{n}{\alpha}$ is the subset of $[0,1]^2$ where $w^{(2)}$ takes value very close to $\alpha $.

We define the related ``bad witnesses" of a vertex $1 \leq i \leq n$ by
\be \label{EqDefBadWitness}
\bad{n,i}{\alpha} = \{ j \in \{1,2,\ldots,n\} \backslash \{i\} \, : \, \GA{\alpha}(i,j)\neq w^{(2)}_{ \alpha} (U_i,U_j)\}.
\ee
We will see that {\sl w.e.p.}\  $j\notin \bad{n,i}{\alpha}$  if $(U_i,U_j)\notin \bad{n}{\alpha}$  (see Lemma \ref{LemmaGoodPointsEq}).

We also define the collection of ``good witnesses" for a \textit{subset} $S \subset \{1,2,\ldots,n\}$ and \textit{pair} of vertices $i,j \in S$ by:
\be
\wit_{S}(i,j) &= \{ k \in S \, : \, \WA{ \alpha}(U_{i},U_{k}) \neq \WA{\alpha}(U_{j},U_{k}) \} \\
\nu_{S}(i,j) &= | \wit_{S}(i,j) |.
\ee
When $S = \{1,2,\ldots,n\}$, we will drop the subscript. These are the vertices that will allow us to distinguish $i,j$.

Finally, we define a cover of $[0,1]$ which will be useful in showing that subgraphs of $\GA{\alpha}$ are connected.
Let $\epsilon >0$ be as in Assumption \ref{AssumptionGoodGraphon}. Fix $R \in \mathbb{N} \cap [\epsilon^{-1}, 2 \epsilon^{-1}]$ and define the sets
\be
\label{DefIntervals}
I_{k} = \left[\frac{k }{4R}, \frac{(k+3)}{4R} \right] \cap [0,1], \quad 0 \leq k \leq 4R - 3.
\ee
Note that these intervals form a cover of $[0,1]$. Moreover, if $U_i\in I_k$ and $U_j\in I_k\cup I_{k+1}$, then $H_\alpha (i,j)=1$. This implies that, for any set $S\subset V$,
\be
\label{EqConnected}
\bigcap_{k=0}^{4R-3}\, \{ S\cap \{ i: U_i\in I_k\}\not= \emptyset\}\subset \{ H_{\alpha} \text{ is connected}\} .
\ee

We now give formal definitions of the collection of bad events:

\begin{defn}[Bad Events] \label{DefBadEvents}

Let graphon $w$, value $\alpha$, random variables $\{ U_i\}$, $\{ U_{i,j}\}$, and graphs $\GA{\alpha}$ and $H_{\alpha}$ be as defined above. Moreover, let $R$, $\{ I_k\}_{k=0}^{4R-3}$ be as in (\ref{DefIntervals}).
We define:
\be
\mathcal{A}_{1} &= \cup_{1 \leq i < j \leq n}(\{(U_{i},U_{j}) \notin \bad{n}{\alpha}\} \cap \{ \GA{ \alpha}(i,j) \neq H_{ \alpha}(i,j)\}) \\
\mathcal{A}_{2} &= \left\{\max_{1 \leq i \leq n} | \bad{n,i}{\alpha} | > \sqrt{n } \, \log(n)^{\cat}\right\} \\
\mathcal{A}_{3} &= \left\{ \min_{0 \leq k  \leq 4R-3} |\{ i\,:\, U_i\in I_k\}|< \frac{n}{2R} \right\} \\
\mathcal{A}_{4} &= \left\{\min_{1 \leq i,j \leq n \, : \, |U_{i}-U_{j}| > \frac{1}{\log(n)^{\cafo}}} \nu (i,j) < \frac{2n}{\log (n)^{\caft}}\right\} \\
&\quad \cup \left\{\min_{1 \leq i,j \leq n \, : \, |U_{i}-U_{j}| > \frac{\log(n)^{\cafi}}{\sqrt{n}}} \nu (i,j) < \sqrt{n} \log(n)^{\cafti} \right\} \\
\mathcal{A}_{5} &= \left\{  \max_{1 \leq i < j \leq n} |\{ k \, : \, U_{k} \in (U_{i},U_{j}) \}| \geq n |U_{i}-U_{j}| +  \sqrt{n}\log(n)  \right\}.\\
\ee

\end{defn}

The event $\mathcal{A}_1^c\cap \mathcal{A}_2^c$ implies property (1) above, while $\mathcal{A}_3^c$, $\mathcal{A}_4^c$ and $\mathcal{A}_5^c$  imply properties (2), (3) and (4), respectively. The bad events mostly involve random variables that are the sum of the independent variables $\{ U_i\}$,$\{ U_{i,j}\}$, and by straightforward application of standard concentration inequalities we can show that \wep\ none of these bad events occur. The lemmas and proofs establishing this fact can be found in the Appendix. Here, we only state their corollary:

\begin{corollary} \label{IneqLatentCor}
Let $\mathcal{A}= \mathcal{A}_1\cup \mathcal{A}_2\cup \mathcal{A}_3\cup \mathcal{A}_4\cup \mathcal{A}_5$. Then $\mathcal{A}^{c}$ holds \wep.
\end{corollary}
\begin{proof}
This follows from Lemmas \ref{LemmaGoodPointsEq}, \ref{LemmaFewBadWitnesses}, \ref{LemmaTightlyConnected}, \ref{LemmaManyWitnesses}, and \ref{LemmaCompOrders}.
\end{proof}

We can therefore assume that our graph $G$ and the derived graph $\GA{\alpha}$ have the properties we need for the algorithm to succeed. 

\subsubsection{Properties of Subsampled Graphs}\label{sec:PropertiesSamples}

In Algorithm \ref{AlgSketch}, small subgraphs are sampled repeatedly from the square thresholded graph $\GA{\alpha}$. In this section, we show that, if $\mathcal{A}^c$ holds and thus $\GA{\alpha}$ has properties as expected, with high probability these samples will have similar good properties. To this end, we define bad events for the sample, and will then proceed to bound the probability that they occur.

Let $S$ be a uniformly chosen subset of $V$, of size  $m = \log(n)^{\csketch}$.
We denote by
\be
\label{DefS''}
S'' = \{ i \in S \, : \, \forall j \in S, \, i \notin  \bad{n,j}{\alpha} \}
\ee
the collection of elements with no ``bad" comparisons.

Choosing indices to roughly match the ``global" bad events  $\mathcal{A}_{1},\ldots,\mathcal{A}_{5}$, we define:

\begin{defn}[Bad Events for Samples]  \label{BadEventsSub}

Let graphon $w$, value $\alpha$, random variables $\{ U_i\}$, $\{ U_{i,j}\}$, and graphs $\GA{\alpha}$ and $H_{\alpha}$ be as defined above.  Moreover, let $R$, $\{ I_k\}_{k=0}^{4R-3}$ be as in (\ref{DefIntervals}). Finally, we let $S \subset V$ and $m=|S|$.

\be
\mathcal{A}_{1}' &= \{\GA{ \alpha}(S) \not= H_{\alpha}(S)\}\\
\mathcal{A}_{3}' &= \left\{ \min_{0 \leq k  \leq 4R-3} |\{ i\in S\,:\, U_i\in I_k\}|< \frac{m}{4R}  \right\} \cup
\{ |S \backslash S''| > \log(n)^{\cS}\} \\
%
\mathcal{A}_{4}' &= \left\{\min_{i,j \in S \, : \, |U_{i}-U_{j}| > \frac{1}{\log(n)}} \nu_{S}(i,j) <  \log(n)^3 \right\} \\
\mathcal{A}_{5}' &= \left\{  \max_{i,j \in S} |\{ k \, : \, U_{k} \in (U_{i},U_{j}) \}| \geq m |U_{i}-U_{j}| + 2 \log(n)^{3.5}  \right\}.\\
\ee
\end{defn}

We think of these events as functions of $S$, viewing the variables $\{ U_i\}$ and $\{ U_{i,j}\}$ that determine the graph $G$ as ``fixed."
We have seen that $\mathcal{A}^c$ holds \wep \, in Corollary \ref{IneqLatentCor}. We will see here that events $\mathcal{A}_{i}'$ are also all unlikely. The proofs are again simple, and are deferred to Appendix \ref{AppSecBadSmallSets}. We need slightly more detailed results (\textit{e.g.} that these events remain rare \textit{even conditional on} the occurrence of certain vertices), so we keep the lemma statements here.

\begin{lemma}\label{LemmaSampleEqual}
Let $S$ be a uniformly chosen subset of $V$, of size  $m = \log(n)^{\csketch}$. Then on the $\mathcal{F}$-measurable event $\mathcal{A}^{c}$,
\[
\P [ \mathcal{A}_1'\,|\,\mathcal{F}]=O( n^{-0.49}).
\]
Furthermore, for any fixed $p,q\in V(G)$ satisfying $(U_p,U_q)\not\in\bad{n}{\alpha}$, we have
\[
\frac{\P [ \mathcal{A}_1'\cap\{ p,q\in S\}\,|\,\mathcal{F}]}{\P [\{ p,q\in S\} | \mathcal{F}]}=O( n^{-0.49}),
\]
where again both probabilities are on $\mathcal{A}^{c}$.
\end{lemma}

For $0\leq k\leq 4R-3$, let $V_k=\{ i\,:\, U_i\in I_k\}$. 

\begin{lemma}\label{LemmaIndGraphConnected}
Let $S$ be a uniformly chosen subset of $V$, of size  $m = \log(n)^{\csketch}$, and let $S''$ be as defined in \eqref{DefS''}.
We have $\P [ \mathcal{A}_3'\,|\,\mathcal{F}]=n^{-\Omega(\log(n))}$ on the $\mathcal{F}$-measurable event $\mathcal{A}^{c}$. That is, \wep\ for all $0\leq k\leq 4R-3$, both $|S\cap V_k|\geq \frac{\log (n)^{\csketch}}{4R}$ and also $|S\setminus S''|\leq \log(n)^{\cS}$ on $\mathcal{A}^{c}$. Moreover, for any fixed $p,q\in V(G)$,
\[
\frac{\P [ \mathcal{A}_3'\cap\{ p,q\in S\}\,|\,\mathcal{F}]}{\P [\{ p,q\in S \} | \, \mathcal{F}]}=O( n^{-\Omega (\log (n)) }),
\]
where again both probabilites are on the event $\mathcal{A}^{c}$.
\end{lemma}

We see that a number of other bad events are ``almost" contained in $\mathcal{A}_3'$.

\begin{lemma}
\label{LemmaDegreeConnected}
Let $S$ be a uniformly chosen subset of $V$, of size  $m = \log(n)^{\csketch}$, and let $S''$ as defined in \eqref{DefS''}. For each $i\in S$, let $D_S(i)=|\{ j\in S: \GA{\alpha}(i,j)=1\}|$ be the number of neighbours of $i$ in $\GA{\alpha}(S)$. Then for all $n$ sufficiently large,
\[
\{ H_{\alpha} (S) \text{ is not connected}\}\, \cup \, \{ H_{\alpha} (S'') \text{ is not connected}\}\cup \left\{ \min_{i\in S} D_S(i) < \frac{m}{8R} \right\}\subset \mathcal{A}_3'.
\]

\end{lemma}

The next lemma shows that \wep, for all pairs of vertices $i,j$ that are sufficiently far apart, we can distinguish $i$ and $j$ since they have distinct neighbourhoods in $\GA{\alpha}$.

\begin{lemma} \label{LemmaCondProbGoodSub}
Let $S$ be a uniformly chosen subset of $V$, of size  $m = \log(n)^{\csketch}$. Then
\[
\P [\mathcal{A}_4'\,|\,\mathcal{F}]= n^{-\Omega (\log n)}
\]
on the $\mathcal{F}$-measurable event $\mathcal{A}^c$.
That is, \wep\ for all $i,j\in S$ so that $|U_{i}-U_{j}| > \displaystyle\frac{1}{\log (n)^{\cafo}}$, $ \nu_{S}(i,j) \geq \log(n)^3  $.
\end{lemma}

Similarly, it follows directly from $\mathcal{A}_5^c$ that the number of vertices between $i,j$ in a uniformly chosen subset $S'$ is roughly proportional to $|U_{i}-U_{j}|$:

\begin{lemma} \label{LemmaCompOrdersS}
On the $\mathcal{F}$-measurable event $\mathcal{A}^{c}$,
\[
\P [\mathcal{A}_5'\,|\,\mathcal{F}]= n^{-\Omega (\log n)}.
\]
That is, \wep\ for all $i,j\in S$, $|\{ k\,:\, U_k\in (U_i,U_j)\}|\leq  m |U_{i}-U_{j}| + 2 \log(n)^{3.5}$.
\end{lemma}

\subsection{Proof Sketch}\label{Sec:ProofSketch}

In the next sections, we will prove the correctness of all the Algorithms that are used to prove Theorem \ref{ThmMaxOrd}. The proofs involve four major ideas, which we outline here. We will assume that $\mathcal{A}^c$ holds, so $G$ and $\GA{\alpha}$ have good properties.

\begin{itemize}
 \item {\bf Most of the small samples can be correctly ordered.} With probability $1 - o(1)$, the random subgraphs of size $m=\log (n)^{\csketch}$ sampled in Algorithm \ref{AlgSketch} have certain good properties (\textit{e.g.}~they are connected, Robinsonian, and have many vertices with distinct neighbourhoods); when this occurs, Algorithm \ref{AlgSketch} will return an ordering that agrees with $\sigma_{\mathrm{true}}$ or $\eurt$ for all vertex pairs sufficiently far apart.


\item {\bf The ordered small samples can be aligned.} At this point, we have many subsamples that each \textit{individually} agree with either  $\sigma_{\mathrm{true}}$ or $\eurt$. The next step is to ``align" these subsamples by reversing some of them, so that either the \textit{entire collection} agrees with $\sigma_{\mathrm{true}}$, or the \textit{entire collection} agrees with $\eurt$. The main observation is that it is easy to align a \textit{particular pair} of samples if they share vertices with latent position close to either 0 or 1; it turns out to be possible to align all samples by greedily aligning pairs in this way.

\item {\bf Aligned sketches give a coarse ordering.} Once the small samples are aligned, say with $\sigma_{\mathrm{true}}$, then  a majority vote is used in Algorithm \ref{AlgSketch} to determine a {\sl sketch}, or coarse ordering, of all the vertices, represented by a comparison $F$. The number of small samples taken is such that \wep\ each pair of vertices is sampled many times, and in the majority of the samples containing this pair, the order agrees with $\sigma_{\mathrm{true}}$. Thus, pairs of vertices that are sufficiently far apart will be ordered according to $\sigma_{\mathrm{true}}$ in the vast majority of the samples in which they occur. 

\item {\bf Coarse ordering is refined using good witnesses.}
Finally, in Algorithm \ref{AlgUseSketch} , the comparison $F$ returned by Alg.~\ref{AlgSketch}, is refined. This refinement is based on the {\sl good witnesses} of the vertices. Since $\WA{\alpha}$ is diagonally increasing, and $\GA{\alpha}$ is very similar to a sample from $\WA{\alpha}$, we can use these good witnesses to infer the ordering of two vertices $s,t$. Namely, if $\sigma_{\rm true}(s)<\sigma_{\rm true} (t)$, i.e.~if $s$ comes before $t$ in the true ordering, then $t$ will have more neighbours than $s$ higher up in the ordering, while $s$ will have more neighbours lower in the true ordering. Thus, by considering the vertices that are neighbours of $s$ but not of $t$, or vice versa, and using the coarse ordering as a proxy for the true ordering, we can infer the correct ordering of $s$ and $t$.

\end{itemize}

We note that the organization of this description doesn't line up exactly with the organization of Algorithm \ref{AlgMerge}. This is somewhat unavoidable -- one step of the proof may apply to many different parts of the algorithm, and the length of a part of the algorithm does not correspond closely to its importance.

\subsection{Analysis of Algorithm \ref{AlgSubsamp}} \label{SecAnalysisAlgGlobal}

In Algorithm \ref{AlgSubsamp}, a relatively small subgraph $S$ of size $m = \log (n)^{\csketch}$ is sampled from $\GA{\alpha}$. Algorithm \ref{AlgSubsamp} applies a graph algorithm to determine whether the sample is a unit interval graph. If the algorithm succeeds, it will return an interval ordering. We prove in this section that, if no bad events occur, then with high probability the algorithm will succeed and give an ordering which agrees with one of the two correct line-embedded permutations of the vertices.

Throughout this section, we will assume that $\mathcal{A}^c$ holds, and study the properties of the subset $S$ on this event. We will not condition on $(\mathcal{A}_3'\cup\mathcal{A}_4'\cup\mathcal{A}_5')^c$, but we recall here that, by the lemmas in Section \ref{sec:PropertiesSamples}, this event holds \wep. Since this holds \wep, we will be able to use the following properties of $S$ when required in our proofs, incurring only a negligible penalty in our probability estimates each time: 
\be \label{Eq:PropertiesS}
&|S\setminus S''|\leq \log(n)^3 \text{, where }S'' \text{ as defined in \eqref{DefS''},}\\
&H_\alpha(S)\text{ and }H_\alpha(S'') \text{ are connected,}\\
&|\{ j\in S: \GA{\alpha}(i,j)=1\}|\geq \frac{m}{8R}\text{ for all }i\in S,\\
&\GA{\alpha}(S'')=H_\alpha (S'')\\
&\nu_S(i,j) \geq \log(n)^3 \text{ for all }i,j\in S\text{ such that } |U_i-U_j|>\frac{1}{\log(n)},\\
&|\{ k\,:\, U_i<U_k<U_j\}|< m|U_i-U_j| + 2\log (n)^{3.5} \text{ for all }i,j\in S.\\
\ee
Note that the second and third property above follow from the event $(\mathcal{A}_3')^c$, together with Lemma \ref{LemmaDegreeConnected}. The fourth property follows from our assumption that we are on $\mathcal{A}^c$: If $\mathcal{A}_1^c$ holds, then this implies that discrepancies between $\GA{\alpha}(S)$ and $H_\alpha (S)$ can only occur between vertices $i,j$ where $i\in \bad{n,j}{\alpha} $, and, by definition, $S''$ does not include any such pairs.

We do not assume that we are on $(\mathcal{A}_1')^c$, and do \textit{not} use this event freely in the following calculations. We do this because Lemma \ref{LemmaSampleEqual} only implies that  $(\mathcal{A}_1')^c$ occurs with probability going to 1, but does not show it holds \wep\ Algorithm \ref{AlgSubsamp} is called many times in Algorithm \ref{AlgSketch}, and it may indeed happen that  $\mathcal{A}_1'$ occurs for a small number of sketches.


We will show that the ordering returned by Algorithm \ref{AlgSubsamp} agrees with $\sigma_{\mathrm{true}}$ or $\eurt$ for all pairs of vertices that are sufficiently far apart, and hence we define:

\begin{defn}
\label{DefPrecision}
Given $S\subset V(G)$, a permutation $\sigma $ of $S$, and a permutation $\tau $ of $V$, we say that $\sigma $ \emph{agrees with $\tau$ to precision level $d$}, if, for all $i,j\in S$ so that $|U_i-U_j|\geq d $, $\sigma (i)<\sigma (j)$ if and  only if $\tau (i)<\tau (j)$.

\end{defn}

\begin{lemma}
\label{LemmaSubSampleCorrect}
Let $\GA{\alpha}$ be as defined in Equation \ref{EqThreshDef}, and assume $\mathcal{A}^c$ holds. Let $m= \log(n)^{\csketch}$, and let $S\sim Unif(\{ T\subset V\,:\, |T|=m\})$.  Let $\mathcal{B}_1'$ be the event that Algorithm
\ref{AlgSubsamp} with input $\GA{\alpha}$ and parameter $m$ succeeds, \emph{and furthermore} returns a total order $\sigma $ of the sampled set $S$ which agrees at precision level $\log (n)^{-1}$ with $\sigma_{\mathrm{true}}$ or with $\sigma_{eurt}$. Then
\[
\P [ \mathcal{B}_1'\, |\,\mathcal{F}]=1-O(n^{-0.49})
\]
on the $\mathcal{F}$-measurable event $\mathcal{A}^{c}$. Moreover, for $1\leq p,q\leq n$ so that $|U_p-U_q|\geq \log (n)^{-1}$,
\[
\frac{\P [ \mathcal{B}_1'\cap \{ p,q\in S\}|\mathcal{F}]}{\P [\{ p,q\in S\}\,|\,\mathcal{F}]}=1-O(n^{-0.49}),
\]
again on $\mathcal{A}^{c}$.
\end{lemma}

\begin{proof}
Algorithm \ref{AlgSubsamp} starts by sampling $S\sim Unif(\{ T\subset V\,:\, |T|=m\})$. By \eqref{Eq:PropertiesS}, we have that $\GA{\alpha} (S)=H_\alpha(S)$ is connected \wep, and thus $\GA{\alpha}(S)$ is a connected unit interval graph. On this event, Algorithm \ref{AlgSubsamp} does not fail, and returns an interval ordering $\sigma$ of $\GA{\alpha }(S)=H_\alpha (S)$. Since $H_{\alpha}(S)$ is a subgraph of the interval graph $H_{\alpha}$, the interval ordering of $H_\alpha (S)$ agrees with  $\sigma_{\mathrm{true}}$ or $\eurt$; assume \emph{wlog} that it agrees with $\sigma_{\mathrm{true}}$. This means that $\sigma $ agrees with $\sigma_{\mathrm{true}}$, except possibly for pairs of vertices with identical neighbourhoods in $\GA{\alpha} (S)$.

Applying \eqref{Eq:PropertiesS} again, \wep\ for all $ i,j\in S$ so that
 $|U_{i}-U_{j}| > \log (n)^{-1}$, we have that $ \nu_{S}(i,j) \geq \log(n)^{3} $. This implies that $i$ and $j$ have distinct neighbourhoods in $\GA{\alpha}(S) $ and thus also in $\GA{\alpha}$ (see discussion in Section \ref{Sec:UnitIntervalGraphs}). Consequently, for pairs $i, j \in S$ with $|U_{i}-U_{j}| > \log (n)^{-1}$, we have
 \[
 \{ \sigma (i)<\sigma (j) \} \leftrightarrow \{\sigma_{\mathrm{true}}(i)<\sigma_{\mathrm{true}}(j)\}.
  \]
  In other words, $\sigma$ agrees with $\sigma_{\mathrm{true}}$ at precision level $\log (n)^{-1}$, completing the proof.

\end{proof}

Algorithm \ref{AlgSubsamp} is called many times in Algorithm \ref{AlgSketch}, to produce a collection of sets and orderings $(\sigma^{(j)},S^{(j)})_{j=1}^{t}$. Since  $\mathcal{B}'_1$ may not occur \wep, some of the set-ordering pairs in this collection may not agree with $\sigma_{\mathrm{true}}$ or $\sigma_{eurt}$ at the required level of precision. 
We will need a minimal amount of agreement to prove that Algorithm \ref{AlgGlobal}, the alignment algorithm, works correctly based on \emph{all} samples. Therefore, we now define a weaker version of agreement between orderings, and show that this weaker property holds \wep.

\begin{defn}
\label{DefRoughlyAgrees}
Define
\be
\mathcal{L}=\{ i\in V\,:\, U_i<r_\alpha (\epsilon/2)\}\text{ and }\mathcal{R}=\{ i\in V\,:\, U_i>\ell_\alpha(1-\epsilon/2)\},
\ee
where $\epsilon>0$ is as in Assumption \ref{AssumptionGoodGraphon}.
For any permutation $\sigma $ of a subset $S\subset V$, let
\be
\mathcal{L} (\sigma, S) &=\{ i\in S\,:\, 1\leq \sigma (i)\leq \zeta\}, \text{ where } \zeta=4\log (n)^{\czeta}, \text{ and}\\
\mathcal{R} (\sigma ,S) &=\{ i\in S\,:\, |S| \geq \sigma (i) \geq  |S| - \zeta + 1\}.
\ee
We say that $\sigma $ \emph{roughly agrees} with $\sigma_{\mathrm{true}}$ if $\mathcal{L} (\sigma, S)\subset \mathcal{L}$ and $\mathcal{R} (\sigma, S)\subset \mathcal{R} $. Similarly, $\sigma $ \emph{roughly agrees} with $\eurt$ if $\mathcal{L} (\sigma, S)\subset \mathcal{R}$ and $\mathcal{R} (\sigma, S)\subset \mathcal{L} $.
\end{defn}

Note that, by part (4) of Assumption \ref{AssumptionGoodGraphon}, we have that $r_\alpha (\epsilon/2)<\ell_\alpha (1-\epsilon/2)$, and thus $\mathcal{L} \cap \mathcal{R}=\emptyset$. This implies that any permutation $\sigma$ can roughly agree with at most one correct permutation.

\begin{lemma}
\label{LemmaAllRoughlyAgree}
Let $\GA{\alpha}$ be as defined in Equation \ref{EqThreshDef}, and assume $\mathcal{A}^c$ holds. Let $m= \log(n)^{\csketch}$, and let $S\sim Unif(\{ T\subset V\,:\, |T|=m\})$.  Let $\sigma$ be the ordering returned by Algorithm
\ref{AlgSubsamp} if the algorithm succeeds, or let $\sigma $ be equal to $\sigma_{\mathrm{true}}$ restricted to $S$ otherwise. Then
\[
\frac{\P [\{\sigma \text{ roughly agrees with }\sigma_{\mathrm{true}}\text{ or  }\eurt\} \,|\,\mathcal{F}]}{\P [\{ \text{Algorithm \ref{AlgSubsamp} succeeds}\}] | \mathcal{F}}=1-n^{-\Omega (\log n)}
\]
on the $\mathcal{F}$-measurable event $\mathcal{A}^{c}$. That is, if Algorithm \ref{AlgSubsamp} succeeds, then \wep\ it returns a permutation that roughly agrees with $\sigma_{\mathrm{true}}$ or $\eurt$.
\end{lemma}

\begin{proof}

Consider  a run of Algorithm  \ref{AlgSubsamp} that succeeds; let $S$ be the returned subset,
and let $\sigma$ be the returned permutation (note that $\sigma$ represents an interval order of $\GA{\alpha}(S)$).
 Let $j\in S$ be so that $\sigma (j)=1$. By \eqref{Eq:PropertiesS}, \wep\  the neighbour set of $j$ has size
\[
|\{ q \in S :\GA{\alpha}(j,q)=1\}|\geq \log(n)^{\csketch}/(8R)>\zeta =|\mathcal{L}(S,\sigma )|.
\]
Since $\sigma$ is an interval ordering of $\GA{\alpha}(S)$, all neighbours of $j$ must form an interval $\mathcal{I}$ in $\sigma$, and  $\mathcal{L}(S,\sigma) \subset \mathcal{I}$. Since $\sigma$ is an interval ordering, this also implies that all vertices in $\mathcal{I}$ (and thus $\mathcal{L}(S,\sigma)$) are mutually adjacent in $\GA{\alpha}$.

We have shown that \wep \ $\mathcal{L}(S,\sigma)$ is a clique, but have not yet shown that even \textit{a single} element $s \in \mathcal{L}(S,\sigma)$ has a small associated latent variable $U_{s}$. Before starting the argument that this is true, we give some basic facts about $S''$. By \eqref{Eq:PropertiesS}, \wep  \ we have both that $\GA{\alpha}(S'')=H_{\alpha}(S'')$, and that $H_\alpha (S'')$ is connected. Since $H_\alpha$ is a unit interval graph, so is $H_\alpha(S'')$.  Since $\sigma$ is an interval order of $\GA{\alpha}(S)$, $\sigma|_{S''}$ is an interval order of $\GA{\alpha}(S'')=H_{\alpha}(S'')$.
Therefore, $\sigma|_{S''}$ on $S''$ must agree with $\sigma_{\mathrm{true}}$ or $\eurt$. We assume \emph{wlog} that it agrees with $\sigma_{\mathrm{true}}$.

The next step is to check that $S''$  has some intersection with the collection
\be \label{def:Lhat}
\hat{\mathcal{L}} = \left\{ s \in S \, : \, U_{s} \leq \frac{1}{\log(n)} \right\}
\ee
of vertices with ``small" latent variables. From \eqref{Eq:PropertiesS}, we have \wep \, the inequalities
\be\label{eq:Lhat}
|\hat{\mathcal{L}}|\geq \frac{m}{\log (n)}- 2\log (n)^{3.5}, \quad |S\setminus S''|\leq \log(n)^3,
\ee
so $|\hat{\mathcal{L}}\cap S''|\geq |\hat{\mathcal{L}}|- |S \setminus S''| >0$ for large enough $n$.

Having shown that $\hat{\mathcal{L}}\cap S'' \neq \emptyset$, we consider  $s\in \hat{\mathcal{L}}\cap S''$ and show that it is in $\mathcal{L}(S,\sigma)$. For any $k \in S$ with $\sigma (k)<\sigma (s)$, one of the following must hold:
\begin{enumerate}
\item $\sigma_{\mathrm{true}}(k)<\sigma_{\mathrm{true}}(s)$. In this case $U_k<U_s$, so $k\in \hat{\mathcal{L}}$. From \eqref{eq:Lhat} above, we have that $|\hat{\mathcal{L}}|\leq \log(n)^4 + 2\log (n)^{3.5}$.
\item $\sigma_{\mathrm{true}}(k)>\sigma_{\mathrm{true}}(s)$. This means that $s,k$ are incorrectly ordered by $\sigma$. Since $\sigma|_{S''}$ agrees with $\sigma_{\mathrm{true}}$, we must be in one of the following subcases:
\begin{enumerate}
\item $k\in S\setminus S''$. As argued above, $|S \setminus S''| \leq \log(n)^{3}$.
\item $s,k$ have the same neighbourhood in $H_\alpha (S'')$, and thus $\nu_S(s,k)\leq |S\setminus S''|\leq \log(n)^3$. From \eqref{Eq:PropertiesS}, this implies that $U_s<U_k\leq U_s+\log (n)^{-1}$. Again from \eqref{Eq:PropertiesS}, there are at most $m\log(n)^{-1}+2\log (n)^{3.5}\leq 2\log(n)^4$ such vertices.
\end{enumerate}
\end{enumerate}
Combining these three cases, we see that there are at most $(\log(n)^4 + 2\log (n)^{3.5})+\log (n)^3+2\log (n)^4$ vertices $k$ with $\sigma (k)<\sigma (s)$. For large $n$, this amount is smaller than $4\log (n)^4=\zeta$, and so $\sigma (s)\leq \zeta$. This implies $s\in \mathcal{L}(S,\sigma)$, and so we have shown that
\be \label{ContHatLSpContLS}
\emptyset \neq \hat{\mathcal{L}}\cap S''\subseteq \mathcal{L}(S,\sigma).
\ee

We are now ready to show that $\mathcal{L}(\sigma,S)\subseteq \mathcal{L}$. For $n$ large enough, $\log(n)^{-1}< r_\alpha(\epsilon/2)$, and thus, by definition, $\hat{\mathcal{L}}\subset\mathcal{L}$.
To complete the argument, let $s \in ( \hat{\mathcal{L}}\cap S'')\subset \mathcal{L}(\sigma,S) $, and consider any other vertex $t \in \mathcal{L}(\sigma,S)$. As argued above, $\mathcal{L}(\sigma,S)$ is a clique in $\GA{\alpha}$, so we know that $s,t$ are connected in that graph. 
Since $s\in S''$ and $t\in S$, we have that $t\not\in\bad{n,s}{\alpha}$. Therefore, on $\mathcal{A}_1$, $H_\alpha (s,t)=\GA{\alpha}(s,t)=1$, so
\[ \label{IneqMainIneqAnnoying28}
U_{t} \leq r_\alpha (U_{s}) \leq r_\alpha \left( \frac{1}{\log(n)} \right) < r_\alpha \left(\frac{\epsilon}{2}\right),
\]  where the last inequality holds for all $n$ sufficiently large. Therefore, $t\in \mathcal{L}$.

Collecting our assumptions throughout the proof of the theorem, we have shown that  \wep \  $\mathcal{L}(\sigma,S)\subset \mathcal{L}$ on $\mathcal{A}^c$. An analogous argument shows that, under the same assumptions, $\mathcal{R}(\sigma,S)\subset \mathcal{R}$. Thus, the  result follows from Corollary \ref{IneqLatentCor}.
\end{proof}

\subsection{Analysis of Algorithm \ref{AlgGlobal}}

In the previous section we showed that \wep\ the output $(\sigma,S)$ from Algorithm \ref{AlgSubsamp} roughly agrees with either $\sigma_{\mathrm{true}}$ or its total reversal, $\eurt$. In this section we show that the function $a$ which is the output of Algorithm \ref{AlgGlobal} correctly identifies whether the alignment is with $\sigma_{\mathrm{true}}$ or with $\eurt$.

\begin{lemma}
\label{LemmaGlobalAlignmentCorrect}
Let Assumptions \ref{AssumptionGoodGraphon} hold.
Let $t\geq (n \log(n))^2$, and let $(\sigma^{(j)},S^{(j)})_{j=1}^{t}$ be the result of repeated calls to Algorithm \ref{AlgSubsamp}, conditioned on non-failure. Define the function $a^*:\{ 1,\dots ,t\}\rightarrow \{ -1,1\}$ by
\be\label{EqTrueAlignment}
a^*(j)=\mathbf{1}_{\{\sigma^{(j)}\text{ roughly agrees with }\sigma_{\mathrm{true}}\}} -\mathbf{1}_{\{\sigma^{(j)}\text{ roughly agrees with }\eurt\} }.
\ee
Suppose Algorithm \ref{AlgGlobal} is called with input $(\sigma^{(j)},S^{(j)})_{j=1}^{t}$, and parameter $
\zeta = 4\log (n)^{\czeta}$. Then on the $\mathcal{F}$-measurable event $\mathcal{A}^c$,
\be \label{EqCorrectAlignment}
\P [ \{ a=a^*\} \cup \{ a=-a^*\} | \mathcal{F}]=1- n^{-\Omega(\log(n))}.
\ee

\end{lemma}

\begin{proof}
Recall that, for $1\leq j\leq t$, $\mathcal{L}^{(j)}$ and $\mathcal{R}^{(j)}$ as computed in step 4 of Algorithm \ref{AlgGlobal} correspond to the sets $\mathcal{L} (S^{(j)}, \sigma^{(j)})$ as in Definition \ref{DefRoughlyAgrees}. By Lemma \ref{LemmaAllRoughlyAgree}, the condition
\be
\label{EqConditionSketches}
\bigcap_{j=1}^t \{\sigma^{(j)}\text{ roughly agrees with }\sigma_{\mathrm{true}}\text{ or  }\eurt\}
\ee
holds \wep \, on $\mathcal{A}^{c}$.

By Definition \eqref{DefRoughlyAgrees}, if $\sigma^{(j)}$ roughly agrees with $\sigma_{\mathrm{true}}$, then $\mathcal{L}^{(j)}\subset \mathcal{L}=\{ i: U_i<r_\alpha (\epsilon/2)\}$ and $\mathcal{R}^{(j)}\subset \mathcal{R}=\{ i: U_i>\ell(\epsilon/2)\}$. Similarly, if $\sigma^{(j)}$ roughly agrees with $\eurt $ then $\mathcal{L}^{(j)}\subset \mathcal{R}$ and $\mathcal{R}^{(j)}\subset \mathcal{L}$ (as argued, by Assumption \ref{AssumptionGoodGraphon}, $\sigma^{(j)}$ can roughly agree with at most one of $\sigma_{\mathrm{true}}$, $\eurt$).

Thus, if $\mathcal{L}^{(j)}\cap \mathcal{L}^{(k)}\not=\emptyset $ or $\mathcal{R}^{(j)}\cap\mathcal{R}^{(k)}\not=\emptyset$, and consequently $H(j,k)$ is set equal to 1 in step 8 of Algorithm \ref{AlgGlobal}, then $\sigma^{(j)}$ and $\sigma^{(k)}$ agree with the same correct permutation ($\{\sigma_{\mathrm{true}}$ or $\eurt\}$), and thus $a^*(j)a^*(k)=1$. Similarly, if $\mathcal{L}^{(j)}\cap \mathcal{R}^{(k)}\not=\emptyset $ or $\mathcal{R}^{(j)}\cap \mathcal{L}^{(k)}\not=\emptyset$, then $a^*(j)a^*(k)=-1$. Thus, condition (\ref{EqConditionSketches}) implies that the following holds for the function $H$ computed in Algorithm \ref{AlgGlobal} \wep:
\be
\label{EqHisgood}
\bigcap_{1\leq j,k,\leq t} \big(\{ H(j,k)=0 \} \cup \{ H(j,k)=a^*(j)a^*(k)\}\big).
\ee
That is, for each pair of samples $j,k$ for which $H(j,k)$ is non-zero, $H$ correctly identifies whether $j$ and $k$ are aligned with the same true ordering or with opposite orderings.

Now consider the graph $\mathcal{G}_H$ with vertex set $V_H=\{ 1,2,\dots ,t\}$ and $\mathcal{G}_H(j,k)=|H(j,k)|$. That is, $j$ and $k$ are adjacent in $\mathcal{G}_H$ if sketch $j$ and $k$ could be aligned by the algorithm. In light of \eqref{EqHisgood}, to prove the result, it is enough to prove that $\mathcal{G}_{H}$ is connected \wep.

As an extension of \eqref{def:Lhat} in Lemma \ref{LemmaAllRoughlyAgree},  define
\be
S''^{(j)}&=\{ i \in S^{(j)} \, : \, \forall j \in S^{(j)}, \, i \notin  \bad{n,j}{\alpha} \},\text{ and}\\
\hat{\mathcal{L}}^{(j)} &= \left\{ s \in S^{(j)} \, : \, U_{s} \leq \frac{1}{\log(n)} \right\}\text{ and }\hat{\mathcal{R}}^{(j)} = \left\{ s \in S^{(j)} \, : \, U_{s} \geq 1- \frac{1}{\log(n)} \right\}.
\ee

By \eqref{ContHatLSpContLS} in the proof of Lemma \ref{LemmaAllRoughlyAgree}, we have that \wep\  both 
\be\label{eqsubset}
\emptyset \neq \hat{\mathcal{L}}^{(j)}\cap S''^{(j)}\text{ and } \hat{\mathcal{L}}^{(j)}\cap S''^{(j)} \subset \mathcal{L}^{(j)}.
\ee  
Together with an analogous argument for $\hat{\mathcal{R}}$, we have \wep\ that either

\begin{enumerate}
\item  $P^{(j)} \equiv \mathcal{L}^{(j)} \cap \hat{\mathcal{L}}^{(j)} \neq \emptyset$ and $Q^{(j)} \equiv \mathcal{\mathcal{R}}^{(j)} \cap \hat{\mathcal{R}}^{(j)} \neq \emptyset$, or
\item $P^{(j)} \equiv \mathcal{R}^{(j)} \cap \hat{\mathcal{L}}^{(j)} \neq \emptyset$ and $Q^{(j)} \equiv \mathcal{L}^{(j)} \cap \hat{\mathcal{R}}^{(j)} \neq \emptyset$.
\end{enumerate}

Note that in the above we give a two-case definition for the sets $P^{(j)}, Q^{(j)}$; we clarify that we always choose the choice that makes both non-empty when such a choice is possible, choosing arbitrarily otherwise. By Lemma \ref{LemmaAllRoughlyAgree}, \wep\ there will be a unique choice making both non-empty.

There is an edge $(i,j)$ in $\mathcal{G}_{H}$ as long as $(P^{(i)} \cup Q^{(i)}) \cap (P^{(j)} \cup Q^{(j)}) \neq \emptyset$.
We make the following claim: \wep\ for any $a,b \in V$ so that $U_a,U_b<\log (n)^{-1}$, there exists $1\leq \ell\leq k$ so that $\{ a,b\}\subset  P^{(\ell)}\cup Q^{(\ell)}$. Before proving the claim, we show that this will complete the proof. For any $1\leq i,j\leq k$, take $a\in 
(P^{(i)}\cup Q^{(i)})$ and $b \in  
(P^{(j)}\cup Q^{(j)})$. Let $\ell$ be so that $\{ a,b\}\subset  P^{(\ell)}\cup Q^{(\ell)}$. Then there are edges $(i,\ell) $ and $(\ell ,j)$ in $\mathcal{G}_H$. This shows that $\mathcal{G}_H$ is connected.

It remains to show that our claim is true. Fix $a,b \in V$ so that $U_a,U_b<\log (n)^{-1}$. Let $\beta=\inf\{ \WA{}(0,y):0\leq y\leq \epsilon\}$. By part (3) of Assumption \ref{AssumptionGoodGraphon}, 
$\beta >\alpha$. This implies that $\bad{n}{\alpha}\cap [0,\epsilon]=\emptyset$ for all $n$ sufficiently large. 
It follows that $(U_a,U_b)\not \in \bad{n}{\alpha}$, and thus $a\not\in\bad{n,b}{\alpha}$ and $b\not\in\bad{n,a}{\alpha}$. 

To show that $a,b\in P^{(\ell)}\cup Q^{(\ell)}$, by \ref{eqsubset} it suffices to show that $a,b\in S''^{(\ell)}$.

We have that
\[
\P [\{ a,b\}\subset S]\geq \left(\frac{m}{n}\right)^2.
\]
Moreover, on $\mathcal{A}^c$, $|\bad{n,a}{\alpha}\cup\bad{n,b}{\alpha}|\leq 2\sqrt{n}\log(n)^2$, and thus
\[
\P [\{ a,b\}\in S''\,|\,\{ a,b\}\subset S]\geq 1-2(m-2)\frac{\log(n)^2}{\sqrt{n}}\geq \frac12,
\]
on $\mathcal{A}^c$, and for $n$ large enough. 
Thus the expected number of sketches for which $\{ a,b\}\subset S^{(j)}$ is at least $t(m^2/(2n^2))\geq \log(n)^{11}$.  Since the sets $\{S^{(j)}\}$ are independent, by the Chernoff bound we obtain that, \wep\ at least one of the sketches contains both $a$ and $b$. This completes the proof.



\end{proof}

\subsection{Analysis of Algorithm \ref{AlgSketch}}

We will give a bound on the output of Algorithm \ref{AlgSketch}. Define the sets
\be
Q_1 &= \{ (p,q) \in \tilde{V} \, : \, |U_{p} - U_{q}| > \log(n)^{-\cub}\}\\
Q_2 &= \{ (p,q) \in \tilde{V} \, :\, (U_p,U_q)\not\in \bad{n}{\alpha }\}.\\
\ee

For a comparison $F$, define the event:
\be
\label{EventGoodComparisonPair}
\mathcal{B}_{2}'(F) &= \{ \forall (p,q)\in Q_1\cap Q_2,\,  F(p,q) = F_{\mathrm{true}}(p,q)\}
\ee
that
$F$ gives the correct comparison for all pairs of sufficiently-distant vertices in $Q_2$. We have:

\begin{lemma} [Sketch Correctness] \label{LemmaSketchCorrectness}
Let $\GA{\alpha}$ be as defined in Equation \ref{EqThreshDef}, and
let $F$ be the output of Algorithm \ref{AlgSketch}, on input $\GA{\alpha}$ and parameters $m=\log (n)^{\csketch}$ and $t=(n \log(n))^2$. Then on the $\mathcal{F}$-measurable event $\mathcal{A}^{c}$,
\[
\P[ \mathcal{B}_{2}'(F) \cup \mathcal{B}_2'(-F)\,|\,\mathcal{F} ]=1- n^{-\Omega(\log(n))}.
\]
\end{lemma}

\begin{proof}

Let $\{(S^{(j)},\sigma_{j})\}_{j=1}^t $ be the sequence of subsets and preorders appearing immediately after the alignment correction in Step 12 of Algorithm \ref{AlgSketch}. Let $\mathcal{C}$ be as in Algorithm \ref{AlgSubsamp} at the end of its run. Let $c_1=\log(n)^7$, and let $\mathcal{N}:\tilde{V}\rightarrow \mathbb{N}$ be the function counting the number of times vertices $s$ and $t$ occurred together in a sample, that is, $\mathcal{N}(s,t)=|\{ j: s,t\in S^{(j)}\}|$.

Next, consider a pair $(p,q)\in Q_2$, so $(U_p,U_q)\not\in \bad{n}{\alpha}$, and let $S$ be a uniformly-chosen size-$m$ subset of $\{1,2,\ldots,n\}$, as sampled by Algorithm \ref{AlgSubsamp} when it is called for the $i$'th time in Step 3 of Alg.~\ref{AlgSketch}. If $p,q\in S$, then $p,q\in S^{(i)}$ if Alg.\ \ref{AlgSubsamp} succeeds. By Lemma \ref{LemmaSubSampleCorrect}, on the event $\mathcal{A}^{c}$ we have

\be
\P [  p,q\in S^{(i)}\,|\,\mathcal{F}] \geq (1-O(n^{-0.49}))\P [ p,q\in S] \geq 0.9 (m/n)^2 .
\ee

Let $e^{(i)}(p,q) = \textbf{1}_{\{p,q \in S^{(i)}\}}$. Then on $\mathcal{A}^{c}$,
\[
\E[ \mathcal{N}(p,q) | \mathcal{F}] =\E[\sum_{i=1}^t e^{(i)}(p,q) | \mathcal{F}]\geq 0.9 (m/n)^2 t=0.9\log(n)^{12} \geq 2c_1.
\]

By the Chernoff bound, on $\mathcal{A}^{c}$ we have
\be\label{IneqCalcManyHits}
\P [ \mathcal{N}(p,q)< c_1 \,|\,\mathcal{F}] =n^{-\Omega (\log(n))}.
\ee
Let $\mathcal{B}_3'$ be the event that, when Algorithm \ref{AlgGlobal} is called with input $(S^{(i)},\sigma^{(i)})_{i=1}^t$ in step 10 of Algorithm \ref{AlgSketch},
the returned function $a$ agrees with one of the true functions $a^*$ or $-a^*$ as defined in the statement of Lemma \ref{LemmaGlobalAlignmentCorrect}. By Lemma \ref{LemmaGlobalAlignmentCorrect},
\be \label{IneqB3HoldsRef}
\P [\mathcal{B}_3'\,|\,\mathcal{F}]=1-n^{-\Omega(\log(n))}
\ee on $\mathcal{A}^{c}$.

Considering the event $\mathcal{B}'_3$ holds, we assume without loss of generality that $a=a^*$ for all arguments. This implies that, after the realignment in step 10 of Algorithm \ref{AlgSketch}, $\sigma^{(i)} $ roughly agrees with $\sigma_{\mathrm{true}}$ for each $i=1, \dots, t$.  Recall the event $\mathcal{B}_1'$ as defined in Lemma \ref{LemmaSubSampleCorrect}; we define $\mathcal{B}_1'(i)$ to be the event that $\mathcal{B}_{1}'$ holds for the call to Algorithm \ref{AlgSubsamp} that produced the pair $(S^{(i)},\sigma^{(i)})$.

Next, fix $(p,q)\in Q_1$, so $|U_p-U_q|\geq \log(n)^{-\cub}$, and assume without loss of generality that $\sigma_{\mathrm{true}}(p) < \sigma_{\mathrm{true}}(q)$. On the event $\mathcal{B}_{1}'(i)$,  $\sigma^{(i)}$ agrees at precision level $\log(n)^{-1}$ with $\sigma_{\mathrm{true}}$. Defining  $f^{(i)}(u,v) = e^{(i)}(p,q) \textbf{1}_{\mathcal{B}'_1(i)}$ and applying Lemma \ref{LemmaSubSampleCorrect}, we have for large enough $n$,
\be
\P[f^{(i)}(u,v) =1 | \mathcal{F}] \geq 0.9\P[e^{(i)}(u,v) =1 | \mathcal{F}]
\ee
on $\mathcal{A}^c$. Thus, $\E [\mathcal{C}(p,q) | \mathcal{F}]\geq 0.9\mathcal{N}(p,q)$ on $\mathcal{A}^{c}$ for all $n$ sufficiently large. Applying the Chernoff bound
and then \eqref{IneqB3HoldsRef},
\be \label{IneqCalcGoodAlignment}
\P[ \mathcal{C}(p,q) > 0 | \mathcal{F}] \geq \P[  \{ \mathcal{C}(p,q) < 0\} \cup (\mathcal{B}_{3}')^c | \mathcal{F}] - \P[\mathcal{B}_{3}^{c} | \mathcal{F}]= 1- n^{-\Omega(\log(n))}
\ee
on $\mathcal{A}^{c}$.
Combining Inequalities \eqref{IneqCalcManyHits} and \eqref{IneqCalcGoodAlignment} and a union bound,
\be \label{IneqCorrectCompSketch1}
\P \left[ \bigcup_{(p,q) \in Q_1\cap Q_2,\, U_p<U_q} ( \{ \mathcal{C}(p,q) \geq  0\}) | \mathcal{F} \right]  =  n^{-\Omega(\log(n))}
\ee
on $\mathcal{A}^{c}$. This event is exactly the complement of $\mathcal{B}_{2}'(F)$, so this completes the proof.
\end{proof}

\subsection{Analysis of Algorithms \ref{AlgUseSketch} and \ref{AlgMerge} and proof of Theorem \ref{ThmMaxOrd}}

Define the set
\be
Q_{3} &= \left\{ (i,j) \in \tilde{V} \, : \, |U_{i} - U_{j}| > \frac{\log(n)^{4}}{\sqrt{n}} \right\} \\
\ee

The following lemma shows that Algorithm \ref{AlgUseSketch}, as called by Alg.~\ref{AlgMerge} and run with the correct input and parameters, gives as output a comparison function that correctly orders all pairs in $Q_3$, i.e., all pairs that are sufficiently far apart. 

\begin{lemma} [Correct refinement]
\label{LemmaLocalRefinementCorrect}
Let $w$ be a graphon and $\alpha\in (0,1)$ be so that Assumptions \ref{AssumptionGoodGraphon} hold, and let $G\sim w$ be a graph of size $n$.  Consider a run of Algorithm \ref{AlgMerge} with parameters as in \eqref{def:GoodParameters}
on input $G,\alpha$; let $F'$ be as it appears on line 3 of Algorithm \ref{AlgMerge}.  Then
\be
\bigcup_{(u,v) \in Q_3} \{ F'(u,v) \neq F_{\mathrm{true}}\}\subset \mathcal{A} \cup (\mathcal{B}_{2}(F)')^{c}.
\ee
\end{lemma}

\begin{proof}
We assume now that $\mathcal{A}^{c}$ and $\mathcal{B}_{2}'(F)$ both hold. Fix $p,q\in V$, and assume without loss of generality that $U_p<U_q$, so $F_{\rm true}(p,q)=1$.
Let $D_{p,q}$ and $D_{q,p}$ be as computed in  Alg.\ \ref{AlgUseSketch}, namely
\be
D_{p,q}&=|\{ x\in N(p)\setminus N(q)\,:\, F(x,q)=1\}|-|\{ x\in N(p)\setminus N(q)\,:\, F(x,q)=-1\}|\\
D_{q,p}&=|\{ x\in N(q)\setminus N(p)\,:\, F(x,p)=1\}|-|\{ x\in N(q)\setminus N(p)\,:\, F(x,p)=-1\}|.
\ee

Let $B=\bad{n,p}{\alpha}\cup \bad{n,q}{\alpha}$ and $c_{2} = 5 \sqrt{n} \log(n)^{2}$. Since $\mathcal{A}^c_2$ holds, $|B|\leq 2\sqrt{n}\log(n)^{\cat}< c_2/2$.

We next wish to show that $F'(p,q) = 1$ for all pairs $p,q$ with $\max(|N(p)\setminus N(q)|, |N(q)\setminus N(p)|) > 2 c_{2}$. We begin with the case $|N(p)\setminus N(q)|> 2c_2$. In this case, $(N(p)\setminus N(q))\setminus B \neq \emptyset$. For each $v\in (N(p)\setminus N(q))\setminus B$, we have that $\GA{\alpha}(q,v)=0$, and, since $v\not\in \bad{n,q}{\alpha}$  and we assume $\mathcal{A}^c$ holds, $H_\alpha(q,v)=0$.  Thus, by Conditions \ref{AssumptionGoodGraphon}, $|U_q-U_v|\geq \epsilon \geq \log(n)^{-1}$. Therefore, since $\mathcal{B}_2'(F)$ holds, we have that $F(v,q)=F_{\rm true}(v,q)$.

By a similar argument, $H_\alpha(p,v)=\GA{\alpha}(p,v)=1$. So $v$ is a neighbour of $p$ but not of $q$ in $H_\alpha$. Since $H_\alpha$ is diagonally increasing 
, this implies that
$U_v<U_q$, and so $\sigma_{\rm{true}}(v)<\sigma_{\rm{true}}(q)$, so $F_{\rm true}(v,q)=1$.

In our first case, we have just shown that $F(v,q)=1$ for all $v\in (N(p)\setminus N(q))\setminus B$ satisfying $F_{\mathrm{true}}(v,q)=1$. Moreover, by our assumption  $|N(p)\setminus N(q)| > 2 c_{2} > \frac{1}{2} c_{2} \geq |B|$,
\be \label{IneqSetMinusBStuff}
&|\{ x\in N(p)\setminus N(q)\,:\, F(x,q)=1\}|\geq |(N(p)\setminus N(q))\setminus B| > 2c_2-|B|, \text{ and}\\
&|\{ x\in N(p)\setminus N(q)\,:\, F(x,q)=-1\}|\leq |B| \leq \frac{1}{2} c_{2}.
\ee
Therefore in this case, $D_{p,q} > c_2$ and so $F'(p,q)=1$.

By essentially the same argument, in the case $|N(q)\setminus N(p)|>2c_2$, we have $D_{q,p}<-c_2$. Combining these two cases, we have that
\be \label{IneqFPNMinN}
F'(p,q)=F_{\mathrm{true}}(p,q),\,  \forall p,q \in V \,\text{s.t.}\,\max(|N(p)\setminus N(q)|, |N(q)\setminus N(p)|) > 2 c_{2}.
\ee

Now suppose $(p,q)\in Q_3$, so $|U_{p} - U_{q}| > \frac{\log(n)^{4}}{\sqrt{n}}$.  Since $\mathcal{A}_4^c$ holds, for all large enough $n$
\[
|N(p) \setminus N(q)|+|N(q) \setminus N(p)|=\nu (p,q)\geq \sqrt{n}\log(n)^{\cafti} > 20\sqrt{n}\log(n)^2=4c_2,
\]
and thus $ \max(|N(p)\setminus N(q)|, |N(q)\setminus N(p)|) > 2 c_{2}$. Applying Inequality \eqref{IneqFPNMinN}, this implies
\be \label{IneqRefinementQuasi}
\cup_{(p, q) \in Q_{3}} \{ F'(p,q) \neq F_{\mathrm{true}}(p,q) \}\subset \mathcal{A} \cup (\mathcal{B}_{2}'(F))^{c},
\ee
completing the proof.

\end{proof}

Finally, we check that $\sigma_{F'}$ is never much less accurate than $F'$:

\begin{lemma}\label{LemmaCyclesRemove}
Let $F$ be any comparison that agrees with a permutation $\sigma_{\mathrm{true}}$ at precision level $d$. Then $\sigma_{F'}$ agrees with $\sigma_{\mathrm{true}}$ at precision level $2d$.
\end{lemma}

\begin{proof}
Assume \emph{wlog} that $\sigma_{\mathrm{true}}(i) = i$ for all $i$. Then for $j > i + 2d$, we have

\be
\sigma_{F}(j) - \sigma_{F}(i) &= \sum_{k =1}^{n} (F(j,k) - F(i,k)) \\
&\geq 2 (j-i) -  | \{ k \, : \, \min(|i-k|, |j-k|) \leq d \}|  \\
 &\geq 2 (j-i) -  4 d >0. \\
\ee

\end{proof}

The proof of Theorem \ref{ThmMaxOrd} now follows directly from the lemmas.
Let $\sigma$ be the result returned by Algorithm \ref{AlgMerge}, and set $Q_3$ as above, and
\be
Q_{4} &= \{ (i,j) \in \tilde{V} \, : \, |\sigma_{\mathrm{true}}(i) -\sigma_{\mathrm{true}}(j)|   > 2 \sqrt{n} \log(n)^{\cfin}\}.
\ee
Define:
\be
\mathcal{B}_{4}' &= \{ \forall \, i,j \in Q_{3}, \, \{ \sigma(i) < \sigma(j) \} \longleftrightarrow  \{ \sigma_{\mathrm{true}}(i) < \sigma_{\mathrm{true}}(j) \} \}\\
\mathcal{B}_{5}' &= \{ \forall \, i,j \in Q_{4}, \, \{ \sigma(i) < \sigma(j) \} \longleftrightarrow  \{ \sigma_{\mathrm{true}}(i) < \sigma_{\mathrm{true}}(j) \} \}\\
\ee

to be the event that $\sigma$ is correct for all vertices that are sufficiently distant. We have:

\begin{lemma} [Merge Correctness] \label{LemmaMergeCorrectness}
Let $w$ be a graphon and $\alpha\in (0,1)$ be so that Assumptions \ref{AssumptionGoodGraphon} hold, and let $G\sim (n,w)$. Let $\sigma$ be the output of Algorithm \ref{AlgMerge}, run with parameters as in \eqref{def:GoodParameters}
on input $G, \alpha$. Then

\be
\P[\mathcal{B}_{4}' \cap \mathcal{B}_{5}'] = 1 - n^{-\Omega(\log(n))}.
\ee

\end{lemma}

\begin{proof}


Let $F$ be the output of Alg.\ \ref{AlgSketch} as called in line 2 of Alg.\ \ref{AlgMerge}.  Let $F'$ be the  output of Alg.\ \ref{AlgUseSketch} when called in line 3 of Alg.\ \ref{AlgMerge}. By Lemma \ref{LemmaLocalRefinementCorrect}, if $\mathcal{B}_{2}'(F)$ and $\mathcal{A}^{c}$ both hold, then $F'$ agrees with the true ordering on $Q_{3}$. Applying Lemma \ref{LemmaCyclesRemove}, this implies

\[
(\mathcal{B}_4')^c\subset \mathcal{A}\cup \left(\mathcal{B}_{2}'(F)\right)^c.
\]

Applying Corollary \ref{IneqLatentCor} and Lemma \ref{LemmaSketchCorrectness} to this inclusion gives
\[ \label{IneqGoodOrderQ3}
\P [(\mathcal{B}'_4)^c]=n^{-\Omega(\log(n))}.
\]

We have shown that all pairs in $Q_{3}$ are ordered correctly \wep; now we show that this implies all pairs in $Q_{4}$ are also ordered correctly \wep. Suppose that $(p,q)\in Q_4$ and $U_p<U_q$. By definition of $Q_4$,
$$ |\{ k \, : \, U_{k} \in (U_{p},U_{q}) \}| = \sigma_{\mathrm{true}}(q)-\sigma_{\mathrm{true}}(p) - 1 \geq \sqrt{n} \log(n)^{\cfin} .$$
If $\mathcal{A}_5^c$ holds, then
$ |\{ k \, : \, U_{k} \in (U_{p},U_{q}) \}|\leq n |U_{q}-U_{p}| + \sqrt{n} \log(n) $. Putting these two inequalities together, on $\mathcal{A}^{c}$ we have
\[
|U_q-U_p|\geq \frac{\sqrt{n} \log(n)^{\cfin}-\sqrt{n}\log(n)}{n}\geq \frac{2\log(n)^{4}}{\sqrt{n}}.
\]
Therefore, by Lemmas \ref{LemmaLocalRefinementCorrect} and \ref{LemmaCyclesRemove}, each pair $(p,q)$ is ordered correctly on $\mathcal{A}$, and so by Inequality \eqref{IneqGoodOrderQ3} and Corollary \ref{IneqLatentCor},
\be
\P[\mathcal{B}'_{5} ] \geq \P[\mathcal{B}'_{4} \cup \mathcal{A}^{c}] - \P[\mathcal{A}]= 1 - n^{-\Omega(\log(n))}.
\ee
This completes the proof.

\end{proof}

Theorem \ref{ThmMaxOrd} now follows immediately from Lemma \ref{LemmaMergeCorrectness}.

\section{Iterative Error Reduction} \label{SecErrorRooting}

In this section, we discuss and analyze Algorithm \ref{AlgIterationHolder}, and prove Theorem \ref{ThmReconMainRes2}. Throughout this section, we fix a graphon $w$ that satisfies Assumptions \ref{AssumptionSharpBoundary} and  \ref{AssumptionGoodGraphon}, and use the constants introduced in those assumptions. We assume that the input graph $G$ being analyzed is sampled from $G \sim w$ and we fix the associated notation from this choice as in the start of Section \ref{SubsecInitNot}; for example, we denote by $\{U_{i}\}$, $\{U_{ij}\}$ the random variables used to construct $G$ in representation \eqref{EqGraphonDef}. Finally, we let other notation follow as in  Algorithm \ref{AlgIterationHolder}; for example, we let $\{B_{j}\}$ be the Bernoulli random variables introduced in step 1, and sets $V_{i}$ be the sets constructed in steps 2 and 5. Let $G_i=G[V_i]$. We observe here that, conditional on $\{ B_{j}\}$, we have $G_i \sim w$ is a random graph of size $|V_{i}|$ that is sampled from $w$.

\begin{remarks} [Strengthening of  Theorem \ref{ThmReconMainRes2}]

Just as Theorem \ref{ThmMaxOrd} is a stronger version of Theorem \ref{ThmReconMainRes1}, our arguments can be used to give other versions of Theorem \ref{ThmReconMainRes2}. As there are two interesting versions, and both can be described entirely in terms of a small number of simple substitutions, we describe these substitutions in this remark rather than copying large amounts of nearly-identical text:

\begin{enumerate}
\item \textbf{Non-Uniform Graphons:} All references to Assumption \ref{AssumptionsSimpleWeakIdentifiabilityAssumptions} in the theorem statement and proof  can be replaced by references to the (strictly weaker) Assumption  \ref{AssumptionGoodGraphon}.
\item \textbf{Other Initial Sketches:} Step 2 of Algorithm \ref{AlgIterationHolder} calls  Algorithm \ref{AlgMerge}, but in fact any algorithm can be used for this initial sketch. The conclusions of Theorem \ref{ThmReconMainRes2} remain true as stated if \textit{all} of the following substitutions are made:
\begin{itemize}
\item Replace the call to Algorithm \ref{AlgMerge} in step 2 of Algorithm \ref{AlgIterationHolder} by a call to another algorithm, which we call \textit{OtherSketch}.
\item Replace Assumption \ref{AssumptionsSimpleWeakIdentifiabilityAssumptions} in the theorem statement with the following two assumptions: (i) part (3) of Assumption \ref{AssumptionGoodGraphon}, and (ii) a new assumption, which we will call \textit{Assumption OS}, that the output of \textit{OtherSketch} satisfies the error bound appearing in Theorem \ref{ThmReconMainRes1}.
\item In the proof, replace all references to Theorem \ref{ThmMaxOrd} by references to \textit{Assumption OS}. 
\end{itemize}
\end{enumerate}

\end{remarks}

\subsection{Algorithm \ref{AlgRefine}: Basic Concepts} \label{SecErrorRootingDisc}

We start by discussing Algorithm \ref{AlgRefine}, which builds a new ordering $\sigma_{2}$ on $V_2$ by looking at the edges in $G_{2}$ and the old ordering $\sigma_{1}$ on $V_1 \subset V_{2}$. Our main goal is to show the following: if Algorithm \ref{AlgRefine} is called with the right choice of parameters $C_1,C_2,C_3$, and $\sigma_1$ agrees with $\sigma_{\mathrm{true}}$ at certain precision level $d_1$, then $\sigma_2$ agrees with $\sigma_{\mathrm{true}}$ at a finer precision level $d_{2}<d_1$. We now outline the proof and define the basic concepts. As in earlier informal descriptions, we elide the distinction between a permutation and its total reverse.

First, we need to define the boundaries where $w$ drops to zero. Let $\delta $ be as in Assumption \ref{AssumptionSharpBoundary}, and define new boundary functions $r_\delta, \ell_\delta \, : \, [0,1] \mapsto [0,1]$  by:
\be
r_\delta(x)&=\sup\{ y \in [0,1] \, :\,w(x,y)\geq \delta\},\\
\ell_\delta(x)&=\inf\{ y \in [0,1] \,:\,w(x,y)\geq \delta\}.
\ee
These definitions will replace the very similar boundaries defined in \eqref{Eq:Boundaries}, which were defined with respect to the square $\WA{}$. By Assumption \ref{AssumptionSharpBoundary}, $w(x,y)=0$ if $y>r_\delta (x)$ or $y<\ell_\delta (x)$.

To start the heuristic analysis, assume that $\sigma_1$ is an ordering of $V_1$ which agrees with $\sigma_{\mathrm{true}}$ at precision level $d_1$.
We now show how $\sigma_1$ is used in steps 1--12 of Algorithm \ref{AlgRefine} to find an ordering $\sigma_2$ of $V_2\setminus V_1$ which agrees with $\sigma_{\mathrm{true}}$ at smaller precision level $d_2$, and has certain other helpful properties to be specified later. Fix $i,j\in V_2\setminus V_1$ so that $|U_i-U_j|>d_2$. Assume without loss of generality that $U_i<U_j$, so we have that $\sigma_{\mathrm{true}}(i)<\sigma_{\mathrm{true}}(j)$. Since $w$ is diagonally increasing,
\be
\{ z\,:\, w(U_i,z)=0\not=w(U_j,z)\text{ or }w(U_j,z)=0\not= w(U_i,z)\}= [r_\delta(U_i),r_\delta(U_j)]\cup[\ell_\delta(U_i),\ell_\delta (U_j)].
\ee
From Assumption \ref{AssumptionSharpBoundary} we have that 
\[
\big(r_\delta(U_j)-r_\delta(U_i)\big)+ \big(\ell_\delta (U_j)-\ell_\delta (U_i)\big)\geq B|U_j-U_i|.
\]
Assume first that
\be\label{AssumptionUiUj}
r_\delta(U_j)-r_\delta(U_i)\geq \ell_\delta (U_j)-\ell_\delta (U_i),
\ee
and thus 
\be \label{IneqSimpleRDelta}
r_\delta(U_i) < r_{\delta}(U_{j})- \frac{B}{2} |U_{i}-U_{j}| \leq r_\delta(U_{j}) - d_2\log(n)^{-1},
\ee where the second inequality holds for all $n $ large enough that $\log(n) > \frac{2}{B}$. Let
\be\label{DefItrue}
I_{\mathrm{true}} (i,j) =[r_{\delta}(U_{j}) -d_2\log(n)^{-1}, r_{\delta}(U_{j})].
\ee
Then  $w(U_{i},x) = 0, \, w(U_{j},x) \geq  \delta$ for all $x \in I_{\mathrm{true}}(i,j)$. Because of this ``sharp" distinction between $w(U_{i},x)$ and $w(U_{j},x)$, the region $I_{\mathrm{true}}$ is very informative in distinguishing the ordering of $i,j$.  Moreover, many points in $V_{1}$ will have latent positions in $I_{\mathrm{true}}(i,j)$ (roughly $p_{1} d_{2} n \log(n)^{-1}$ points on average). Step 2 of the algorithm computes the set $R(i,j)$, defined as the $C_1$ neighbours of either $i$ or $j$ in $V_1$ that are ranked highest according to $\sigma_1$. The idea behind the algorithm is that $R(i,j)$ is a small set that must contain many points with latent positions in this strongly-distinguishing set $I_{\mathrm{true}}(i,j)$; when this works, we expect the vertex with larger value $U_j$ to have significantly more neighbours in $R(i,j)$.

The main motivation behind our definition of $R(i,j)$ is thus to include the entire ``signal" in $I_{\mathrm{true}}$, while including as few additional ``noisy" vertices as possible. We will prove the following estimates, which make this intuition rigorous:

\begin{enumerate}
\item \textbf{Large Signal:} Denote by
\be\label{def:Dist}
{\mathbf{Dist}}_{R}(i,j) \equiv \{k \in V_1\, :  U_k\in I_{\mathrm{true}}(i,j)\}
\ee
the set of vertices that strongly distinguish between $i$ and $j$ on the right.
We will show that  ${\mathbf{Dist}}_{R}(i,j)$ is relatively large. In particular, we show in
Lemma \ref{BoundDist} below that \wep\
$|\mathbf{Dist}_R(i,j) | \geq  p_1d_2n\log(n)^{-2}.
$

As pointed out earlier, ${\mathbf{Dist}}_{R}(i,j)\cap N(i)=\emptyset$, so all vertices in $R(i,j)\cap \mathbf{Dist}_R(i,j)$ are neighbours of $j$ and not of $i$.
%
This leads to the definition:
\be \label{def:Signal}
{\mathbf{Signal}}_{R}(i,j) &\equiv {\mathbf{Dist}}_{R}(i,j)\cap N(j).
\ee
We will show in Lemma \ref{LemmaSignalBig} that, for appropriately chosen $C_2$, w.e.p.~
\be \label{Ineq:SigSizeHeuristic}
|\mathbf{Signal}_{R}(i,j)|\geq 2C_2.
\ee
Moreover, we show in Lemma \ref{LemmaSignalContainedinR} that w.e.p.
\be \label{Ineq:SignalSubsetR}
\mathbf{Signal}_{R}(i,j) \subseteq R(i,j).
\ee

\item \textbf{Small Noise:} There are many vertices $k \in R(i,j)$ that are not in the ``high-signal" region $\mathbf{Dist}_{R}(i,j) $. 
Due to the fact that $G$ is sampled from a diagonally increasing graphon, which satisfies inequality \eqref{IneqEmbBasic}, we expect to see at least as many neighbours of $j$ as neighbours of $i$ in $R(i,j)\setminus \mathbf{Dist}_{R}(i,j)$, but there may be very small (or no) signal in this discrepancy. Indeed, the expected discrepancy is \textit{exactly} 0 for our test-case graphon \eqref{EqSimpleGraphon}. Thus, we view these edges as essentially adding noise to our comparison. We define

\be \label{IneqDefNoise}
\mathbf{Noise}_{R}(i,j) &\equiv |R(i,j)\cap N(i)|-|(R(i,j)\cap N(j))\setminus \mathbf{Signal}_{R}(i,j) |.
\ee
When $\mathbf{Signal}_{R}(i,j) \subseteq R(i,j)$, then,

\[
|N(j)\cap R(i,j)|-|N(i)\cap R(i,j)|= |\mathbf{Signal}_R(i,j)|-\mathbf{Noise}_R(i,j).
\]
We will show in Lemma \ref{LemmaSmallNoise} below that, w.e.p.~$\mathbf{Noise}_{R}(i,j)
<C_2$.
In particular, this noise term will always be small compared to the signal term in Inequality \eqref{Ineq:SigSizeHeuristic}. When these high-probability events all occur, the function $F^{(2)}$ as defined in Steps 3--9 of Algorithm \ref{AlgRefine} agrees with $\sigma_{\mathrm{true}}$ if $|U_i-U_j|>d_2$.
\end{enumerate}

Finally, $L$, $ \mathbf{Dist}_L$, $\mathbf{Signal}_L$, $\mathbf{Noise}_L$ are the obvious analogues of $R$, $ \mathbf{Dist}_R$, $\mathbf{Signal}_R$, $\mathbf{Noise}_R$ on the left, and the obvious analogues to the above bounds hold by the same arguments for the case where (\ref{AssumptionUiUj}) does not hold (that is, where $r_\delta(U_j)-r_\delta(U_i)< \ell_\delta (U_j)-\ell_\delta (U_i)$). Therefore, \wep\ the  ordering $\sigma_{2}'$ on $V_2\setminus V_1$ obtained in Step 12 of the algorithm agrees with $\sigma_{\mathrm{true}}$ at precision level $d_2$.

The second part  of the algorithm (steps 13--24) then extends this ordering to all of $V_2$.

\begin{enumerate}
\item[(3)] \textbf{Highest and lowest ranked neighbours indicate ordering} The values $t(i)$ and $b(i)$, computed in step 14, are the ranks, according to $\sigma_2'$, of the top and the bottom element in $N(i)\cap (V_2\setminus V_1)$, respectively. If $|U_i-U_j|>d_2\log(n)^2$, then \wep\ there will be many neighbours $k$ of $j$ with latent position $U_{k}$ greater than that of any neighbour of $i$. Thus the top elements of $N(i)$ and $N(j)$ in $V_2\setminus V_2$, ranked according to $\sigma_{\mathrm{true}}$, will differ significantly in rank. Since $\sigma_2'$ is a good approximation of $\sigma_{\mathrm{true}}$, this will also be true for these same elements when ranked according to $\sigma_2'$. Thus, as will be made precise in Lemma \ref{LemmaExtendOrdering}, the ordering $\sigma$ as returned by the algorithm will agree with $\sigma_{\mathrm{true}}$ at precision level $d_2\log(n)^2$.
\end{enumerate}

\subsection{Algorithm \ref{AlgRefine}: Correctness} \label{Sec:Refine}

We are now ready to prove the statements in the previous section. We retain the same notation, and now give conditions on the parameters which we assume to hold throughout the remainder of this section:

\be \label{def:RefineParameters}
1\geq p_2 &\geq  p_1\log (n)^{3},\\
\log(n)^{-5} \geq d_1 &\geq \log(n)^{10}/(p_1n),\\
d_2 &= \sqrt{d_1/(p_1n)}\log (n)^{-1},\\
C_1 &= \lceil p_1 d_1n\log(n)^{4}\rceil,\\
C_2 &= \lfloor \sqrt{p_1d_1 n}\log(n)^{6}\rfloor,\\
C_3 &= \lfloor \sqrt{p_1d_1 n} \log(n)^{2}\rfloor.
\ee

Note that the definition of $C_1,C_2,C_3$ as given above agrees with the definition of these parameters in Step 5 of Algorithm \ref{AlgIterationHolder} in its first iteration (with $i=1$). The first three conditions are satisfied by our choice of the sequence $p_{i},d_{i}$ in Equation  \eqref{DefParameters}.
We use the following immediate consequences of \eqref{def:RefineParameters}:
\be
d_2&\leq d_1 \log(n)^{-6}<d_1,\\
d_1&\geq \frac{\log(n)^{13}}{n} \\
C_1& \leq p_1n\log(n)^{-1}+1\text{, and}\\
\sqrt{p_1d_1n} &\geq \log(n)^{5}. \\
\ee
Moreover, we assume throughout that $n$ is large enough so that $\log(n)$ exceeds any of the constants appearing in this section.

For the correctness proof of Algorithm \ref{AlgRefine}, we define certain bad events on the variables $U_i$ and $U_{i,j}$, and show that \wep\ they do not occur. Let 
\be 
\mathcal{A}_{6} &=  \bigcup_{k,\ell\in V,\, U_k<U_\ell -\frac{1}{n}} \left\{ \big|\, |\{ i\in V: U_k<U_i<U_\ell\}|- n |U_{\ell}-U_{k}|\,\big|> \sqrt{(U_{\ell}-U_{k}) n}\log^2 (n)\right\}.\\
\ee
On the event $\mathcal{A}_6^c$, the number of vertices in each interval is concentrated around its expected value. The event is a property of variables $U_i$, and is closely related to the event $\mathcal{A}_5$ as given in Definition \ref{DefBadEvents}. It is a straightforward consequence of Chernoff's inequality and a union bound that $\mathcal{A}_{6}^c$  occurs \wep

We will also need this kind of nice behaviour for all our sampled sets $V_i$. For a given $p$, let $S = S(p) \equiv \{ i: B_i\leq p\}$ and define 
\be \label{DefA6subset}
\mathcal{A}'_{6}(p) &=  \bigcup_{k,\ell\in V,\, U_k<U_\ell-\frac{1}{np}} \left\{ \big|\, |\{ i\in S(p) : U_k<U_i<U_\ell\}|-np|U_{\ell}-U_{k}|\,\big|> \right.\\
&\hspace*{8.5cm}\left. 2\sqrt{np(U_{\ell}-U_k)}\log (n)\right\}.\\
\ee

By the Chernoff bound, for any given sequence $p = p(n) \in (0,1)$ satisfying $\limsup_{n \rightarrow \infty} \frac{-\log(p(n))}{\log(n)} < 1$,  $(A_6'(p))^c$ holds \wep \, This also means that it holds simultaneously \wep\ for any sequence $p_{1},p_{2},\ldots,p_{k}$, as long as $k$ grows at most polynomially quickly in $n$.

We can  use $\mathcal{A}_6'$ to establish bounds on the size of $\mathbf{Dist}_R(i,j)$, where $i$ and $j$ are so that $U_i > U_i +d_2$. We assume without loss of generality that $r_\delta(U_j)-r_\delta(U_i)\geq \ell_\delta(U_j)-\ell_\delta(U_i)$. Applying Inequality \eqref{IneqSimpleRDelta}, 

\be \label{IneqGoodUiUj}
U_{j} > U_{i} + d_{2} \quad  \text{and} \quad r_\delta(U_j)-r_\delta(U_i)\geq d_2\log(n)^{-1}.
\ee

\begin{lemma}\label{BoundDist}
W.e.p.\ for all $i,j$ satisfying \eqref{IneqGoodUiUj},
\be
|\mathbf{Dist}_R(i,j) | \geq  p_1d_2n\log(n)^{-2}.
\ee
\end{lemma}

\begin{proof}

Let $i$ and $j$ as stated. By definition, $|I_{true}(i,j)|= d_2\log(n)^{-1}$, and so by the choice of parameters as given in (\ref{def:RefineParameters}) the set $\mathbf{Dist}_R(i,j)=\{  s\in V_1\,:\,s\in I_{true}(i,j)\}$  has expected size
\[
\mathbb{E}(|\mathbf{Dist}_R(i,j)|) = n p_{1} |I_{true}(i,j)| = p_{1} d_2 n\log(n)^{-1}.
\]
We also note that $\frac{1}{p_{1}n} \leq  d_{2} \leq |U_{i} - U_{j}|$. Combining these two facts, for all $i,j\in V_2\setminus V_1$,
\be 
\{ |\mathbf{Dist}_R(i,j) | < p_1d_2n\log(n)^{-2}\}\,\subset \,\mathcal{A}_6'(p_1).
\ee
Since \wep\ $\mathcal{A}_6'(p_1)^c$ holds, the lemma follows.
\end{proof}

We will next check that all moderately-long intervals contain roughly the expected number of neighbours of any given vertex (recall that $N(i)$ is the usual graph neighbourhood of $i$ in $G$). More precisely, for a set $S\subset V$, an ordering $\sigma$ on $S$, and a pair $k,\ell \in S$, define the interval $I(S, \sigma, k,\ell) = \{ s\in S\,:\,\sigma(k)<\sigma (s)<\sigma (\ell)\}$ and define $\mathcal{I}(S,\sigma) = \{I(S,\sigma,k,\ell):k,\ell \in S\}$ to be the collection of such intervals. For any vertex $i \in V$ and set $I \subset V$, let $W(i,I)=\sum_{j\in I} w(U_i,U_j) = \E [|N(i)\cap I|]$.
For sets $S,T\subseteq V$, and ordering $\sigma$ of $S$, define
\be
\mathcal{A}_{7}(S,T,\sigma) &= \bigcup_{I\in {\mathcal{I}}(S,\sigma),\,i\in T }
\left\{\big| \, | I\cap N(i)\big|  - W(i,I)\big| >\sqrt{W(i,I)}\log(n)+\log(n)^2 \right\}\notag \\
\ee
We now use this to show that $\mathbf{Signal}_R(i,j)$ is sufficiently large.

\begin{lemma} \label{LemmaSignalBig}
W.e.p.\ for all $i,j$ satisfying \eqref{IneqGoodUiUj},
\be
|\mathbf{Signal}_R(i,j) | \geq  2C_2.
\ee
\end{lemma}

\begin{proof} We first show that $\mathcal{A}_{7}^{c}$ occurs \wep \, for the relevant parameters.
Let $\mathcal{H}$ be the $\sigma$-algebra generated by the variables $\{ U_i\}$ and  $\{ B_i\}$, and let $\mathcal{H}(G_1)$ be the $\sigma$-algebra of the variables $\{ U_i\}$ and  $\{ B_i\}$, and  $\{ U_{i,j}:i,j\in V_1\}$. Recalling that the size of the set $|I\cap N(i)|$ appearing in the definition of $\mathcal{A}_{7}$ can be written as a sum of independent Bernoulli random variables in these $\sigma$-algebras,\footnote{Recall we have assumed that $\sigma_1$ is completely determined by $G_1=G[V_1]$.} then applying Hoeffding's inequality and a union bound, we immediately have
\be\label{BoundA7}
\P [ \mathcal{A}_{7}(V_1,V_2\setminus V_1,\sigma_1)\,|\,\mathcal{H}(G_1)]&=n^{-\Omega(\log (n))},\\
\P [ \mathcal{A}_{7}(V_2,V_2,\sigma_{\mathrm{true}})\cup \mathcal{A}_{7}(V_1,V_2,\sigma_{\mathrm{true}})\,|\,\mathcal{H}]&=n^{-\Omega(\log (n))}.
\ee

By definition $\mathbf{Dist}_R(i,j)$ is an interval in $\mathcal{I}(V_1,\sigma_{\mathrm{true}})$; also recall $|\mathbf{Signal}_R(i,j)|=|\mathbf{Dist}_R(i,j)\cap N(j)|$. By Assumption \ref{AssumptionSharpBoundary}, $w(U_j,U_k)\geq \delta$ for all $k\in \mathbf{Dist}_R(i,j)$, and thus $W(j, \mathbf{Dist}_R(i,j))\geq \delta |\mathbf{Dist}_R(i,j)|$. Also, by \eqref{def:RefineParameters}, $C_2/\delta \geq \log(n)^3$ for $n$ sufficiently large. We can then directly conclude that
\be\label{Eq:LargeSignal}
&\{ |\mathbf{Signal}_R(i,j)|< 2C_2 \}\\
&\subset \, \{ |\mathbf{Dist}_R(i,j) | < 3C_2/\delta\}\,\cup \,\{ |\mathbf{Dist}_R(i,j)\cap N(j)|<(2/3)\delta |\mathbf{Dist}_R(i,j)|\}\\
&\subset {\mathcal{A}'}_6(p_1) \, \cup \, \mathcal{A}_7(V_1,V_2\setminus V_1,\sigma_{\mathrm{true}}).
\ee
Since {the complements of the events} $\mathcal{A}_6'(p_1)$ and $\mathcal{A}_7(V_1,V_2\setminus V_1,\sigma_{\mathrm{true}})$ occur \wep, the result follows.
\end{proof}
In addition, we show that this large signal is indeed contained in $R(i,j)$:

\begin{lemma}\label{LemmaSignalContainedinR}
W.e.p.\ for all $i,j\in V_2\setminus V_1$ satisfying \eqref{IneqGoodUiUj},
\[
\mathbf{Signal}_R(i,j) \subset R(i,j).
\]
\end{lemma}

\begin{proof}
Suppose $\mathbf{Signal}_R(i,j) = \mathbf{Dist}_{R}(i,j)\cap N(j) \not\subset R(i,j)$. Let
\[
k=\text{argmax} \{ \sigma_1(s): s\in (\mathbf{Dist}_{R}(i,j)\cap N(j))\setminus R(i,j)\}.
\]
That is, $k$ is the most highly ranked vertex in the signal that is not included in $R(i,j)$.

For simplicity of notation, let $r_j=r_\delta(U_j)$. By definition, $U_k\in I_{true}$, so $U_k\geq r_j-d_2\log(n)^{-1}$. Each $\ell\in R(i,j)$ is a neighbour of $i$ or $j$, so $U_\ell<r_j$. Moreover, since $\ell\in R(i,j)$ and $k\not\in R(i,j)$, it follows from the definition of $R(i,j)$ that  $\sigma_1(\ell)>\sigma_1(k)$. We assumed that $\sigma_1$ agrees with $\sigma_{\mathrm{true}}$ at precision level $d_1$. So then either $\ell$ and $k$ are correctly ordered by $\sigma_1$ (in which case $U_\ell>U_k\geq r_j-d_2\log(n)^{-1}$), or $|U_\ell-U_k|<d_1$ (in which case $U_\ell>U_k-d_1\geq r_j-d_2\log(n)^{-1}-d_1$). Since $d_2<d_1$,
it follows that for each $\ell\in R(i,j)$, $r_j-2d_1<U_\ell<r_j$. Since we have that $R(i,j)\subset V_1$ and $C_1\geq p_1n d_1\log (n)$ (by \eqref{def:RefineParameters}),
\be
\{  \mathbf{Dist}_{R}(i,j)\cap N(j)\not\subset R(i,j)\}
\subset \{ |\{ s\in V_1: r_j-2d_1<U_s<r_j\}|\geq C_1 \}
\subset  {\mathcal{A}'}_6(p_1).
\ee
We saw immediately following \eqref{DefA6subset} that $(\mathcal{A}'_6(p_1))^{c}$ occurs \wep, completing the proof.
\end{proof}

From this lemma and Lemma \ref{LemmaSignalBig}, we conclude that \wep ,
\be \label{Conclude:LargeSignal}
|\mathbf{Signal}_R(i,j) \cap  R(i,j)|=|\mathbf{Signal}_R(i,j)| \geq 2C_2.
\ee
Finally, we show that the noise in $R(i,j)$ is smaller than this lower bound on the signal.

\begin{lemma}
\label{LemmaSmallNoise}
W.e.p.\ for all $i,j\in V_2\setminus V_1$ satisfying \eqref{IneqGoodUiUj},
\[
\mathbf{Noise}_{R}(i,j) \leq C_2.
\]
\end{lemma}

\begin{proof}
We first show that \wep\ $U_j<U_k$ for all $k\in R(i,j)$. By Assumption \ref{AssumptionGoodGraphon}, $\WA{}$ is $(\epsilon,\alpha)$-connected, and it follows that there must be an $\epsilon'>0$ so that $w(x,y)>0$, and thus $w(x,y)\geq \delta$, for all $x,y$ with $|x-y|\leq\epsilon'$. Consider the interval $J=\{ s\in V_1\,:\, U_j < U_s < U_j+\epsilon'\}$. Then for each $s\in J$, $U_s\leq r_\delta (U_j)$, so $w(U_j,U_s)\geq \delta$. If $\mathcal{A}_7(V_1,V_2\setminus V_1,\sigma_{\mathrm{true}})^c$ and $\mathcal{A}_6'(p_1)^c$ hold, then $|N(j)\cap J|\geq (\delta /2)|J|\geq (\delta/4)\epsilon' p_1n >C_1,$ where the last inequality holds for all $n$ sufficiently large. It follows that \wep\ $j$ has at least $C_1$ neighbours with $U$-values greater than $U_j$, and so  $U_j<U_k$ for all $k\in R(i,j)$ as desired.

We now proceed to bound the noise, i.e.~the excess number of neighbours of $i$ over the neighbours of $j$ in $R(i,j)$. Since $w$ is diagonally increasing, it follows that $W(i,R(i,j))\leq W(j,R(i,j))\leq |R(i,j)|=C_1$. Moreover, by definition of $R(i,j)$, there must be an interval $I\in \mathcal{I}(V_1, \sigma_1)$ so that $R(i,j)=(I\cap N(i))\cup (I\cap N(j))$. By our choice of parameters \eqref{def:RefineParameters}, for all sufficiently large $n$
\be \label{IneqSimpleParamC2C1Comp}
C_2/2\geq  \sqrt{d_1p_1n}\log (n)^6/2 - 1 > 2\sqrt{p_1d_1n}\log(n)^5 \geq \sqrt{C_1}\log(n)+\log(n)^2.
\ee

Let $D_i=|N(i)\cap I|-W(i,I)$ and $D_j=W(j,I)-|N(j)\cap I|$. Then
\[
\mathbf{Noise}_{R}(i,j)\leq |N(i)\cap R(i,j)|-|N(j)\cap R(i,j)|=|N(i)\cap I|-|N(j)\cap I|\leq D_i+D_j.
\]
We then have
\be
\{\mathbf{Noise}_{R}(i,j)>C_2\}\,&\subset \, \{ D_i+D_j\geq C_2\}\\
&\subset\, \{ D_i\geq C_2/2\} \cup \{ D_j\geq C_2/2\}\\
&\subset\, \{ D_i\geq \sqrt{C_1}\log(n)+\log (n)^2\} \cup \{ D_j\geq \sqrt{C_1}\log(n)+\log (n)^2\} \\
&\subset\, \mathcal{A}_{7}(V_1,V_2\setminus V_1,\sigma_1),
\ee
where \eqref{IneqSimpleParamC2C1Comp} is used in the third line. The result follows from \eqref{BoundA7}.
\end{proof}

This completes the proof of the ``correctness" of the first step of Algorithm \ref{AlgRefine}, as summarized in the next lemma. For  sets $S\subseteq V$, parameter $d\in (0,1)$, and ordering $\sigma$ of $S$, define
\be\label{DefA8}
\mathcal{A}_{8}(S,\sigma ,d) &= \bigcup_{i,j\in S,\, |U_i-U_j|>d} \{ \mathbf{1}_{\sigma (i)<\sigma (j)}=\mathbf{1}_{\sigma_{\mathrm{true}}(i)>\sigma_{\mathrm{true}}(j)}\}.
\ee
That is, $(\mathcal{A}_{8}(S,\sigma ,d))^c$ holds if $\sigma$ agrees with $\sigma_{\mathrm{true}}$ at precision level $d$.
 \\

\begin{lemma}\label{LemmaRefineFirst}
Suppose parameters $C_1, C_2,p_1,p_2,d_1,d_2$ are chosen according to (\ref{def:RefineParameters}). Let $V_1,V_2$ be sampled from $V$ at rate $p_1,p_2$, respectively, and assume $\sigma_1$ is an ordering of $V_1$, derived only from $G[V_1]$, which agrees with $\sigma_{\mathrm{true}}$ at precision level $d_1$.
Then, for  $\sigma_2$ as computed in step 12 of Algorithm \ref{AlgRefine},
\[
\P [ \mathcal{A}_{8}(V_2\setminus V_1,\sigma_2 ,d_2)\,|\,\mathcal{H}[G_1]]=n^{-\Omega(\log{n})}.
\]
That is, \wep\, $\sigma_2$ is an ordering of $V_2\setminus V_1$ that agrees with $\sigma_{\mathrm{true}}$ at precision level $d_2$.
\end{lemma}

\begin{proof}
Suppose $i,j\in V_2\setminus V_1$ and $|U_i-U_j|>d_2$. Assume \emph{wlog}\ that $U_i<U_j$. We consider first the case 
\be \label{IneqRefineFirstCase1}
r_\delta(U_j)-r_\delta(U_i)\geq \ell_\delta(U_j)-\ell_\delta(U_i).
\ee  
In this case, $i,j$ satisfy \eqref{IneqGoodUiUj}, and we conclude from (\ref{Eq:LargeSignal}), Lemma \ref{LemmaSignalContainedinR} and \ref{LemmaSmallNoise} that \wep
\[
|R(i,j)\cap N(j)|-|R(i,j)\cap N(i)| \geq  \mathbf{Signal}_R(i,j)-\mathbf{Noise}_R(i,j)\geq C_2,
\]
and thus $F(i,j)$ is set to 1 in Steps 4--9 of Algorithm \ref{AlgRefine}. 


If \eqref{IneqRefineFirstCase1} does not hold, then there may not be a sufficiently long interval $I_{true}$ to the right of $U_j$, but by Assumption \ref{AssumptionSharpBoundary} there must be a corresponding interval $I_{true}=[\ell_\delta(U_i),\ell_\delta(U_i)+d_2\log(n)^{-1}]$ so that $w(x,i)\geq \delta$ and $w(x,j)=0$ for all $x\in I_{true}$. By applying the arguments above to the values $1-U_i$, we see then that \wep\ $|L(i,j)\cap N(i)|-|L(i,j)\cap N(j)|\geq C_2$, and thus $F(i,j)$ is set to 1 in Steps 4--9 as above. Therefore, in both cases \wep\ $i$ and $j$ will ordered in $\sigma_2$ according to $\sigma_{\mathrm{true}}$.
The result now follows by a union bound.
\end{proof}

Finally, we show that the extension of $\sigma_2$ to all of $V_2$ has the desired precision level.

\begin{lemma}\label{LemmaExtendOrdering}
Suppose $p_1,p_2,d_2,C_3$ are as given in (\ref{def:RefineParameters}), and let $V_1,V_2$ be sets sampled from $V$ at rate $p_1,p_2$, respectively. Suppose $\sigma_{2}'$ is an ordering of $V_2\setminus V_1$ that agrees with $\sigma_{\mathrm{true}}$ at precision level $d_2$. Then
\[
\P [\bigcup_{i,j\in V_2,\, U_j>U_i+d_2\log(n)^{2}} \{t(j)-t(i)<C_3\} \cap \{ b(j)-b(i)<C_3\}]=n^{-\Omega(\log(n))}.
\]
That is, \wep\ any ordering of $V_2$ extending $\sigma_{2}'$ and based on the function $F^{(2)}$ as computed in steps 16--20 of Algorithm \ref{AlgRefine} agrees with $\sigma_{\mathrm{true}}$ at precision level $d_2\log(n)^{2}$.
\end{lemma}

\begin{proof}
Fix $i,j\in V_2$ satisfying $U_{j} \geq U_{i} + d_2\log(n)^2$ and \eqref{IneqGoodUiUj}.
By the same argument leading to Inequality \eqref{IneqSimpleRDelta}, this gives $r_\delta(U_j)-r_\delta(U_i)> 2d_2\log (n)$.  Define
\[
I_2(j)=\{ s\in V_2\setminus V_1\, :\, U_{s} > r_\delta(U_j)-d_2\log(n)\}.
\]
We will first show that, for all $s\in  I_2(j)$, $\sigma'_2(s)>t(i)$. Let $k \in N(i)$ be the vertex at the top of the range of $N(i)$ according to $\sigma_{2}'$ - that is, $\sigma_{2}'(k)=t(i)$.
 Since $k\in N(i)$, we have that $U_k\leq r_\delta (U_i)$, and so for $s \in I_{2}(j)$
\[
U_{s}-U_{k} \geq (r_\delta(U_j)-d_2\log (n))-r_\delta(U_{i})>d_2.
\]
This implies that $s$ and $k$ are correctly ordered by $\sigma_2'$ and thus $t(i) =\sigma_{2}'(k)<\sigma_{2}'(s)$ as desired.

This correct ordering implies that $k\not\in I_2(j)$, and so we see that $t(j)-t(i)\geq |I_2(j)\cap N(j)|$, so we have the containment
\be \label{IneqAlltogether1}
\{ t(j)-t(i) <C_3 \} &\subset \{ |I_2(j)\cap N(j)|< C_3\}.
\ee

From \eqref{def:RefineParameters}, we know that
\be
C_3 \leq \sqrt{d_1p_1n}\log(n)^2
 = d_2 p_1n\log(n)^3
 \leq  d_2p_2n.
\ee
For large enough $n$,  if $(\mathcal{A}_6'(p_2)\cup \mathcal{A}_6'(p_1))^c$ holds then
\be \label{IneqAlltogether2}
|I_2(j)|\geq \frac14 (d_2\log(n))p_2n>(2/\delta)C_3,
\ee
and if $\mathcal{A}_7(V_2,V_2\setminus V_1,\sigma_{\mathrm{true}})$ holds, then
\be \label{IneqAlltogether3}
|I_2\cap N(j)|\geq \frac12 W(j,I_2(j))\geq \frac{\delta}{2} |I_2(j)|.
\ee
Putting together \eqref{IneqAlltogether1}, \eqref{IneqAlltogether2} and \eqref{IneqAlltogether3},
\be
\{ t(j)-t(i) <C_3 \} \subset \mathcal{A}_6'(p_2)\cup \mathcal{A}_6'(p_1)\cup \mathcal{A}_7(V_2,V_2\setminus V_1,\sigma_{\mathrm{true}}).
\ee
Since we have already shown $\P [\mathcal{A}_6'(p_2)\cup \mathcal{A}_6'(p_1)\cup\mathcal{A}_7(V_2,V_2\setminus V_1,\sigma_{\mathrm{true}})]=n^{-\Omega(\log(n))}$, it holds \wep\ that $t(j)-t(i)\geq C_3$, and thus $F^{(2)}(i,j)$ is set to 1 in Step 19 of Algorithm \ref{AlgRefine}.

Similarly, if $U_i<U_j$ and the opposite of \eqref{IneqSimpleRDelta} holds,
then $\P [\{ b(j)-b(i)<C_3 \}]\leq n^{-\Omega (\log (n))}$. Therefore, \wep\ $F^{(2)}(i,j)$ is set to 1 in Step 19 of Algorithm \ref{AlgRefine}.
The result follows by a union bound.
\end{proof}

The correctness of Algorithm \ref{AlgRefine} now follows directly from the results in this section.

\begin{lemma}\label{LemmaRefineCorrect}
Let $V_1,V_2$ be sets sampled at probabilities $p_1,p_2$, respectively, according to the variables $\{ B_i\}$, and suppose $\sigma_1$ is an ordering of $V_1$ that agrees with $\sigma_{\mathrm{true}}$ at precision level $d_1$.
Suppose Algorithm \ref{AlgRefine} is executed with parameters $C_1,C_2,C_3$. Assume that the conditions (\ref{def:RefineParameters}) hold. Let
$\sigma = Refine(V_1,V_2,\sigma_1)$. Then
\[
\P [ \mathcal{A}_{8}(V_2,\sigma ,\sqrt{ d_1/(p_1n)}\log (n) )]=n^{-\Omega(\log{n})}.
\]
That is, \wep\ $\sigma$
 is a total order of $V_2$ that agrees with $\sigma_{\mathrm{true}}$ at precision level $\sqrt{ d_1/(p_1n)}\log (n)$.
\end{lemma}

\subsection{Proof of Theorem \ref{ThmReconMainRes2}}

Finally, we show the correctness of Algorithm \ref{AlgIterationHolder} and prove Theorem \ref{ThmReconMainRes2}. That is, we show that, if Algorithm \ref{AlgIterationHolder} is executed with a particular set of parameters, and given as input a large enough graph $G$, then \wep\  the algorithm will return a total order of $V$ with error at most $n^{\epsilon}$, where $\epsilon >0$ is any desired constant error exponent. The proof will follow directly from Theorem \ref{ThmMaxOrd} and Lemma \ref{LemmaRefineCorrect}.

First, we define the parameters. For any $0<\epsilon < 0.5$, let integer $k$ and sequence $\{ p_i,d_i\}_{i=1}^k$ be as follows:

\be\label{DefParameters}
k &= \lfloor -\log_2(\epsilon )\rfloor +1,\\
\beta &= \frac{\epsilon - 2^{-k}}{k},\\
p_i &= n^{-(k-i)\beta},\,i=1,\dots, k,\\
d_{1} &= n^{-0.5\,(1 -k\beta)}, \\
d_{i+1} &= \sqrt{d_{i}/(p_{i} n)}\log (n),\, i=1,\dots , k-1.
\ee

The definition of $d_2$ according to \eqref{DefParameters} above differs by a factor $\log^2(n)$ from the definition of $d_2$ as given in \eqref{def:RefineParameters}; this extra factor corresponds to the extra $\log(n)^{2}$ appearing in the statement of in Lemma \ref{LemmaExtendOrdering}, which bounds the error introduced when extending the ordering. Note that $\{p_i\}$ is an increasing sequence, with $p_k=1$.  We will use the lemmas from previous sections to show that, after each iteration of Algorithm \ref{AlgIterationHolder}, the returned ordering of the subgraph induced by $V_i$ agrees with the true ordering at precision level $d_i$. The lemma below will have as corollary that $d_k\ll n^{-1+\epsilon} $, from which our result will follow.

\begin{lemma}\label{LemParameters}
For the parameters as defined in \eqref{DefParameters}, and $n$ sufficiently large, we have that for all $1\leq i\leq k$,
\[
d_i\leq n^{-(1-2^{-i})(1-k\beta)}\log (n).
\]
In particular, there exists $0 < \epsilon' < \epsilon$ so that $d_k\leq  n^{-1+\epsilon'}$.
\end{lemma}
\begin{proof}
We prove the statement by induction on $i$. The statement for $i=1$ can be directly verified from the definition of $d_1$. Suppose then that the statement holds for $1\leq i<k$. Then
\be 
d_{i+1} &= \sqrt{d_{i}/(p_{i} n)}\log(n) \\
&\leq n^{-0.5\,((1-2^{-i})(1-k\beta)-(k-i)\beta +1)}\log^{3/2}(n).
\ee
Now
\be 
 0.5\,((1-2^{-i})(1-k\beta)-(k-i)\beta +1)
 &= 0.5\, (2-2^{-i}-2k\beta +i\beta + 2^{-i}k\beta)\\
&= (1-2^{-(i+1)}-k\beta +2^{-(i+1)}k\beta) +0.5i\beta\\
&\geq (1-2^{-(i+1)})(1-k\beta) +0.5\beta.
\ee
For $n$ large enough, $n^{-0.5\beta}\geq \log^{1/2} (n)$. This completes the proof of the first statement. 
To prove the second statement, observe first that
\be
(1-2^{-k})(1-k\beta) &= (1-2^{-k})(1-(\epsilon -2^{-k}))\\
&= 1 -\epsilon + 2^{-k}(\epsilon -2^{-k}).\\
\ee
Let $\delta= 2^{-(k+1)}(\epsilon -2^{-k}) \in (0,\epsilon)$ and $\epsilon'=\epsilon -\delta$. It then follows that 
\[
d_k\leq n^{-1+\epsilon-2\delta}\log (n)=(n^{-\delta}\log(n))n^{-1+\epsilon'}.
\]
The result then follows for $n$ large enough so that $n^{-\delta}\log(n)\leq 1$.

\end{proof}

In fact, a computation very similar to that used in the above lemma shows that $\{ d_i\}_{i=1}^k$ is a decreasing sequence. 
 
In the following, we fix a value $\epsilon >0$ and take the sequence $\{ p_i,d_i\}_{i=1}^k$ as defined in \eqref{DefParameters}. We assume the Assumptions of Theorem \ref{ThmReconMainRes2} hold, and for all $i=1,\dots ,k$ the set $V_i=\{ s\,:\, B_s\leq p_i\}$ is as in Algorithm \ref{AlgIterationHolder}.

Algorithm \ref{AlgIterationHolder} first calls Algorithm \ref{AlgMerge} with input $G_1=G[V_1]$, where $V_1$ is sampled from $V$ at rate $p_1$. The algorithm returns a total order $\sigma_1$ of $V_1$. We will first show that, if $p_1,d_1$ are as defined above in \eqref{DefParameters}, then \wep\ $\sigma_1$ agrees with $\sigma_{\mathrm{true}}$ at precision level $d_1$.

\begin{lemma}\label{LemmaValidp1d1}
W.e.p., the ordering $\sigma_1$ as returned by Algorithm \ref{AlgMerge} in Step 3 of Algorithm \ref{AlgIterationHolder}, or its total reverse $\sigma_1^{trev}$, agrees with $\sigma_{\mathrm{true}}$ at precision level $d_1$. That is,
\[
P [ \mathcal{A}_{8}(V_1,\sigma_1,d_1 )\cap \mathcal{A}_{8}(V_1,\sigma_1^{trev},d_1 ) ]= n^{-\Omega(\log(n))}.
\]
Moreover, $\sigma_1$ only depends on $G_1=G[V_1]$.
\end{lemma}

\begin{proof}
From \eqref{DefParameters}, we have that $p_1=n^{-(k-1)\beta}$ and $d_1=n^{-0.5\, (1-k\beta)}$.
Note that $V_1$ is a set of size $n_1=|V_1|$ randomly sampled from $[0,1]$. By Theorem \ref{ThmMaxOrd}, \wep\ the returned ordering $\sigma_1$ agrees with $\sigma_{\mathrm{true}}$ or its reverse at precision level $n_1^{-0.5}\log (n_1)^4$. If $\mathcal{A}_6'(p_1)^c$ holds and $n$ is large enough, then $n_1\geq 0.5\, p_1n$, and
\[
n_1^{-0.5}\log (n_1)^4 \leq \sqrt{2}n^{-0.5(1-(k-1)\beta)} \log (n)^4= (\sqrt{2} n^{-0.5\,\beta} \log(n)^4) d_1,
\]

which is smaller than $d_{1}$ for all $n$ sufficiently large. The result follows.
\end{proof}

The next lemma shows that Algorithm \ref{AlgRefine}, when called by \ref{AlgIterationHolder} to extend the ordering from $V_i$ to $V_{i+1}$, improves the precision from $d_i$ to $d_{i+1}$.

\begin{lemma} \label{LemmaValidSeqs}
Let $1\leq i<k$. Suppose $\sigma_i$ is an ordering of $V_i$ which agrees with $\sigma_{\mathrm{true}}$ at precision level $d_i$, and which depends only on $G[V_i]$. Let $\sigma_{i+1}$ be the ordering returned by Algorithm \ref{AlgRefine} as it is called in Step  6 of Algorithm \ref{AlgIterationHolder}. Then \wep\ $\sigma_{i+1}$ agrees with $\sigma_{\mathrm{true}}$ at precision level $d_{i+1}$.
\end{lemma}

\begin{proof}
Fix $1\leq i<k$ and assume  $\sigma_i$ agrees with $\sigma_{\mathrm{true}}$ at precision level $d_i$. Note that the parameters $C_1,C_2,C_3$ as set in Step 6 of Algorithm \ref{AlgIterationHolder} and used  in the call to Algorithm \ref{AlgRefine}, satisfy the conditions (\ref{DefParameters}), with $p_i$ and $d_i$ taking the role of $p_1$ and $d_1$, respectively. Lemma \ref{LemmaRefineCorrect} then shows that \wep\ $\sigma_{i+1}=Refine(V_i,V_{i+1},\sigma_i)$ agrees with $\sigma_{\mathrm{true}}$ at precision level $\sqrt{d_i/(p_in)}\log (n)=d_{i+1}$.

\end{proof}

Finally, we prove Theorem \ref{ThmReconMainRes2}:

\begin{proof} [Proof of Theorem \ref{ThmReconMainRes2} ]

By definition, $V_k=V$, and Algorithm \ref{AlgIterationHolder} returns the ordering $\sigma=\sigma_{k}$ of $V$. By the previous lemmas and a union bound,  \wep\ $\sigma$ will agree with $\sigma_{\mathrm{true}}$ (or its reverse) at precision level $d_k$. By Lemma \ref{LemParameters},
 $d_k=n^{1-\epsilon'}$,
where $0 < \epsilon'<\epsilon$. 

The techniques from the previous section readily show that, \wep, if an ordering agrees with $\sigma_{\mathrm{true}}$ at precision level $d$, then it has error at most $d n\log(n)$. Since $n^{-(\epsilon -\epsilon')} \log (n) <1$ for large enough $n$, it follows that \wep\ $\sigma$ has error at most $n^\epsilon$.
\end{proof}

\section{Finding $\alpha$} \label{SecSeekingAlpha}

The first part of our results relies on knowing a good thresholding value of $\alpha \in [0,1]$. To implement the algorithm and obtain our guarantees, we need to find a value $\alpha$ for which Assumption \ref{AssumptionGoodGraphon} holds. Fortunately, for the general class of graphons satisfying the stronger conditions of Assumption \ref{AssumptionsSimpleWeakIdentifiabilityAssumptions}, the set of possible values for $\alpha$  will contain an interval of positive measure. We now list a few simple consequences of our properties that allow us to estimate a suitable value of $\alpha$ from the observed graph $G$. It is straightforward to check that for certain classes of graphons (including \textit{e.g.} \eqref{EqSimpleGraphon}), these tests will be able to find a suitable value of $\alpha$; in general, however, the following tests are \textit{necessary} but not \textit{sufficient} for the value of $\alpha$ to be suitable.

We begin with condition (1) of Assumption \ref{AssumptionGoodGraphon}, namely the requirement that $\WA{\alpha}$ is diagonally increasing. We have seen that, if $\WA {\alpha}$ is diagonally increasing, then with high probability, a sample of size $\log(n)^5 (n)$ from $G^{(2)}_\alpha$ will be a proper interval graph. Such samples are taken repeatedly in Algorithm \ref{AlgSketch}, and these samples are then tested by a proper interval graph recognition algorithm. Thus, if Step 3 of Algorithm \ref{AlgSketch} fails more than a small percentage of the times, then it is likely that $\WA{\alpha}$ is not diagonally increasing. When this occurs, the algorithm should be restarted with a new value of $\alpha$.

The second condition is more straightforward: we can test whether $\WA{}$ is uniformly $(A,\delta)$- good by testing if
\be
\max_{v\in V} \left| \left\{ w\in V\,:\, \left|\alpha - \frac{\GA{}(w,v)}{n-2} \right|  \leq \delta' \right\}\right| \leq A\delta' n,
\ee
for several test values of $0<\delta' <\delta$. Moreover, we can restrict the choice of $\delta$ and $\delta'$ to values in the range of $\frac{\GA{}}{n-2}$.

To test whether $\WA{\alpha}$ is $(\epsilon,\alpha)$-connected, we can use the test
\be
\min_{v\in V} |\{ w\in V\,:\, \GA{}(w,v)\geq \alpha (n-2)\} > \epsilon n.
\ee

To test whether $\WA{\alpha}$ is $B$-separated or has an $\epsilon$-split is harder, since we do not have the ordering of the vertices.
To test whether $\WA{\alpha}$ is $B$-separated, we could define for any vertex $v$ the vector
\be 
s_{v}[u] = |\{w \in V \, : \, \GA{}(v,w) > \alpha(n-2) > \GA{}(u,w) \text{ or } \GA{}(v,w) < \alpha(n-2) < \GA{}(u,w) \}.
\ee 
This vector should be close to the volume in the definition of $B$-separation. Let $\tilde{s}_{v}$ be the elements of $s_{v}$, sorted in increasing order. If $\WA{\alpha}$ is $B$-separated, then we would expect the inequality
\be 
\tilde{s}_{v}[j] -\tilde{s}_{v}[i] \geq B(j-i)
\ee 
to hold for typical $v \in V$ and $1 \leq i < j \leq n$ (and to be close to true for all such values). 

If $\WA{\alpha}$ has an $\epsilon$-split, then pairs of vertices at extreme ends of the ordering cannot have any common neighbours. Precisely, for pairs of vertices $i,j$ with $|U_i-U_j|\geq 1-\epsilon$, $N(i) \cap N(j) = \emptyset$. We expect there to be about $\epsilon^2n^2/2$ such pairs. On the other hand, 
if 
$\inf_{|x|, |y| < \delta} \WA{}(U_{i} + x, U_{j}+y) > 0$ for some $\delta >0$, then by the law of large numbers $N(i) \cap N(j) \neq \emptyset$ for sufficiently large $n$. This suggest that we should check 
\[ 
| \{ (i,j) \, : \, N(i) \cap N(j) = \emptyset \}| \geq \epsilon^{2}n^2/2.
\]

It is easy to integrate these tests in the algorithm, so that a suitable value of $\alpha$ can be found in a reasonable amount of trials.


We note that we have not searched for a value of $\delta$ that satisfies Assumption \ref{AssumptionSharpBoundary}, as we don't need to know $\delta$ to run Algorithm \ref{AlgIterationHolder}. In practice, we expect that it is often easy to diagnose graphons that don't satisfy Assumption \ref{AssumptionSharpBoundary} as follows. Let $F^{(2)}$ be as in step (12) of Algorithm \ref{AlgRefine} on the first time that it is called by Algorithm \ref{AlgIterationHolder}; if  Assumption \ref{AssumptionSharpBoundary} is satisfied, then with extreme probability the maximum size of the set on which disagreement occurs
\be 
\max_{i,j \in V_{2} \backslash V_{1} \, : \, F^{(2)}(i,j) = -1} | \{ k \in V_{2} \backslash V_{1} \, : \, F^{(2)}(i,k) = F^{(2)}(k,j) = 1\}|
\ee 
will not be much larger than $d_{2} n$. In other words, we can diagnose a failure of Assumption \ref{AssumptionSharpBoundary} by simply checking whether $F^{(2)}$ could plausiblY have precision better than $d_{2}$ or so.

\section{Large Error of Embeddings} \label{SecLargeError}

It is common to study the seriation problem using something like the following two-step procedure:

\begin{enumerate}
\item Compute an estimate $\hat{U} = (\hat{U}_{1}, \hat{U}_{2},\ldots,\hat{U}_{n})$ of the latent positions $U = (U_{1},U_{2} \ldots,U_{n})$ of the vertices $1,2,\ldots,n$, using \textit{e.g.}  a spectral embedding of the graph $G$.
\item Compute an estimate $\hat{\sigma}$ of $\sigma_{\mathrm{true}}$ based \textit{only} on $\hat{U}$, typically using the formula:
\be
\hat{\sigma}(i) = | \{ j \in [1:n] \, : \, \hat{U}_{j} \leq \hat{U}_{i} \} |.
\ee
\end{enumerate}

It is then straightforward to show that, if $\hat{U}$ has small error, the error of the induced ordering $\hat{\sigma}$ will inherit a similarly-small error. See \textit{e.g.} \cite{little2018analysis} for recent work that conducts this sort of analysis.

This approach can only give a near-optimal estimate of the error of $\hat{\sigma}$ if the error of an optimal estimator $\hat{\sigma}$ is not much smaller than the error of an optimal estimator $\hat{U}$ (after appropriate scaling). We will show in this section that this is rarely the case for the seriation problem studied in this paper.

To make this precise, we need some notation that is not used in the rest of the paper. For a graphon $w$ and sequences $u^{(1)} = \{u[j]^{(1)}\}_{j=1}^{n}$, $u^{(2)} = \{u[i,j]^{(2)}\}_{i,j=1}^{n}$ define $G = G(w; u^{(1)}, u^{(2)})$ to be the usual graph obtain from the graphon $w$ and these sequences from the formula \eqref{EqGraphonDef}.

Next, for constant $\delta \in [0,1)$ and graphon $w$, define
\be
\tilde{w}_{\delta}(x,y) = w(x^{1-\delta}, y^{1-\delta}).
\ee
Similarly, for any sequence $u \in \mathbb{R}^{d}$, define
\be
\tilde{u}_{\delta} = u^{1-\delta}.
\ee
Notice that taking the $(1-\delta)$'th power is a monotone bijection on $[0,1]$, and in particular these transformations preserve both (i) the diagonally-increasing property of $w$ and (ii) the ordering of $u$. We have the following theorem:

\begin{thm} \label{ThmNoDownstreaming}
Fix $\epsilon > 0$. There exist constants $c = c(\epsilon), c' = c'(\epsilon), c'' = c''(\epsilon) > 0, $ and couplings of two sequences $U = (U_{1},\ldots,U_{n}), V = (V_{1},\ldots,V_{n}) \stackrel{i.i.d.}{\sim} \mathrm{Unif}([0,1])$ such that, for all $n$ sufficiently large, the following events all occur simultaneously with probability at least $(1-\epsilon)$:

\begin{enumerate}
\item For all graphons $w$ and sequences $u^{(2)} = \{u_{i,j}^{(2)}\}_{i,j=1}^{n}$,
\be \label{ConcEqualGraphs}
G(w; V, u^{(2)}) = G\left(\tilde{w}_{\frac{c}{\sqrt{n}}}; U, u^{(2)}\right).
\ee
\item There exists a single permutation $\sigma \in S_{n}$ such that
\be \label{ConcEqualPerms}
U_{\sigma(1)} < U_{\sigma(2)} < \ldots < U_{\sigma(n)}, \, V_{\sigma(1)} < V_{\sigma(2)} < \ldots < V_{\sigma(n)};
\ee
that is, the two sequences have the same ordering.
\item Finally,
\be \label{ConcBigDisp}
\left| \left\{1 \leq i \leq n \, : \, |U_{i} - V_{i}| \geq \frac{c'}{\sqrt{n}} \right\} \right| \geq c'' n.
\ee
\end{enumerate}
\end{thm}

\begin{proof}
See Appendix \ref{SecAppAntiConc}.
\end{proof}

We give the following informal interpretation: it is possible to couple samples $G, \tilde{G}$ from the two graphons $w, \tilde{w}_{\frac{c}{\sqrt{n}}}$ so that with high probability both (i) the graphs themselves are identical, but (ii) a positive fraction of the ``true" latent vertex embeddings  are at least $\frac{c'}{\sqrt{n}}$ apart. In particular, since the observed graphs $G, \tilde{G}$ are identical, there is no way to simultaneously estimate both of ``true" latent vertex embeddings without a positive fraction of errors being at least $\frac{c'}{\sqrt{n}}$.

In principle this conclusion could be avoided if \textit{e.g.} $w$ is assumed to belong to a class that does not include $\tilde{w}_{\delta}$, even for very small values of $\delta$. While possible, such an assumption seems to be \textit{extremely} strong - indeed, we have shown that sequences of samples from $w$, $\tilde{w}_{\frac{c}{\sqrt{n}}}$ can't be told apart by looking at their associated graphs! Furthermore, even if the \textit{particular} perturbation $\tilde{w}_{\delta}$ can be ruled out, a quick inspection of the proof of Theorem \ref{ThmNoDownstreaming} will convince the reader that similar results hold for a variety of other small perturbations of $w$.

\appendix

\section{Proof of Theorem \ref{ThmReconMainRes1}} \label{SubsecWeakStrongAssumptionsCheck}

In this section we show that Theorem \ref{ThmMaxOrd} implies Theorem \ref{ThmReconMainRes1} by checking that Assumption \ref{AssumptionsSimpleWeakIdentifiabilityAssumptions} implies Assumption \ref{AssumptionGoodGraphon}. We will be using notation from those assumptions throughout this section, so we repeat them for ease of reference. Assumption \ref{AssumptionsSimpleWeakIdentifiabilityAssumptions}:

\begin{assumption} 

Let $w$ be a uniformly embedded graphon with link probability function $f:[0,1]\rightarrow [0,1]$. In addition, $f$ is decreasing, and there exist constants $0<d <0.5$ and $0 \leq c<f(0)$ so that $f(z)=c$ for all $z\geq d$. Finally, there exists $\alpha \in (0,1)$ so that $f$ satisfies
\be\label{eq:alpha1}
\inf_{s\in [0,d]}\int_0^1 f(z)f(|s-z|) dz  > \alpha > \int_0^1 f\left( \left|z-\frac{1-d'}{2}\right|\right) f\left(\left|z-\frac{1+d'}{2}\right|\right) dz,
\ee
where $d'=\min\{ 0.5,2d\}$.
\end{assumption}

Assumption \ref{AssumptionGoodGraphon}: 

\begin{assumptions} [Weak Assumptions] 
Let $w$ be a graphon, and $\alpha, \epsilon \in (0,1)$ and $A > 0$ be constants, so that:
\begin{enumerate}
\item $\WA{\alpha}$ is diagonally increasing, and
\item $w^{(2)}$ is uniformly $(A,\epsilon)$-good at $\alpha$, and
\item $\WA{}$ is $(\epsilon ,\alpha )$-connected, and
\item $\WA{\alpha}$ has an $\epsilon$-split, and
\item $\WA{\alpha}$ is $\epsilon $-separated.

\end{enumerate}
\end{assumptions}

Repeating the details of these five component assumptions would take quite a bit of space; we refer to Section \ref{SectionSlightlyStronger} for details.

\subsubsection{Some Basic Lemmas}\label{Sec:BasicLemmas}

Under Assumption \ref{AssumptionsSimpleWeakIdentifiabilityAssumptions}, $w$ is a uniformly embedded graphon, with decreasing probability function $f$. We also assume that, for some value $d<0.5$, $f(x)=c$ for $x\geq d$. In the following, we assume \emph{wlog} that $d$ is minimal with respect to this property, so $f(x)>c$ for $x<d$. We now show that for uniformly embedded graphons, $\WA{}$ is continuous.

\begin{lemma}
\label{Lemma:SquareContinuous}
Let $w$ be a uniformly embedded graphon, with decreasing link probability function $f$. Then $w^{(2)}$ is uniformly continuous.
\end{lemma}

\begin{proof}
Fix $ 0 < \epsilon < 1$, and let $\delta =(\epsilon/12)^2$. Let $D=\{ x\,:\, f(x)-f(x+\delta)\geq  \epsilon/6\}$ be the set of points where  $f$ makes a big jump. Since the range of $f$ is contained in $[0,1]$, and $f$ is decreasing, $D$ can include at most $\frac{6}{\epsilon}$ points that are mutually at distance at least $\delta$ from each other. Thus, $D$ is contained in the union of at most $\frac{6}{\epsilon}$ intervals of length $2\delta$, and thus $\text{Vol}(D) \leq  (6/\epsilon)(2\delta)=\epsilon/12$.  Fix $0\leq y\leq 1-\delta$,
 and let 
\[
D'=\{ z<y: y-z \in D \},\quad D''=\{ z>y\,:\,z-y-\delta\in D \}.
\]
Then $\text{Vol}(D'), \, \text{Vol}(D'') \leq \text{Vol}(D)$. Note that, for all $0 \leq x \leq y$,
\be
&|w^{(2)}(x,y+\delta)-w^{(2)}(x,y)|\\
&\leq \int_0^{y-\delta} f(|z-x|)\left(f(y-z) -f(y-z+\delta)\right) dz \\
&+ \int_{y+\delta}^1 f(|z-x|)\left(f(z-y-\delta)-f(z-y)\right)dz + 2\delta.
\ee
Moreover,
\be
&\int_0^{y-\delta} f(|z-x|)\left(f(y-z)-f(y-z+\delta)\right) dz \\
&\leq  \int_{D'}f(|z-x|)\left(f(y-z)-f(y-z+\delta)\right) dz+\int_{[0,y-\delta]\setminus D'}f(z-x)\frac{\epsilon}6 dz  \\
&\leq \text{Vol}(D^\prime) +\left( \frac{\epsilon}6 \right) y\text{, and}\\
&\int_{y+\delta}^1 f(|z-x|)\left(f(z-y-\delta)-f(z-y)\right)dz \leq \text{Vol}(D^{\prime\prime})+\left(\frac{\epsilon}6\right)(1-y).
\ee
Therefore, $|w^{(2)}(x,y+\delta)-w^{(2)}(x,y)|\leq \text{Vol}(D')+\text{Vol}(D'')+\frac{\epsilon}{6}+2\delta <\epsilon/2$.

Similarly, we can show that $|w^{(2)}(x,y-\delta)-w^{(2)}(x,y)|< \epsilon/2$. Using the fact that $w^{(2)}$ is symmetric, this shows that $|w^{(2)}(x',y')-w^{(2)}(x,y)|<\epsilon $ whenever $||(x,y)-(x',y')||_\infty<\delta$, so $w^{(2)}$ is uniformly continuous.
\end{proof}

Next we consider an alternative representation of diagonally increasing graphons, which will be helpful in some cases. We saw earlier (see discussion before Equation \eqref{Eq:Boundaries}) that, if  $w$ is a $\{ 0,1\}$-valued graphon, then $w$ is diagonally increasing if and only if there exist two non-decreasing boundary functions $\ell:[0,1]\rightarrow [0,1]$ and $r:[0,1]\rightarrow [0,1]$ which demarcate the regions in $[0,1]^2$ where $w$ assumes the value 1. Precisely, $w$ is  of the form
\be
\label{AltRep01w}
w(x,y)=\left\{ \begin{array}{ll} 1&\text{if }y\in I(x)\\
0&\text{otherwise, }\end{array}\right.
\ee
where $(\ell (x),r(x)) \subset  I(x)\subset [\ell(x),r(x)]$. Thus, the boundaries define $w$ up to the values exactly on points of the form $(\ell(x),x)$ and $(x,r(x))$; of course this collection of points has measure 0, and thus does not influence the distribution of samples from $w$.

More generally, $w$ is diagonally increasing if and only if, for each $s\in [0,1]$, the $\{ 0,1\}$-valued graphon $w_s$ is diagonally increasing. Thus, if $w$ is diagonally increasing, then for each $s\in [0,1]$, there exist boundaries $\ell_s$ and $r_s$ that define the thresholded graphon $w_s$ (up to a set of measure zero) as in \eqref{AltRep01w}. Let
\be \label{Eq:Ialpha}
I_s(x)=\{ y:w_s (x,y)=1\}=\{ y:w(x,y)\geq s\}
\ee
as given above. This leads to an alternative representation of $w$:

\begin{lemma}
For any diagonally increasing graphon $w$,
\be \label{AltRepw}
w(x,y)=\int_{0}^{1} \mathbf{1}_{\{y\in I_{s}(x)\}}ds.
\ee

\end{lemma}

\begin{proof}

For all $0\leq x<y\leq 1$, by definition of $I_s$,
\be
w(x,y)=\int_0^1 \mathbf{1}_{\{s \leq w(x,y)\}} ds =\int_0^1 \mathbf{1}_{\{ y\in I_s (x)\}}ds.
\ee
\end{proof}

The representation of $w$ as given in (\ref{AltRepw}) leads to the alternative formula of $w^{(2)}$ :
\be\label{Eq:w2Alternative}
w^{(2)}(x,y)&=\int_0^1\left(\int_s \mathbf{1}_{\{ z\in I_s(x)\}}ds\right)\,\left(\int_t \mathbf{1}_{\{ z\in I_t(y)\}}dt\right) dz\\
&=\int_s\int_t |I_s(x)\cap I_t(y)|dsdt,
\ee
where $I_s(x)$, $I_t(y)$ are as defined in (\ref{Eq:Ialpha}).

In the case of graphons satisfying Assumptions \ref{AssumptionGoodGraphon}, the length of the interval $I_s(x)$ is determined only by $s$. Precisely, for all $s,x\in [0,1]$, $(x-d_s,x+d_s)\subset I_s(x)\subset [x-d_s,x+d_s]$, where
\be
d_s=\sup\{ \delta\,:\,f(\delta )\geq s\}.
\ee 
Under Conditions \ref{AssumptionsSimpleWeakIdentifiabilityAssumptions}, we also know that $d_s\leq d$ for all $s>c$, and $d_s=1$ otherwise.

In Lemma \ref{lem:decreaseW2}, to follow, we use this alternative representation to show a crucial property of $\WA{}$. First, we make a simple observation, which will help us to conclude the second half of the lemma. 
\be\label{Eq:W1-x}
\WA{}(1-x,1-y)&=\int_0^1 f(|1-x-z|)f(|1-y-z|)dz\\
& =\int_0^1 f(|\zeta-x|)f(|y-\zeta|) d\zeta \\
& =\WA{}(x,y).
\ee

\begin{lemma}\label{lem:decreaseW2}
Let $w$ be a graphon satisfying Conditions \ref{AssumptionsSimpleWeakIdentifiabilityAssumptions}, and let $f,d,c,\alpha$ be as in the statement of these conditions. Then for all $x\in [0,1]$, 
\begin{itemize}
\item[(i)] $w^{(2)}(x,\cdot)$ is decreasing on $[x+d,1]$, 
\item[(ii)] for all $x+d<y'<y<x+2d$,
\be\label{eq:minslopeW2}
(y-y')\geq \WA{}(x,y')-w(x,y)\geq (y-y') \frac12\left(f(\frac{y-x}{2})-c\right)^2.
\ee
In particular,  $w^{(2)}(x,\cdot)$ is strictly decreasing on $(x+d,x+2d)$.
\item[(iii)]  $w^{(2)}(x,\cdot)$  is increasing on $[0,x-d]$.

\item[(iv)] For all $x-2d<y<y'<x-d$,
\be\label{eq:minslopeW2symm}
(y'-y)\geq \WA{}(x,y')-w(x,y)\geq (y'-y) \frac12\left(f(\frac{x-y'}{2})-c\right)^2.
\ee
In particular,  $w^{(2)}(x,\cdot)$ is strictly increasing on $(x-2d,x-d)$.
\end{itemize}
Moreover, since $w$ is symmetric, equivalent statements hold for $\WA{}(\cdot, x)$.
\end{lemma}

\begin{proof}

Inspecting the alternative representation given in \eqref{Eq:w2Alternative}, we see that the contribution to $\WA{}(x,y)$ due to any particular pair of values $s,t$ is determined by the size of the intersection $I_s(x)\cap I_t(y)$. We can derive the size of this intersection  from the values of $d_s,d_t$.
Fix $x\in [0,1-d)$, and consider $w^{(2)}(x,y)$ as a function of $y\in [x,1]$. For $s,t\in [0,1]$, define
\be
\label{Def:phi}
\phi_{s,t} (y)&= |I_s(x)\cap I_t(y)|.
\ee

The set $I_s(x)\cap I_t(y)$ has measure zero if and only if $x+d_s\leq y-d_t$ (i.e. $d_s+d_t\leq y-x$). If $|I_s(x)\cap I_t(y)|>0$, then 
\be 
\phi_{s,t}(y)=|I_s(x)\cap I_t(y)| = \min\{ x+d_s, y+d_t, 1\}-\max\{ x-d_s,y-d_t, 0\}.
\ee
Thus $\phi_{s,t}$ is piecewise linear, with each of the linear pieces having slope in $\{ -1,0,1\}$.

We are almost ready to prove that $w^{(2)}(x,\cdot)$ is decreasing on $[x+d,1]$, based on the following proposition:

\begin{prop} \label{PropNPSlope}
Fix $x \in [0,1]$. Then for all $s,t \in [0,1]$, $\phi_{s,t}$  has non-positive slope for almost every $y \in (x+d,1]$.
\end{prop}

\begin{proof}
For $y \in (x,1]$, define
\be
R^+(y)=\{ (s,t)\,:\,y<d_t<1-y,\,d_s> d_t+(y-x)\}.
\ee

Suppose that for some $s,t \in [0,1]$, $y$ lies in the interior of a region on which $\phi_{s,t}$ has slope 1; we will check that $(s,t)\in R^+(y)$. So see this, note that, if  $\phi_{s,t}$ has slope 1 at $y$, then the left boundary of $I_s(x)\cap I_t(y)$ cannot be given by $y-d_t$, while the right boundary must be $y+d_t$. The right boundary condition (and the fact that we are in the interior of a region with slope 1) implies that $y+d_t<1$ and $y+d_t<x+d_s$. This implies that $d_s> d_t+(y-x)\geq d_t$, and thus $x-d_s\leq y-d_t$. Thus we must have that the left boundary of $I_s(x)\cap I_t(y)$ equals 0, and thus $y-d_t<0$, and so $(s,t)\in R^+(y)$.

Next, we check that, if $y>x+d$, then $R^+(y)=\emptyset$. As argued earlier, it follows from Conditions \ref{AssumptionsSimpleWeakIdentifiabilityAssumptions} that for all $s$, $d_s\leq d$ or $d_s=1$ otherwise. Suppose then that  $y-x>d$ and $(s,t)\in R^+(y)$. By definition, $d_s>d_t+(y-x)>d$ and thus $d_s=1$. But then, $d_t<1-y<1$ so $d_t\leq d\leq x+d<y$, which contradicts the condition $y>x+d$. 

Combining the conclusions in these two paragraphs completes the proof.

\end{proof}

By Proposition \ref{PropNPSlope} if $x+d<y'<y$, then $\phi_{s,t}(y)-\phi_{s,t}(y')\leq 0$, and consequently,
\[
w^{(2)}(x,y)-w^{(2)}(x,y')=\int_s\int_t \phi_{s,t}(y)-\phi_{s,t}(y') dsdt\leq 0.
\]
This shows part $(i)$ of Lemma \ref{lem:decreaseW2}, namely that $\WA{}(x,\cdot )$ is decreasing on $[x+d,1]$.

We will next prove Inequality \eqref{eq:minslopeW2}. The lower bound  follows immediately from the fact that for all $s,t$, $\phi_{s,t}$ has slope at least $-1$. To prove the upper bound, we begin by defining
\be
R^-(z)=\{(s,t)\,:d_t\leq d_s\leq d, d_s+d_t\geq z-x\}
\ee
for all $z \geq x$. Note that $R^-(y)\subset R^-(z)$ for all $y>z$.

We now check that 
\be \label{IneqPhiNegSlope1}
\phi_{s,t}(y)-\phi_{s,t}(y')=-(y-y')
\ee
holds for all $x+d<y'<y<x+2d$ and $(s,t)\in R^-(y)$. We will do this by showing that, for all $(s,t)\in R^-(y)$,  
$\phi_{s,t}$ has slope $-1$ at any point $z\in [y',y]$.
Fix $(s,t)\in R^-(y)$ and $z\in [y',y ]$. Then $(s,t)\in R^-(z)$. The condition $ d_s+d_t\geq z-x$ in the definition of $R^{-}(z)$ guarantees that $I_s(x)\cap I_t(z)$ is non-empty. Since $d_s\geq d_t$ and $x<z$, $x-d_s < z-d_t$. Since $z> x+d\geq d\geq d_t$, $z-d_t>0$. Thus, the left boundary of $I_s(x)\cap I_t(z)$ equals $z-d_t$. On the other hand, $z+d_t\geq z>x+d\geq x+d_s$ so the right boundary is not given by $z+d_t$. Thus, $\phi_{s,t}$ has slope $-1$ at $z$. This implies that $\phi_{s,t}$ is linear with slope $-1$ on the interval $[y',y]$, which implies  \eqref{IneqPhiNegSlope1}.

%
Finally, we bound the size of $R^-(y)$ for $x+d<y<x+2d$. Fix $y\in (x+d,x+2d)$, and define $c_1 \equiv f(\frac{y-x}{2})$. In this case, $(y-x)/2<d$ and thus $c_1 >c$. Then for all $c_1\geq t\geq s>c$ and for all $y'\leq z\leq y$, $(y-x)/2\leq d_t\leq d_s<d$, and thus $(s,t)\in R^-(y)$. Therefore, $|R^-(z)|\geq (c_1-c)^2/2$. 

Applying Equations \eqref{Eq:w2Alternative} and \eqref{Def:phi} and then Inequality \eqref{IneqPhiNegSlope1},  we have for $x+d<y'<y<x+2d$,
\be
\label{Eq:DerivativeSquare}
w^{(2)}(x,y)-w^{(2)}(x,y')&=\int_s\int_t (\phi_{s,t}(y)-\phi_{s,t}(y')) dsdt \\
&\leq  - (y-y') |R^-(y)| \\
&\leq -\frac12 (f(\frac{y-x}{2})-c)^2 (y-y').
\ee
In particular, $w^{(2)}(x,\cdot)$ is strictly decreasing on $(x+d,x+2d)$. This proves the upper bound in \eqref{eq:minslopeW2}, completing our proof of item $(ii)$ of the lemma.

Due to the symmetry shown in \eqref{Eq:W1-x} we have that $(iii), (iv)$ follow immediately from $(i), (ii)$ respectively.

\end{proof}

Finally, we establish a general lemmas about the behaviour of $w^{(2)}(x,y)$ when $w$ is uniformly embedded and diagonally increasing.

\begin{lemma}
\label{Lemma:DiagonalValues}
Let $w$ be a uniformly embedded graphon, with decreasing link probability function $f$. Fix $\delta\in [0,\frac12]$. Then for all $x\in [0,1-\delta]$,
\be
\label{Eq:DiagonalValues}
w^{(2)}(0,\delta)\leq w^{(2)}(x,x+\delta)\leq w^{(2)}(\frac{1-\delta}2,\frac{1+\delta}2).
\ee
\end{lemma}

\begin{proof}
As shown in Equation \eqref{Eq:W1-x}, we have that $w^{(2)}(1-x,1-y)=w^{(2)}(x,y)$ for all $x,y$. Thus we may assume without loss of generality that $x\leq 1-(x+\delta)$, so $x\in [0, \frac12-\frac{\delta}2]$. According to the definition, and making the appropriate substitutions,
\be 
w^{(2)}(x,x+\delta)&=\int_0^1 f(|x-z|)f(|x+\delta-z|)dz\\
&=\int_0^x f(s)f(s+\delta)ds+\int_0^\delta f(s)f(\delta-s)ds+ \int_0^{1-\delta-x} f(s)f(s+\delta)ds.\\
\ee
Note that the middle integral does not depend on $x$, while the first and last integral are taken over the same function. Since
$f$ is decreasing, so is $f(s)f(s+\delta)$ (taken as a function of $s$). Therefore, this expression is maximized when
$x=1-\delta-x$, so $x=\frac{1-\delta}2$, and minimized when $x=0$.
\end{proof}

\subsubsection{Assumption \ref{AssumptionsSimpleWeakIdentifiabilityAssumptions} implies Assumption \ref{AssumptionGoodGraphon}}

We start by showing that Assumption \ref{AssumptionsSimpleWeakIdentifiabilityAssumptions} implies that $\WA{\alpha}$ is diagonally increasing. 

\begin{remark}\label{rem:counterexample}

It is not true that $w^{(2)}$ is diagonally increasing for all diagonally-increasing $w$. We consider as simple counterexamples some graphons of the form \eqref{EqSimpleGraphon}, which are uniformly embedded with  link probability function
\[
f(x)=\left\{ \begin{array}{ll}p,& 0\leq x\leq d,\\
q, & d<x\leq 1.
\end{array}\right.
\]
We consider parameters $0\leq q<p\leq 1$, and $0<d<\frac12$.
Then, for $y\in [0,d]$,  $w^{(2)}(0,y)= q^2+(p^2-q^2)d+q(p-q)y$. This function is increasing as $y$ moves away from 0 and thus is further removed from the diagonal. This contradicts the defining property of diagonally increasing graphons as given in (\ref{IneqEmbBasic}).

Even if $w^{(2)}$ itself is not diagonally increasing, there may still exist a value $\alpha $ so that $\WA{\alpha}$ is diagonally increasing. In the above example, for $y\in [d,2d]$, $w^{(2)}(0,y)$ is linearly decreasing with slope $-(p-q)^2$, while for $y\in [2d,1-d]$, $w^{(2)}(0,y)$ is constant equal to $q^2+3dq(p-q)$. If there exists a value $w^{(2)}(0,0)>\alpha> w^{(2)}(0,2d)$ then  $\WA{\alpha}(0,\cdot) $ is  decreasing, even if $w^{(2)}(0,\cdot)$ is not. However, it is possible to choose values for $p$ and $q$ (e.g.~$p=1/3$, $q=1/6$, $d=0.3$) for which $\WA{}(0,2d)\geq \WA{}(0,0)$, and thus such a value $\alpha$ does not exist.

\end{remark}

We now show that, for a graphons $w$ satisfying our assumptions, $\WA{\alpha}$ is a diagonally increasing graphon.

\begin{lemma}\label{lem:wLinEmbw2Also}
Let $w$ and $\alpha$ be a graphon and constant satisfying Conditions \ref{AssumptionsSimpleWeakIdentifiabilityAssumptions}. Then $w^{(2)}_{\alpha}$ is diagonally increasing. Moreover, $w^{(2)}_{\alpha}(x,y)=1$ if and only if $y\in [\ell_\alpha(x),r_\alpha (x)]$. 

\end{lemma}

\begin{proof}

Fix $x\in [0,1]$. By Lemma \ref{Lemma:SquareContinuous}, $\WA{}$ is continuous, so $\WA{}(x,r_\alpha (x))=\alpha$.
By Assumptions \ref{AssumptionsSimpleWeakIdentifiabilityAssumptions} and Lemma \ref{Lemma:DiagonalValues}, we have that  for all $0\leq \delta\leq d$,
\[
\WA{}(x,x+\delta)\geq \WA{}(0,\delta)>\alpha,
\]
and thus $ r_\alpha (x)> x+d$.
By Lemma \ref{lem:decreaseW2}, $\WA{}(x,\cdot)$ is decreasing on $[x+d,1]$, and thus $\WA{}(x,y)\geq \WA{}(x,r_\alpha (x))=\alpha$ for $y\in [d,r_\alpha (x)]$. Thus for $y\geq x$, $\WA{}(x,y)\geq \alpha $ if and only if  $y\leq r_\alpha (x)$.

An analogous argument applied to $\ell_\alpha (x)$ shows that, for $y\leq x$, $\WA{}(x,y)\geq \alpha $ if and only if  $y\geq \ell_\alpha (x)$. This completes the proof.

\end{proof}

The following lemma shows that $r_\alpha (x)-x$ and $x-\ell_\alpha(x)$ are bounded well away from their lower and upper bounds $d$ and $d'$. 

\begin{lemma}\label{Lemma:d-d+}
Let $w$ be a graphon satisfying Conditions \ref{AssumptionsSimpleWeakIdentifiabilityAssumptions}, and let $f,d,c,\alpha$ be as in the statement of these conditions. Let
\be\label{eq:d-d+}
d^{-} &=\inf \{ r_\alpha(x)-x\,:\, x\in [0,\ell_\alpha(1)]\},\text{ and}\\
d^+ &=\sup \{ r_\alpha(x)-x\,:\, x\in [0,1]\}\}.
\ee
Then $d<d^-\leq d^+<d'$, where $d'=\min\{ 2d, 0.5\}$ as in Assumption \ref{AssumptionsSimpleWeakIdentifiabilityAssumptions}.
\end{lemma}

\begin{proof}
By our choice of $\alpha$, we have that
\be \label{IneqRAlphaSimp}
r_\alpha(x) \geq \min(x+d,1)
\ee for all $x\in [0,1]$. Let $\gamma = \inf\{ \WA{}(0,s)\,:\, 0\leq s \leq d\}$. By Assumption \ref{AssumptionsSimpleWeakIdentifiabilityAssumptions}, 
\[
\gamma =\inf_{s\in [0,d]}\WA{}(0,s)=\inf_{s\in [0,d]}\int_0^1 f(z)f(|s-z|) dz  > \alpha.
\]
By Lemma \ref{Lemma:DiagonalValues}, this implies that, for all $x\in [0,1-d]$,
$
\WA{}(x,x+d)\geq \WA{} (0,d)\geq \gamma >\alpha.
$

By Lemma \ref{Lemma:SquareContinuous}, $\WA{}$ is uniformly continuous. Let $0<\epsilon < \gamma -\alpha $, and let $\delta >0$ be so that $\WA{}(x,y)< \epsilon$ whenever $|x-y|<\delta$. Then for all $x\in [0,1-d-\delta/2]$, 
$$
\WA{}(x,x+d+\delta/2)>\WA{}(x,x+d)+\epsilon\geq \gamma +\epsilon >\alpha.
$$
Therefore, $r_\alpha (x)\geq x+d+\delta/2$ for all $x\in [0,1-d-\delta/2]$. Moreover, since $\WA{}(1-d-\delta/2,1)>\alpha$, $\ell_{\alpha }(1)<1-d-\delta/2$. By compactness, $d^- \geq d+\delta/2>d$.

Let $\gamma' = \WA{}(\frac{1-d'}{2},\frac{1+d'}{2})$.
By Assumption \ref{AssumptionsSimpleWeakIdentifiabilityAssumptions}, 
$$
\gamma'=\int_0^1 f(|z-\frac{1-d'}{2}|) f(|z-\frac{1+d'}{2}|) dz <\alpha.
$$
By Lemmas \ref{Lemma:DiagonalValues} and \ref{lem:decreaseW2}, we have that  for all $x\in [0,0.5]$,
\[
\WA{}(x,x+d')\leq \WA{}(\frac{1-d'}{2},\frac{1+d'}{2})=\gamma' <\alpha,
\]
and thus $ r_\alpha (x)< x+ d'$.
By compactness, $d^+\leq d'$. As argued in the previous case, since $\WA{}(x,r_\alpha (x))=\alpha $ and $\WA{}$ is continuous, we have that $d^+<d'$.

\end{proof}

\begin{lemma}\label{Lemma:Connected}

Let $w$ and $\alpha$ be a graphon and constant satisfying Conditions \ref{AssumptionsSimpleWeakIdentifiabilityAssumptions}.  Then there exists an $\epsilon >0$ so that $\WA{\alpha}$ is $(\epsilon ,\alpha)$-connected.
\end{lemma}

\begin{proof}
Define $d^{-}$ as in the statement of Lemma \ref{Lemma:d-d+} and take $0<\epsilon \leq d^{-}$. If $0\leq x\leq y\leq 1$ and $y-x<\epsilon$, then either $x\in [0,\ell_\alpha(1)]$ and $y< x+\epsilon\leq r_\alpha(x)$, or $x>\ell_\alpha (1)$. In both cases, $\WA{}(x,y)>\alpha$. Since $\WA{}$ is continuous (Lemma \ref{Lemma:SquareContinuous}), the result follows.
\end{proof}

\begin{lemma}\label{Lemma:Split}
Let $w$ and $\alpha$ be a graphon and constant satisfying Conditions \ref{AssumptionsSimpleWeakIdentifiabilityAssumptions}. Then there exists an $\epsilon >0$ so that  so that $\WA{\alpha}$ has an $\epsilon$-split.
\end{lemma}

\begin{proof}
Define $d^{+}$ as defined in \eqref{eq:d-d+} and let $\epsilon =0.5-d^+ $. 
Fix $x<y$ so that $y-x\geq 1-\epsilon$. Then
\[
r_\alpha (x)+\epsilon\leq x+d^+ +\epsilon \leq y-(1-\epsilon) + d^+ +\epsilon = y-1+2\epsilon +d^+ = y-d^+\leq \ell_\alpha (y),
\]
and thus $r_\alpha(x)<\ell_\alpha(y)$. Thus $\WA{\alpha}$ has an $\epsilon$-split.

\end{proof}

\begin{lemma}
\label{Lemma:Goodness}
Let $w$ be a graphon satisfying Conditions \ref{AssumptionsSimpleWeakIdentifiabilityAssumptions}, and let $f,d,c,\alpha$ be as in the statement of these conditions.
Then  there exist $A,\delta >0$ so that $\WA{}$ is uniformly $(A,\delta )$-good at $\alpha$.
\end{lemma}

\begin{proof}
We need to show that there exist constants $\delta,A>0$ so that, for all $x\in [0,1]$ and $0<\delta'<\delta$, the size of the set where $\WA{}$ takes values within $\delta'$ of $\alpha$ is bounded by $A\delta'$. 

Let $d^{-}, d^{+}$ be as in Equation \eqref{eq:d-d+}. 
By our earlier assumption on the minimality of $d$ (see beginning of this Section \ref{Sec:BasicLemmas}), we have that $f(x)>c$ whenever $x<d$.  Let
\be\label{eq:beta}
\beta = \frac12( f(d^-  /2) -c)^2.
\ee
We first show that $\beta >0$. By Lemma \ref{Lemma:d-d+}, $d^-<d'\leq 2d$, and thus $d^-/2<d$. It follows from our assumptions that $f(d^-/2)>c$, so $\beta > 0$.

By Lemma \ref{lem:decreaseW2} and the fact that $f$ is decreasing, for all $x\in [0,1]$, $x+d^-\leq y<x+2d$, and all $x+d<y'<y$,
\be \label{eq:slope}
\WA{}(x,y')-\WA{}(x,y)\geq (y-y')\frac12\left(f \left(\frac{y-x}{2} \right)-c\right)^2\geq \beta (y-y').
\ee

Let  $ \epsilon >0$ be so that $d^- -\epsilon >d$ and $d^+ +\epsilon < d' \leq 2d$; such a value exists by Lemma \ref{Lemma:d-d+}. Fix $0<\epsilon' <\epsilon$.
Taking $y=r_\alpha (x)$ and $y'=r_\alpha (x)-\epsilon'$ in (\ref{eq:slope}), we have that
\be\label{eq:left}
\WA{} (x,r_\alpha (x)-\epsilon')\geq  \WA{}(x,r_\alpha (x))+ \beta \epsilon' =\alpha +\beta\epsilon',
\ee
where the equality uses the fact that $\WA{}$ is uniformly continuous (see Lemma \ref{Lemma:SquareContinuous}).  Similarly, taking $y=r_\alpha (x)+\epsilon'$ and $y'=r_\alpha (x)$, we obtain that
\be\label{eq:right}
\WA{} (x,r_\alpha (x)+\epsilon')\leq  \alpha - \beta\epsilon'.
\ee

The above, together with the fact that $\WA{}$ is diagonally increasing shows that, for values $y$ outside of the interval $(r_\alpha (x)-\epsilon, r_\alpha(x)+\epsilon)$, $\WA{}(x,y)$ takes values that differ from $\alpha$ by at least $\beta |x-y|$.  This leads to our definition of $\delta$ and $A$ for which $\WA{}$ is uniformly $(A,\delta )$-good at $\alpha$.

Namely, let $\delta = \epsilon \beta $ and $A=2/\beta$. Fix any $0<\delta'<\delta$. By setting $\epsilon'=\delta'/\beta$, we can conclude from \eqref{eq:left} and \eqref{eq:right} that 
\[
\{ y \in [0,1] \, : \, |\alpha - w(x,y)| \leq  \delta' \}\subseteq \{ y\,:\, |r_\alpha(x)-y|<\delta'/\beta\}.
\]
The conclusion follows since the superset has volume at most $2(\delta'/\beta)=A\delta'$.

\end{proof}

The fact that $\WA{\alpha}$ is diagonally increasing implies that the boundaries $r_\alpha$ and $\ell_\alpha$ are increasing. The following lemma establishes bounds on their slope.

\begin{lemma}\label{Lemma:MinSlopeBoundaries}
Let $w$ be a graphon satisfying Conditions \ref{AssumptionsSimpleWeakIdentifiabilityAssumptions}, and let $f,d,c,\alpha$ be as in the statement of these conditions. Then the boundary $r_\alpha$ is strictly increasing on the domains $[0,\ell_\alpha (1)]$, and has slope bounded away from zero.

Precisely, there exists $\epsilon >0$ so that for all $0\leq x'<x<\ell_\alpha (1)$ with $x-x'<\epsilon$,
\[
r_\alpha (x)-r_\alpha (x')\geq \beta (x-x'),
\]
where $\beta >0 $ is as defined in (\ref{eq:beta}). 

Similarly $\ell_\alpha $ is strictly increasing on  $[r_\alpha (0),1]$,  and there exists $\epsilon >0$ so that, for all $r_\alpha (0)\leq y'<y< 1$ with $y-y'<\epsilon$,
\[
\ell_\alpha (y)-\ell_\alpha (y')\geq \beta (y-y').
\]
\end{lemma}

\begin{proof}
We show first that $\ell_\alpha$ is strictly increasing on the domain  $[r_\alpha (0),1]$. Suppose, by contradiction, that $\ell_\alpha (y_1)=\ell_\alpha(y_2)=x$ for some $r_\alpha (0) \leq y_1<y_2$.  Then $\WA{}(x,y_1)=\WA{}(x,y_2)=\alpha$, and thus $\WA{}(x,\cdot)$ is constant on $[y_1,y_2]$.  This contradicts  Lemma \ref{Lemma:Goodness}. A similar argument shows that $r_\alpha$ is strictly increasing on the domain $[0,\ell_\alpha (1)]$.

Let $d^-,d^+$ be as in \eqref{eq:d-d+}. By Lemma \ref{Lemma:d-d+}, we conclude that for all $x\in [0,\ell_\alpha (1)]$,
\be
x+d < x+d^-\leq r_\alpha (x) \leq x + d^+ <x+ d'.
\ee
Fix $0\leq x'<x\leq  \ell_{\alpha}(1)$ so that $x-x'< d^- -d$, and thus $x+d< x'+d^-$.
Let $y=r_\alpha (x)$ and $y'=r_\alpha (x')$. Since $r_\alpha $ is strictly increasing, $y'<y$. 
Then 
\[
x+d < x'+d^- \leq  y'\leq y\leq x+d^+<x+d'.
\]
Note further that $y',y\in [r_\alpha(0),1]$, and $x=\ell_\alpha (y)$, $x'=\ell_\alpha (y')$. Lemma \ref{lem:decreaseW2} gives bounds on the slope of $\WA{}(x,\cdot)$ on the interval $[x+d,x+d']$, and the above shows that $y,y'$ are in this interval. Therefore the bounds apply, and we have that
\be
\WA{}(x,y')-\alpha &=\WA{}(x,y')-\WA{}(x,y) \geq \beta (y-y')\text{, and}\\
\WA{}(x,y')-\alpha  &= \WA{}(x,y')-\WA{}(x',y')\leq (x-x').
\ee
Combining these inequalities we obtain that $\ell_\alpha(y)-\ell_\alpha (y')=x-x'\geq \beta (y-y')$. This completes the proof that $\ell_\alpha $ has slope at least $\beta$ on $ [r_\alpha (0),1]$.

A similar argument shows that $r_\alpha$ is increasing with slope at least $\beta$ on $[0,\ell_\alpha(1)]$.
\end{proof}

\begin{lemma}
\label{Lemma:Separation}
Let $w$ be a graphon satisfying Conditions \ref{AssumptionsSimpleWeakIdentifiabilityAssumptions}, and let $f,d,c,\alpha$ be as in the statement of these conditions. Then there exists $B>0$ so that $\WA{\alpha}$ is $B$-separated.
\end{lemma}

\begin{proof}
Fix $0\leq x<y\leq 1$. We need to show that there exists a constant $B>0$ so that the region where $\WA{\alpha}(x,\cdot )$ and $\WA{\alpha}(y,\cdot)$ take different values is bounded below by $B(y-x)$. 

For any $z\in [0,1]$,  $\WA{\alpha}(z,x)=1$ and $\WA{\alpha} (z,y)=0$ precisely when $\ell_\alpha (x)\leq z<\ell_\alpha (y)$. Similarly, $\WA{\alpha}(z,x)=0$ and $\WA{\alpha} (z,y)=1$ precisely when $r_\alpha (x)\leq z<r_\alpha (y)$.
 We have that 

\be
\mathrm{Vol}(\{z  \in [0,1] \, : \, |\WA{\alpha}(z,x) - \WA{\alpha}(z,y)| \geq 1\}) &\geq \max( \ell_\alpha(y)-\ell_\alpha (x),r_\alpha(y)-r_\alpha(x)) \\
&\geq \frac{1}{2}(\ell_\alpha(y)-\ell_\alpha (x) + r_\alpha(y)-r_\alpha(x)).
\ee

Let $d^-,d^+$ be as in \eqref{eq:d-d+} and fix  $0 < B\leq  \min\{ \beta , 2(1-2d^+)\}$; this is possible because $d^+<d'\leq 0.5$. We now repeatedly apply Lemma \ref{Lemma:MinSlopeBoundaries}. If $y>x>r_\alpha (0)$, then $\ell_\alpha (y)-\ell_\alpha (x)\geq \beta (y-x)\geq B(y-x)$. If $x<y<\ell_\alpha (1)$, then $r_\alpha(y)-r_\alpha (x)\geq \beta (y-x)\geq B(y-x)$. If $x\leq r_\alpha (0)$ and $y\geq \ell_\alpha (1)$, then
\be
\ell_\alpha(y)-\ell_\alpha (x)+r_\alpha(y)-r_\alpha(x) &= \ell_\alpha (y) -0 + 1-r_\alpha (x)\\
&\geq (y-d^+)+1-(x+d^+)\\
&\geq (1-2d^+) +1-(2d^+)\\
&= 2(1-2d^+)\,\geq B\,\geq B(y-x).
\ee
Therefore, $\WA{\alpha}$ is $B$-separated.

\end{proof}

Combining the above results, we can conclude
\begin{lemma}\label{Lemma:WeakAssumptionsImplyThm}
Let $w$ and $\alpha$ be a graphon and constant satisfying Conditions \ref{AssumptionsSimpleWeakIdentifiabilityAssumptions}.  Then Assumption \ref{AssumptionGoodGraphon} also holds for $w,\alpha$.
\end{lemma}

\begin{proof}
Let $A>0$, $\epsilon_1,\epsilon_2,\epsilon_3,\epsilon_4\in (0,1)$ be so that $w$ is $(A,\epsilon_1)$-good at $\alpha$, and $\WA{\alpha}$ is $\epsilon_2$-separated, has an $\epsilon_3$-split, and is $(\epsilon_{4}, \alpha)$-connected. Such values exist by the earlier lemmas in this section. Taking $\epsilon=\min\{\epsilon_1,\epsilon_2,\epsilon_3,\epsilon_4\}$ completes the proof.
\end{proof}

We are ready to prove our the first of our main results:

\begin{proof} [Proof of Theorems \ref{ThmReconMainRes1}]

As noted earlier, Theorem \ref{ThmMaxOrd} follows immediately from Lemma \ref{LemmaMergeCorrectness}. Theorem \ref{ThmReconMainRes1} then follows immediately from Theorem \ref{ThmMaxOrd} and Lemma \ref{Lemma:WeakAssumptionsImplyThm}, which relates their assumptions.

\end{proof}

\section{Bad events are rare: proofs}

We give proofs deferred from Section \ref{SecMainEstHeuristic}.  This appendix will use the notation from Section \ref{SecMainEstHeuristic} freely, and in particular it will rely on the relationship between the observed random graph $G_{n}$ and the driving randomness $\{U_{i}\}, \, \{U_{ij}\}$.

\subsection{Bad Events for Latent Variables}

Section \ref{SecMainEstHeuristic} introduces the collection of ``bad" events  $\mathcal{A}_1,\dots ,\mathcal{A}_5$ in Definition \ref{DefBadEvents}. We recall the definition for convenience:

\begin{defn}[Bad Events] 

Let graphon $w$, value $\alpha$, random variables $\{ U_i\}$, $\{ U_{i,j}\}$, and graphs $\GA{\alpha}$ and $H_{\alpha}$ be as defined above. Moreover, let $R$, $\{ I_k\}_{k=0}^{4R-3}$ be as in (\ref{DefIntervals}).
We define:
\be
\mathcal{A}_{1} &= \cup_{1 \leq i < j \leq n}(\{(U_{i},U_{j}) \notin \bad{n}{\alpha}\} \cap \{ \GA{ \alpha}(i,j) \neq H_{ \alpha}(i,j)\}) \\
\mathcal{A}_{2} &= \left\{\max_{1 \leq i \leq n} | \bad{n,i}{\alpha} | > \sqrt{n } \, \log(n)^{\cat}\right\} \\
\mathcal{A}_{3} &= \left\{ \min_{0 \leq k  \leq 4R-3} |\{ i\,:\, U_i\in I_k\}|< \frac{n}{2R} \right\} \\
\mathcal{A}_{4} &= \left\{\min_{1 \leq i,j \leq n \, : \, |U_{i}-U_{j}| > \frac{1}{\log(n)^{\cafo}}} \nu (i,j) < \frac{2n}{\log (n)^{\caft}}\right\} \\
&\quad \cup \left\{\min_{1 \leq i,j \leq n \, : \, |U_{i}-U_{j}| > \frac{\log(n)^{\cafi}}{\sqrt{n}}} \nu (i,j) < \sqrt{n} \log(n)^{\cafti} \right\} \\
\mathcal{A}_{5} &= \left\{  \max_{1 \leq i < j \leq n} |\{ k \, : \, U_{k} \in (U_{i},U_{j}) \}| \geq n |U_{i}-U_{j}| +  \sqrt{n}\log(n)  \right\}.\\
\ee

\end{defn}

In this section we prove Corollary \ref{IneqLatentCor}, which says that with high probability none of these events occur. Again, we repeat the statement for convenience:

\begin{corollary} 
Let $\mathcal{A}= \mathcal{A}_1\cup \mathcal{A}_2\cup \mathcal{A}_3\cup \mathcal{A}_4\cup \mathcal{A}_5$. Then $\mathcal{A}^{c}$ holds \wep.
\end{corollary}

 The proofs mainly follow from standard concentration bounds, applied to functions of the independent variables $\{ U_i\}$ and $\{ U_{i,j}\}$. See \textit{e.g.} \cite{boucheron2013concentration} for an overview of concentration phenomena.

%

We begin by checking that the ``approximate" thresholded graph $\GA{ \alpha}$ agrees with the ``correct" thresholded graph $H_{ \alpha} \sim \WA{\alpha}$ outside of the bad set defined in Equation \eqref{EqDefBadSet}:

\begin{lemma} \label{LemmaGoodPointsEq}
Fix $0 < \alpha < 1$ and $n \in \mathbb{N}$. Then $\P [\mathcal{A}_{1}]=n^{-\Omega(\log (n))}$.
That is, \wep\ for all $i,j$ such that $(U_i,U_j)\not \in \bad{n}{\alpha}$, $\GA{\alpha}(i,j)=H_\alpha (i,j)$.
\end{lemma}

\begin{proof}
Recall the definition of the square-threshold graph $\GA{\alpha}$ from Equation \eqref{EqThreshDef}. For $0 \leq a \neq b \leq n$, define $\ell_{a,b} = \textbf{1}_{\{ U_{a,b} < w(U_{a},U_{b})\}}$ to be the indicator function for the edge $(a,b)$ being in $G$. With this notation we have

\be
G^{(2)}(1,2) = \sum_{a=3}^{n} \ell_{1,a} \ell_{a,2},
\ee
Recall that $\GA{\alpha}(1,2)=1$ if and only the above sum is greater than $\alpha (n-2)$. Let $\mathcal{H}$ be the $\sigma$-algebra generated by $(U_{1},U_{2})$; for all $a\geq 3$,
$\E [\ell_{1,a} \ell_{a,2}\, |\, \mathcal{H}]= w^{(2)} (U_1,U_2)$. For $(U_1,U_2)\in \bad{n}{\alpha}$, the bound is clear. In the remaining case, we assume that $(U_1,U_2)\not \in \bad{n}{\alpha}$, and thus $|w^{(2)}(U_1,U_2) - \alpha| > \frac{ \log(n)^{\cbad}}{\sqrt{n }}$. If $w^{(2)}(U_1,U_2)=0$, then $\GA {\alpha}(1,2)=H_\alpha (1,2)=0$, so we also assume $w^{(2)}(U_1,U_2)>0$.
By Hoeffding's bound on the sum of independent variables,
\be 
\P[ \GA{\alpha}(1,2) \neq H_{\alpha}(1,2) | \mathcal{H}] &\leq \P[|\sum_{a=3}^{n} \ell_{1,a} \ell_{a,2} - (n-2)  w^{(2)}(U_{1},U_{2})| > {\log(n)^{\cbad}\sqrt{n}} \, | \, \mathcal{H}] \notag\\
&\leq \P[|\sum_{a=3}^{n} \ell_{1,a} \ell_{a,2} - \E[\sum_{a=3}^{n} \ell_{1,a} \ell_{a,2} \, | \, \mathcal{H}] | >  {\log(n)^{\cbad}\sqrt{n-2}}  \, | \, \mathcal{H}] \\
&=n^{-\Omega(\log(n))},
\ee

Adding back in the trivial cases discounted above, we have the unconditional bound
\be
\{ \GA{\alpha}(1,2) = H_{\alpha}(1,2) \} \cup \{ (U_{1},U_{2}) \in \bad{n}{\alpha}\}
\ee
holds \wep. There is nothing special about the indices $(1,2)$, so we have shown that \wep, $\GA{\alpha}(i,j) = H_{\alpha}(i,j)$ for any choice of $1\leq i<j\leq n$. Therefore, \wep\ the same holds for all pairs $i,j$ simultaneously.
This completes the proof.

\end{proof}

If Assumption \ref{AssumptionGoodGraphon} (item 4) holds, then there exist $A,\delta >0$ so that $w^{(2)}$ is uniformly $(A,\delta)$-good at $\alpha$. For $n$ large enough, $\frac{\log (n)}{\sqrt{n}}<\delta$ and $8A\leq \log (n)$, so for all $y\in [0,1]$,
\[
Vol ( \{ y\in [0,1]\,:\, |\alpha - w^{(2)}(x,y)|\leq \frac{\log (n)}{\sqrt{n}}\}) \leq A\frac{\log (n)}{\sqrt{n}}\leq \frac{\log (n)^2}{8\sqrt{n}}.
\]
Moreover, since $w$ is diagonally increasing, for each $x\in [0,1]$, the set $\{ y\in [0,1]\,:\, |\alpha - w^{(2)}(x,y)|\leq \frac{\log (n)}{\sqrt{n}}\}$ is the union of two intervals, one each around the endpoints $\ell_{\alpha} (x), r_{\alpha}(x)$ of the interval $I_{\alpha}(x)$ associated to $w_{\alpha}$ in Equation \eqref{AltRep01w}. This implies that, for all $x\in [0,1]$,
\be
\label{EqBadRegionWidth}
\bad{n}{\alpha}\subset \{ (x,y): |r_\alpha(x)-y|\leq \frac{\log (n)^2}{8\sqrt{n}}\}\cup \{ (x,y): |\ell_{\alpha}(x)-y|\leq \frac{\log (n)^2}{8\sqrt{n}}\}.
\ee
In the following lemma, this is used to bound the number of pairs of vertices in $\bad{n}{\alpha}$.

\begin{lemma} \label{LemmaFewBadWitnesses}
Let Assumptions \ref{AssumptionGoodGraphon} hold. Then \wep , $\mathcal{A}_{2}^{c}$ holds.
That is, \wep\ for all $i$, $| \bad{n,i}{\alpha} | \leq  \sqrt{n } \, \log(n)^{\cat}$.
\end{lemma}

\begin{proof}
For all $1 \leq i,j \leq n$, let $e_{i,j} = \textbf{1}_{\{j \in \bad{n,i}{\alpha}\}}$, $d=\frac{\log(n)^{2}}{8\sqrt{n}}$, and  $f_{i,j} = \textbf{1}_{\{ |r_{\alpha}(U_{i}) - U_{j}| \leq  d\}\cup \{ |\ell_{\alpha}(U_{i}) - U_{j}| \leq  d\}}$.  If $\mathcal{A}_1^c$ holds, then by Lemma \ref{LemmaGoodPointsEq} and by \eqref{EqBadRegionWidth}, $e_{i,j}\leq f_{i,j}$.

Fix $1\leq i\leq n$ and let $\mathcal{H}=\sigma (U_i)$. Note that $\E [f_{i,j}\,|\, \mathcal{H}] \leq 4d = \frac{\log(n)^{2}}{2\sqrt{n}}$.  By Hoeffding's bound, Equation \eqref{EqBadRegionWidth}, and Lemma \ref{LemmaGoodPointsEq},

\be
\P[\sum_{j \neq i} e_{i,j} > \sqrt{n} \log(n)^{\cat}\,|\, \mathcal{H}]
&\leq \P[\sum_{j \neq i} f_{i,j} > \sqrt{n} \log(n)^{\cat}\,|\, \mathcal{H}] + \P[\mathcal{A}_{1}] \\
&= n^{-\Omega (\log(n))}
\ee

Therefore, \wep\ $| \bad{n,i}{\alpha} |\leq  \sqrt{n } \, \log(n)^{\cat}$ for any particular choice of $i$, and thus also for all $i$ simultaneously.

\end{proof}

\begin{lemma}\label{LemmaTightlyConnected}
We have $\P [\mathcal{A}_{3}]=n^{-\Omega(\log(n))}$.
That is, \wep\ for all $0\leq k\leq 4R-3$, $|\{ i:U_i\in I_k\}|\geq \frac{n}{2R}$.
\end{lemma}

\begin{proof}
For all $0\leq k\leq 4R-3$ and all $0\leq i\leq n$, let $s_{k,i}=\mathbf{1}_{\{ U_i\in I_k\}}$. Then $\E [s_{k,i}]=\frac{3}{4R}$ and, for fixed $k$, the variables $s_{k,i}$, $1\leq i\leq n$ are i.i.d.\ variables. Thus the result follows from Hoeffding's inequality and a union bound.
\end{proof}

Recall that $\nu (i,j)$ is the number of good witnesses of $i$ and $j$, i.e.~the number of vertices that are adjacent to precisely one of $i,j$ in $H_\alpha$, and thus can be used to distinguish $i$ and $j$. Next, we show that vertices $i,j$ with $|U_{i} - U_{j}|$ large enough have many good witnesses, and thus have substantially different neighbourhoods in $H_{\alpha}$:

\begin{lemma} \label{LemmaManyWitnesses}
Let Assumptions \ref{AssumptionGoodGraphon} hold. Then $\P[\mathcal{A}_4]=n^{-\Omega(\log(n))}$. Precisely, \wep\ for all $i,j$ so that  $|U_i-U_j|> \frac{1}{\log (n)^{\cafo} }$ holds, we also have $\nu (i,j)\geq \frac{2n}{\log(n)^{\caft}}$. Similarly, for all
$i,j$ so that $ |U_{i}-U_{j}| > \frac{\log(n)^{\cafi}}{\sqrt{n}}$, we also have $ \nu (i,j) \geq  \sqrt{n} \log(n)^{\cafti}$.
\end{lemma}

\begin{proof}
Fix $0 < i,j < 1$ and let $\mathcal{H}$ be the $\sigma$-algebra generated by $(U_{i},U_{j})$. We consider the $\mathcal{H}$-measurable event $|U_{i} - U_{j}| > d$, where $d=d(n)=\frac{1}{\log(n)^{}}$. Let $e_{k} = \textbf{1}_{\{\WA{\alpha}(U_{i},U_{k}) \neq \WA{\alpha}(U_{j},U_{k})\} }$ be the indicator function that vertex $k$ distinguishes the vertices $i,j$.  By Part 6 of Assumption \ref{AssumptionGoodGraphon}, there exists $\epsilon >0$ so that $w^{(2)}_\alpha$ is $\epsilon$-separated. Assume that $n$ is large enough so that $\epsilon \geq 4\log (n)^{-1}$. Then:
\be
\E[e_{k} \, | \, \mathcal{H}] \geq \epsilon |U_{i}-U_{j}|> \epsilon d \geq \frac{4d}{\log (n)}.
\ee
Now conditional on $\mathcal{H}$, $\nu (i,j)=\sum_{k \in S\setminus\{i,j\}} e_{k}$ is the sum of $(n-2)$ i.i.d.\ Bernouilli variables. Thus, by Hoeffding's inequality,
\be
\P[\{ \sum_{k \in V(G)\setminus\{i,j\}} e_{k} <  \frac{2dn}{\log (n)} \} \, | \, \mathcal{H}] & = n^{-\Omega(\log(n))}.
\ee
Applying a union bound completes the proof of the first part. For the second part, the same argument applies, but with $d=
\frac{\log(n)^{\cafi}}{\sqrt{n}}$.

\end{proof}

Finally, we show that the number of vertices between $i,j$ in the true order is roughly proportional to $|U_{i}-U_{j}|$ with high probability:

\begin{lemma} \label{LemmaCompOrders}
For all $i,j$ so that $U_i<U_j$,
we have that $|\{ k \, : \, U_{k} \in (U_{i},U_{j}) \}| <  n |U_{i}-U_{j}| + \sqrt{n}\log (n) \log(n) \}$ \wep. Thus,  \wep\ $\mathcal{A}_5^c$ holds.
\end{lemma}

\begin{proof}
This is just the rephrased version of Hoeffding's inequality.
\end{proof}


%
%

\subsection{Bad Events for Subsampled Graphs} \label{AppSecBadSmallSets}

We give the proofs from Section \ref{sec:PropertiesSamples}, using the notation from that section. Recall that $\mathcal{A}= \mathcal{A}_1\cup \mathcal{A}_2\cup \mathcal{A}_3\cup \mathcal{A}_4\cup \mathcal{A}_5$, where $\mathcal{A}_1\cup \mathcal{A}_2\cup \mathcal{A}_3\cup \mathcal{A}_4\cup \mathcal{A}_5$ are the ``bad events" described in Definition \ref{DefBadEvents}. Most of this section is concerned with analogous ``bad events for samples" given in Definition \ref{BadEventsSub}, repeated here for convenience:

\begin{defn}[Bad Events for Samples] 

Let graphon $w$, value $\alpha$, random variables $\{ U_i\}$, $\{ U_{i,j}\}$, and graphs $\GA{\alpha}$ and $H_{\alpha}$ be as defined above.  Moreover, let $R$, $\{ I_k\}_{k=0}^{4R-3}$ be as in (\ref{DefIntervals}). Finally, we let $S \subset V$ and $m=|S|$.

\be
\mathcal{A}_{1}' &= \{\GA{ \alpha}(S) \not= H_{\alpha}(S)\}\\
\mathcal{A}_{3}' &= \left\{ \min_{0 \leq k  \leq 4R-3} |\{ i\in S\,:\, U_i\in I_k\}|< \frac{m}{4R}  \right\} \cup
\{ |S \backslash S''| > \log(n)^{\cS}\} \\
%
\mathcal{A}_{4}' &= \left\{\min_{i,j \in S \, : \, |U_{i}-U_{j}| > \frac{1}{\log(n)}} \nu_{S}(i,j) <  \log(n)^3 \right\} \\
\mathcal{A}_{5}' &= \left\{  \max_{i,j \in S} |\{ k \, : \, U_{k} \in (U_{i},U_{j}) \}| \geq m |U_{i}-U_{j}| + 2 \log(n)^{3.5}  \right\}.\\
\ee
\end{defn}

We first prove Lemma \ref{LemmaSampleEqual}, restated here for convenience:

\newtheorem*{lemma9}{Lemma \ref{LemmaSampleEqual}}

\begin{lemma9}
Let $S$ be a uniformly chosen subset of $V$, of size  $m = \log(n)^{\csketch}$. Then on the $\mathcal{F}$-measurable event $\mathcal{A}^{c}$,
\[
\P [ \mathcal{A}_1'\,|\,\mathcal{F}]=O( n^{-0.49}).
\]
Furthermore, for any fixed $p,q\in V(G)$ satisfying $(U_p,U_q)\not\in\bad{n}{\alpha}$, we have
\[
\frac{\P [ \mathcal{A}_1'\cap\{ p,q\in S\}\,|\,\mathcal{F}]}{\P [\{ p,q\in S\} | \mathcal{F}]}=O( n^{-0.49}),
\]
where again both probabilities are on $\mathcal{A}^{c}$.
\end{lemma9}

\begin{proof} 
It follows from Definition \ref{EqDefBadWitness} that $\GA{\alpha}(i,j)\not= H_\alpha (i,j)$ precisely when $j\in \bad{n,i}{\alpha}$.  If $\mathcal{A}_2^c$ holds, then the size of $\bad{n,i}{\alpha}$ is bounded by $\sqrt{n}\log (n)^{\cat}$.
Using this bound we obtain that, for a fixed $1\leq i\leq n$,
\be
\label{eqnNotEqual}
&  \P [\bigcup_{j\in S\setminus\{ i\}}\{\GA{ \alpha}(S)(i,j) \not= H_{\alpha}(S)(i,j)\}\,|\,\sigma (\{i\in S\}),\,\mathcal{F}]\\
=& \P [\bad{n,i}{\alpha}\cap S\not= \emptyset \,|\,\sigma (\{ i\in S\}),\,\mathcal{F}]\\
\leq &(m-1) \left( \frac{|\bad{n,i}{\alpha}|}{n}\right) \\
\leq &\frac{\log (n)^{7}}{\sqrt{n}},
\ee
where all conditional probabilities are on the event $\mathcal{A}^{c}$. By a union bound, $\P [\mathcal{A}'_1\,|\,\mathcal{F}]\leq m\left(\frac{\log (n)^{7}}{\sqrt{n}}\right)=O(n^{-0.49})$ on $\mathcal{A}^{c}$.

Fix $p,q\in V(G)$,  so that $(U_p,U_q)\not\in\bad{n}{\alpha}$. If $(\mathcal{A}_1\cup \mathcal{A}_2)^c$ holds, then $q\not
\in \bad{n,p}{\alpha}$ and $p\not\in \bad{n,q}{\alpha}$. Thus, for a fixed $1\leq i\leq n$, $i\not\in \{ p,q\}$, Equation \ref{eqnNotEqual} holds for $S\setminus \{ i,p,q\}$, so
\[
\P [\cup_{j\in S\setminus\{ i,p,q\}} \{ \GA{ \alpha}(S)(i,j) \not= H_{\alpha}(S)(i,j)\}\,|\,\sigma (\{i\in S\}),\,\mathcal{F}] \textbf{1}_{\{p,q\in S\}}\leq \frac{\log (n)^{7}}{\sqrt{n}}
\]
on $\mathcal{A}^{c}$. Again by a union bound,
$$\P [\mathcal{A}_1'\,|\,\mathcal{F}]\textbf{1}_{\{p,q\in S\}}=O(n^{-0.49})$$
on $\mathcal{A}^{c}$, from which the result follows.
\end{proof}

We next prove Lemma \ref{LemmaIndGraphConnected}, restated here for convenience:

\newtheorem*{lemma10}{Lemma \ref{LemmaIndGraphConnected}}
\begin{lemma10}
Let $S$ be a uniformly chosen subset of $V$, of size  $m = \log(n)^{\csketch}$, and let $S''$ be as defined in \eqref{DefS''}.
We have $\P [ \mathcal{A}_3'\,|\,\mathcal{F}]=n^{-\Omega(\log(n))}$ on the $\mathcal{F}$-measurable event $\mathcal{A}^{c}$. That is, \wep\ for all $0\leq k\leq 4R-3$, both $|S\cap V_k|\geq \frac{\log (n)^{\csketch}}{4R}$ and also $|S\setminus S''|\leq \log(n)^{\cS}$ on $\mathcal{A}^{c}$. Moreover, for any fixed $p,q\in V(G)$,
\[
\frac{\P [ \mathcal{A}_3'\cap\{ p,q\in S\}\,|\,\mathcal{F}]}{\P [\{ p,q\in S \} | \, \mathcal{F}]}=O( n^{-\Omega (\log (n)) }),
\]
where again both probabilites are on the event $\mathcal{A}^{c}$.
\end{lemma10}

\begin{proof}  
Assume $\mathcal{A}_3^c$ holds, so for all $0\leq k\leq 4R-3$, $|V_k|\geq \frac{n}{2R}$.  For all $1\leq i\leq n$, let $e_{i}=\mathbf{1}_{\{ i\in S\}}$.
Fix $0\leq k\leq 4R-3$. Then $|\{ i\in S\,:\, U_i\in I_k\}|=\sum_{\{ i\in V_k\}}e_i$ and  $\E [\sum_{\{ i\in V_k\}}e_i\,|\, \mathcal{F},\mathcal{A}^c]\geq \frac{m}{2R}$.
By Azuma's inequality, we obtain that
\be
\label{Eq1}
\P [\{\sum_{\{ i\in V_k\}} e_i < \frac{m}{4R}\}]=n^{-\Omega(\log (n))}.
\ee
Applying a union bound, we obtain the first part of the result.

To prove the second part, note first that, for fixed $1\leq i\leq n$, by definition
\[
\{ i\in S''\}=\{ i\in S\}\cap \{ i\not\in\cup_{j\in S\setminus \{ i\}} \bad{n,j}{\alpha}\}.
\]
If $\mathcal{A}_2^c$ holds, then for all $j$, $| \bad{n,j}{\alpha}|\leq \sqrt{n}\log (n)^{\cat}$, and thus
\be
\P[ \{i \in S''\} \, | \, \sigma(S \backslash \{i\}),\,\mathcal{F}]  \geq 1 - |S| \frac{\log(n)^{\cat}}{\sqrt{n}}\geq 1-\frac{2 \log(n)^{7}}{\sqrt{n}}
\ee
on $\mathcal{A}^c$. Applying Azuma's inequality,
\be \label{IneqADoublePrime2}
\P[\{|S''| < m - \log^{3}(n)\}\,|\, \mathcal{F}] = n^{-\Omega(\log(n))}
\ee
on $\mathcal{A}^c$. This completes the proof of the bound on $\P[\mathcal{A}_{3}' | \mathcal{F}]$.

To prove the second bound, fix $p,q\in V(G)$. An analogous argument applied to the set $S-\{ p,q\}$ gives the result.
\end{proof}

We next prove Lemma \ref{LemmaDegreeConnected}, restated here for convenience:

\newtheorem*{lemma11}{Lemma \ref{LemmaDegreeConnected}}
\begin{lemma11}

Let $S$ be a uniformly chosen subset of $V$, of size  $m = \log(n)^{\csketch}$, and let $S''$ as defined in \eqref{DefS''}. For each $i\in S$, let $D_S(i)=|\{ j\in S: \GA{\alpha}(i,j)=1\}|$ be the number of neighbours of $i$ in $\GA{\alpha}(S)$. Then for all $n$ sufficiently large,
\[
\{ H_{\alpha} (S) \text{ is not connected}\}\, \cup \, \{ H_{\alpha} (S'') \text{ is not connected}\}\cup \left\{ \min_{i\in S} D_S(i) < \frac{m}{8R} \right\}\subset \mathcal{A}_3'.
\]

\end{lemma11}

\begin{proof}  
For any set $S\subset V$, if $H_\alpha (S)$ is not connected then there must be $0\leq k\leq 4R-3$ so that $S\cap V_k=\emptyset$. Thus it follows immediately that $\{ H_{\alpha} (S) \text{ is not connected}\}\subset \mathcal{A}_3'$.

Moreover, if $(\mathcal{A}_3')^c$ holds, then $|S\setminus S''|\leq \log (n)^3$ and for each $0\leq k\leq 4R-3$,
\[
|S\cap V_k|\geq \log(n)^{\csketch}/(4R).
\]
Thus we get that $|S''\cap V_k|\geq |S\cap V_k|-|S/S''|\geq \log (n)^{\csketch}/(8R)$ for $n$ sufficiently large. Thus $V_k\cap S''\not=\emptyset$ for all $k$, which implies that $H_\alpha (S'')$ is connected.

By the choice of $R$, for any two vertices $i,j\in V_k$, $H_\alpha (U_i,U_j)=1$. Fix $i\in S$. By definition, for each $j\in S''$, $H_{\alpha}(i,j)=\GA{\alpha}(i,j)$. Let $k$ be so that $U_i\in I_k$. Then for all vertices $j\in S''\cap V_k$, $\GA{\alpha} (i,j)=1$. As argued above, there are at least $ \log (n)^{\csketch}/(8R)$ such vertices. This shows that $\{ \min_{i\in S} D_S(i) < \frac{\log(n)^{\csketch}}{8R}\}\subset \mathcal{A}_3'$.
\end{proof}

We next prove Lemma \ref{LemmaCondProbGoodSub}, restated here for convenience:

\newtheorem*{lemma12}{Lemma \ref{LemmaCondProbGoodSub}}

\begin{lemma12}
Let $S$ be a uniformly chosen subset of $V$, of size  $m = \log(n)^{\csketch}$. Then
\[
\P [\mathcal{A}_4'\,|\,\mathcal{F}]= n^{-\Omega (\log n)}
\]
on the $\mathcal{F}$-measurable event $\mathcal{A}^c$.
That is, \wep\ for all $i,j\in S$ so that $|U_{i}-U_{j}| > \displaystyle\frac{1}{\log (n)^{\cafo}}$, $ \nu_{S}(i,j) \geq \log(n)^3  $.
\end{lemma12}

\begin{proof}  
Fix $1\leq i<j\leq n$ so that $|U_{i}-U_{j}| > \frac{1}{\log(n)^{\cafo}}$.
If $\mathcal{A}_4^c$ holds, then $ \nu (i,j) \geq  \frac{2n}{\log (n)^{\caft}} $. Thus $\E [\nu_{S}(i,j)\,|\mathcal{F},\mathcal{A}^c] \geq 2m/\log(n)^{\caft}=2\log(n)^{3}$. The result follows from Azuma's inequality and a union bound.
\end{proof}

Finally, we prove Lemma \ref{LemmaCompOrdersS}, restated here for convenience:

\newtheorem*{lemma13}{Lemma \ref{LemmaCompOrdersS}}
\begin{lemma13} 
On the $\mathcal{F}$-measurable event $\mathcal{A}^{c}$,
\[
\P [\mathcal{A}_5'\,|\,\mathcal{F}]= n^{-\Omega (\log n)}.
\]
That is, \wep\ for all $i,j\in S$, $|\{ k\,:\, U_k\in (U_i,U_j)\}|\leq  m |U_{i}-U_{j}| + 2 \log(n)^{3.5}$.
\end{lemma13}

\begin{proof} 
Fix $1\leq i,j\leq n$, so that $U_i<U_j$, and let $W_{i,j}=\{ k\,:\, U_k\in (U_i,U_j)\}$. For each $k\in W_{i,j}$, let  $e_k=\mathbf{1}_{\{  k\in S \}}$. Then $|S\cap W_{i,j}|=\sum_{i\in W_{i,j}} e_k$, and $\E [\sum_{i\in W_{i,j}}e_k\,|\,\mathcal{F}]=(m/n)|W_{i,j}|$.

If $\mathcal{A}_5^c$ holds, then $(m/n)|W_{i,j}|\leq m|U_j-U_i|+ m\frac{\log (n)}{\sqrt{n}}\leq m|U_i-U_j|+ 1$ for $n$ sufficiently large.  By Azuma's inequality, on $\mathcal{A}^{c}$
\[
\P[\mathcal{A}_5'\,|\,\mathcal{F}] \leq  \P[ |\sum_{i\in V_k} e_k - \E [\sum_{i\in V_k} e_k \, |\, \mathcal{F}]| > \log(n)^{3.4} \,|\,\mathcal{F}]=n^{-\Omega(\log (n))}.
\]
The result then follows by a union bound.
\end{proof}

\section{A Simple Anticoncentration Lemma} \label{SecAppAntiConc}

We need the follow technical lemma, which follows almost immediately from Corollary 1.8 of \cite{chatterjee2017general}:

\begin{lemma} \label{LemmaImmediateFromChat}
Let $U[1],U[2],\ldots \stackrel{i.i.d.}{\sim} \mathrm{Unif}([0,1])$, and for $n \in \mathbb{N}$ and $c > 0$ let $\tilde{U}_{\frac{c}{\sqrt{n}}}[i] = e^{1 - \frac{c}{\sqrt{n}}} U[i]$. Then there exists a universal constant $A > 0$ not depending on $c$ or $n$ so that\footnote{Note that the following bound refers to only the \textit{distribution} of the vectors $U$, $U_{\frac{c}{\sqrt{n}}}$. In particular, it ``forgets" the coupling $\tilde{U}_{\frac{c}{\sqrt{n}}}[i] = e^{1 - \frac{c}{\sqrt{n}}} U[i]$ between the vectors. }
\be  \label{IneqUVers}
\| \mathcal{L}((U[1],\ldots,U[n])) - \mathcal{L}((\tilde{U}_{\frac{c}{\sqrt{n}}}[1],\tilde{U}_{\frac{c}{\sqrt{n}}}[1])) \|_{\mathrm{TV}} \leq Ac.
\ee
\end{lemma}

\begin{proof}
Let $V[i] = -\log(U[i])$, and note that $V[1], V[2],\ldots$ are i.i.d. random variables with exponential distribution and mean 1. Define $\tilde{V}_{\frac{c}{\sqrt{n}}}[i] = (1 - \frac{c}{\sqrt{n}}) V[i]$. By Corollary 1.8 of \cite{chatterjee2017general}, there exists a universal constant $A> 0$ such that
\be \label{IneqVVers}
\| \mathcal{L}((V[1],\ldots,V[n])) - \mathcal{L}((\tilde{V}_{\frac{c}{\sqrt{n}}}[1],\tilde{V}_{\frac{c}{\sqrt{n}}}[1])) \|_{\mathrm{TV}} \leq Ac.
\ee

But the random variables appearing in Inequality \eqref{IneqUVers} and \eqref{IneqVVers} differ only by the monotone 1-1 transformations
\be
V[i] = -\log(U[i]), \, \hat{V}_{\frac{c}{\sqrt{n}}}[i] = -\log(\hat{U}_{\frac{c}{\sqrt{n}}}[i]),
\ee
and so the TV distances are in fact equal.
\end{proof}

We give the proof of Theorem \ref{ThmNoDownstreaming}:

\begin{proof} [Proof of Theorem \ref{ThmNoDownstreaming}]

Our coupling is based on the following three observations:

\begin{enumerate}
\item Fix $c > 0$ and define $\delta_{n} = \frac{c}{\sqrt{n}}$. By Lemma \ref{LemmaImmediateFromChat}, there exists a universal constant $0 < A < \infty$ so that
\be \label{EqSwitchHatAppr}
\| \mathcal{L}(U^{(1)}) - \mathcal{L}(\hat{U}_{\delta_{n}}^{(1)}) \|_{\mathrm{TV}} \leq Ac.
\ee

\item For \textit{any} value of $\delta, w$ and sequences $u^{(1)}, u^{(2)}$, we have
\be \label{EqSwitchHatEq}
G(w;\hat{u}^{(1)}_{\delta}, u^{(2)}) = G(\hat{w};u^{(1)}_{\delta}, u^{(2)}).
\ee
\item For \textit{any} value of $u,\delta \in (0,1)$ we have by Taylor's theorem
\be
|u - \tilde{u}_{\delta}| \geq | \delta \log(u)| - \frac{e}{2} \delta^{2} \log(u)^{2}.
\ee
In particular, for $\delta_{n}$ as above and any $\eta > 0$, there exists a universal constant $B = B(\eta)$ so that
\be \label{IneqMostlyBigPert}
\P[|U[i] - \tilde{U}_{\delta_{n}}[i]| \geq B \frac{c}{\sqrt{n}}] \geq 1 - \eta.
\ee
\end{enumerate}

Define the event $\mathcal{A}_{c} = \{V = \tilde{U}_{\frac{c}{\sqrt{n}}} \}$. By Equation \eqref{EqSwitchHatAppr}, there is a coupling of $(U,V)$ such that
\be
\P[\mathcal{A}_{c}] \geq Ac.
\ee

On this event, Equation \eqref{ConcEqualGraphs} holds from Equation \eqref{EqSwitchHatEq}, Equation \eqref{ConcEqualPerms} is immediate from the definition, and Inequality \eqref{ConcBigDisp} follows immediately from Inequality \eqref{IneqMostlyBigPert} and Hoeffding's inequality. This completes the proof.
\end{proof}



\begin{thebibliography}{00}
\bibitem[\protect \citeauthoryear{Atkins, Boman, and Hendrickson}{Atkins
et~al.}{1998}]{atkins1998spectral}
Atkins, J.~E., E.~G. Boman, and B.~Hendrickson (1998). A spectral
algorithm for seriation and the consecutive ones problem. \emph{SIAM
Journal on Computing\/}~\emph{28\/}(1), 297--310.

\bibitem[\protect \citeauthoryear{Bagaria, Ding, Tse, Wu and Xu}{Bagaria et~al.}{2020}]{RePEc:inm:oropre:v:68:y:2020:i:1:p:53-70}
Bagaria, V., J. Ding, D. Tse, Y. Wu, and J. Xu (2020).
Hidden Hamiltonian cycle recovery via linear programming.
\emph{Operations Research\/}~\emph{68\/}(1), 53--70.


\bibitem[\protect \citeauthoryear{Barth\'{e}lemy and Brucker}{Barth\'{e}lemy and
Brucker}{2001}]{barthelemy2001overlappingclustering}
Barth\'{e}lemy, J.-P. and F.~Brucker (2001). NP-hard approximation
problems in overlapping clustering. \emph{J. Classif.\/}~\emph{18},
159--183.

\bibitem[\protect \citeauthoryear{Boucheron}{Boucheron, Lugosi and Massart}{2013}]{boucheron2013concentration}
Boucheron, S., Lugosi, G.~and Massart, P. (2013). Concentration
inequalities: A nonasymptotic theory of independence. \emph{Oxford
University Press}.


\bibitem[\protect\citeauthoryear{Cai, Liang and Rakhlin}{Cai et al.}{2017}]{7924316}
Cai, T., T.~Liang and A.~Rakhlin (2017).
On detection and structural reconstruction of small-world random
  networks.
\emph{IEEE Transactions on Network Science and Engineering\/}~\emph{4\/}(3), 165--176.

\bibitem[\protect \citeauthoryear{{\c{C}}ela, Deineko, and Woeginger}{{\c{C}}ela
et~al.}{2015}]{cela2015QAProbinson}
\c{C}ela, E., V.~G. Deineko, and G.~J. Woeginger (2015). A
new tractable case of the QAP with a Robinson matrix. In
\emph{Combinatorial Optimization and Applications}, pp.\ 709--720. Springer.

\bibitem[\protect \citeauthoryear{Chatterjee}{Chatterjee}{2015}]{chatterjee2015matrix}
Chatterjee, S. (2015). Matrix estimation by universal singular
value thresholding. \emph{The Annals of Statistics\/}~\emph{43\/}(1),
177--214.

\bibitem[\protect \citeauthoryear{Chatterjee}{Chatterjee}{2017}]{chatterjee2017general}
Chatterjee, S. (2019). A general method for lower bounds on fluctuations
of random variables. \emph{The Annals of Probability\/}~\emph{47\/}(4),
2140--2171.

\bibitem[\protect \citeauthoryear{Chepoi, Fichet, and Seston}{Chepoi
et~al.}{2009}]{chepoi2009seriationhardness}
Chepoi, V., B.~Fichet, and M.~Seston (2009). Seriation in the
presence of errors: NP-hardness of $\ell _{\infty }$-fitting Robinson
structures to dissimilarity matrices. \emph{J. Classification\/}~\emph{26\/}(3),
279--296.

\bibitem[\protect \citeauthoryear{Chepoi and Seston}{Chepoi and
Seston}{2011}]{chepoi2011seriationapproximation}
Chepoi, V. and M.~Seston (2011, Apr). Seriation in the presence
of errors: A factor 16 approximation algorithm for $\ell _{\infty }$-fitting
Robinson structures to distances. \emph{Algorithmica\/}~\emph{59\/}(4),
521--568.

\bibitem[\protect \citeauthoryear{Chuangpishit, Ghandehari, Hurshman, Janssen,
and Kalyaniwalla}{Chuangpishit et~al.}{2015}]{chuangpishit2015linear}
Chuangpishit, H., M.~Ghandehari, M.~Hurshman, J.~Janssen, and N.~Kalyaniwalla
(2015). Linear embeddings of graphs and graph limits.
\emph{Journal of Combinatorial Theory, Series B\/}~\emph{113},
162--184.

\bibitem[\protect \citeauthoryear{Chuangpishit, Ghandehari, and
Janssen}{Chuangpishit et~al.}{2017}]{chuangpishit2017uniform}
Chuangpishit, H., M.~Ghandehari, and J.~Janssen (2017). Uniform
linear embeddings of graphons. \emph{European Journal of Combinatorics\/}~\emph{61},
47--68.

\bibitem[\protect \citeauthoryear{Corneil}{Corneil}{2004}]{corneil2004}
Corneil, D.~G. (2004). A simple 3-sweep LBFS algorithm for
the recognition of unit interval graphs. \emph{Discrete Applied
Mathematics\/}~\emph{138\/}(3), 371 -- 379.


\bibitem[\protect \citeauthoryear{Ding, Wu, Xu and Yang}{Ding et~al.}{2021}]{DingSmallWorld}
Ding, J., Y.~Wu, J.~Xu and D.~Yang (2021).
Consistent recovery threshold of hidden nearest neighbor graphs.
\emph{IEEE Trsnsactions on Information Theory\/}~\emph{67\/}(8), 5211--5229.

\bibitem[\protect \citeauthoryear{Evangelopoulos, Brockmeier, Mu, and
Goulermas}{Evangelopoulos et~al.}{2017}]{evangelopoulos2017non-convex}
Evangelopoulos, X., A.~J. Brockmeier, T.~Mu, and J.~Y. Goulermas (2017).
A graduated non-convexity relaxation for large scale seriation.
In \emph{Proc.~SIAM 2017 Intl.~Conf.~Data Mining}, pp.\ 462--470.

\bibitem[\protect \citeauthoryear{Flammarion, Mao, and Rigollet}{Flammarion
et~al.}{2019}]{flammarion2019optimal}
Flammarion, N., C.~Mao, and P.~Rigollet (2019, 02). Optimal rates
of statistical seriation. \emph{Bernoulli\/}~\emph{25\/}(1), 623--653.

\bibitem[\protect \citeauthoryear{Fogel, d\'{}Aspremont, and Vojnovic}{Fogel
et~al.}{2016}]{fogel2016spectralranking}
Fogel, F., A.~d\'{}Aspremont, and M.~Vojnovic (2016). Spectral
ranking using seriation. \emph{J. Mach. Learn. Res.\/}~\emph{17},
Paper No. 88, 45.

\bibitem[\protect \citeauthoryear{Fogel, Jenatton, Bach, and
D\'{}Aspremont}{Fogel et~al.}{2013}]{fogel2013convex}
Fogel, F., R.~Jenatton, F.~Bach, and A.~D\'{}Aspremont (2013).
Convex relaxations for permutation problems. In C.~J.~C.
Burges, L.~Bottou, M.~Welling, Z.~Ghahramani, and K.~Q. Weinberger (Eds.),
\emph{Advances in Neural Information Processing Systems 26}, pp.\ 1016--1024.
Curran Associates, Inc.

\bibitem[\protect \citeauthoryear{Fortin and Tseveendorj}{Fortin and
Tseveendorj}{2015}]{fortin2015signchanges}
Fortin, D. and I.~Tseveendorj (2015). Minimizing sign changes
rowwise: Consecutive ones property and beyond. In G.~O. Tost
and O.~Vasilieva (Eds.), \emph{Analysis, Modelling, Optimization, and Numerical
Techniques}, Cham, pp.\ 37--48. Springer International Publishing.

\bibitem[\protect \citeauthoryear{Gao, Lu, and Zhou}{Gao
et~al.}{2015}]{gao2015rate}
Gao, C., Y.~Lu, and H.~H. Zhou (2015). Rate-optimal graphon estimation.
\emph{The Annals of Statistics\/}~\emph{43\/}(6), 2624--2652.

\bibitem[\protect \citeauthoryear{Ghandehari and Janssen}{Ghandehari and
Janssen}{2020}]{ghandehari2020sequences}
Ghandehari, M. and J.~Janssen (2020). Graph sequences sampled
from Robinson graphons. \emph{arXiv preprint arXiv:2005.05253v2}.

\bibitem[\protect \citeauthoryear{Laurent and Seminaroti}{Laurent and
Seminaroti}{2015}]{laurent2015ToeplitzRobinsonian}
Laurent, M. and M.~Seminaroti (2015). The quadratic assignment
problem is easy for Robinsonian matrices with Toeplitz structure.
\emph{Operations Research Letters\/}~\emph{43\/}(1), 103 -- 109.

\bibitem[\protect \citeauthoryear{Laurent and Seminaroti}{Laurent and
Seminaroti}{2017}]{laurent2016similarityfirst}
Laurent, M. and M.~Seminaroti (2017). Similarity-first search:
a new algorithm with application to Robinsonian matrix recognition.
\emph{SIAM J.~Discrete Math.\/}~\emph{31\/}(3), 1765--1800.

\bibitem[\protect \citeauthoryear{Liiv}{Liiv}{2010}]{Liiv2010}
Liiv, I. (2010, April). Seriation and matrix reordering methods:
An historical overview. \emph{Stat. Anal. Data Min.\/}~\emph{3\/}(2),
70--91.

\bibitem[\protect \citeauthoryear{Lim and Wright}{Lim and
Wright}{2014}]{lim2014convex}
Lim, C.~H. and S.~Wright (2014). Beyond the Birkhoff polytope:
Convex relaxations for vector permutation problems. In Z.~Ghahramani,
M.~Welling, C.~Cortes, N.~D. Lawrence, and K.~Q. Weinberger (Eds.), \emph{Advances
in Neural Information Processing Systems 27}, pp.\ 2168--2176. Curran
Associates, Inc.

\bibitem[\protect \citeauthoryear{Little, Xie, and Sun}{Little
et~al.}{2018}]{little2018analysis}
Little, A., Y.~Xie, and Q.~Sun (2018). An analysis of classical
multidimensional scaling. \emph{arXiv preprint arXiv:}{1812.11954}.

\bibitem[\protect \citeauthoryear{Lov{\'a}sz}{Lov{\'a}sz}{2012}]{lovasz2012book}
Lov\'{a}sz, L. (2012). \emph{Large networks and graph limits},
Volume~60 of \emph{American Mathematical Society Colloquium Publications}.
American Mathematical Society, Providence, RI.

\bibitem[\protect \citeauthoryear{Lov{\'a}sz and Szegedy}{Lov{\'a}sz and
Szegedy}{2006}]{lovasz2006limits}
Lov\'{a}sz, L. and B.~Szegedy (2006). Limits of dense graph sequences.
\emph{Journal of Combinatorial Theory, Series B\/}~\emph{96\/}(6),
933--957.

\bibitem[\protect \citeauthoryear{Mao, Pananjady, and Wainwright}{Mao
et~al.}{2018a}]{mao2018breaking}
Mao, C., A.~Pananjady, and M.~J. Wainwright (2018a). Breaking
the $1/\sqrt{n}$ barrier: Faster rates for permutation-based models in
polynomial time. \emph{Conference on Learning Theory (CoLT)\/}~\emph{75},
2037--2042.

\bibitem[\protect \citeauthoryear{Mao, Pananjady, and Wainwright}{Mao
et~al.}{2020}]{mao2018towards}
Mao, C., A.~Pananjady, and M.~J. Wainwright (2020). Towards
optimal estimation of bivariate isotonic matrices with unknown permutations.
\emph{The Annals of Statistics\/}~\emph{48\/}(6), 3183 -- 3205.

\bibitem[\protect \citeauthoryear{Moitra and Valiant}{Moitra and
Valiant}{2010}]{moitra2010settling}
Moitra, A. and G.~Valiant (2010). Settling the polynomial learnability
of mixtures of Gaussians. In \emph{2010 IEEE 51st Annual Symposium
on Foundations of Computer Science}, pp.\ 93--102. IEEE.

\bibitem{doi:10.1057/palgrave.ivs.9500042}
Steven~A. Morris, Benyam Asnake, and Gary~G. Yen.
\newblock Dendrogram seriation using simulated annealing.
\newblock {\em Information Visualization}, 2(2):95--104, 2003.

\bibitem[\protect \citeauthoryear{Natik and Smith}{Natik and
Smith}{2020}]{natik2020consistency}
Natik, A. and A.~Smith (2020). Consistency of spectral seriation.
\emph{arXiv preprint arXiv:}{2112.04408}.

\bibitem[\protect \citeauthoryear{Petrie}{Petrie}{1899}]{petries99prehistoric}
Petrie, W.~F. (1899). Sequences in prehistoric remains.
\emph{The Journal of the Anthropological Institute of Great Britain
and Ireland\/}~\emph{29\/}(3/4), 295--301.

\bibitem{Putnam_2016}
Putnam, N.H., B.~L. O'Connell, J.~C. Stites, B.~J. Rice,
  M. Blanchette, R. Calef, C.~J. Troll, A. Fields, P.~D.
  Hartley, and C.~W. Sugnet (2016).
Chromosome-scale shotgun assembly using an in vitro method for
  long-range linkage.
\emph{Genome Research\/}~\emph{26\/}(3):342350.

\bibitem[\protect \citeauthoryear{Rocha, Janssen, and Kalyaniwalla}{Rocha
et~al.}{2018}]{rocha2018recovering}
Rocha, I., J.~Janssen, and N.~Kalyaniwalla (2018). Recovering
the structure of random linear graphs. \emph{Linear Algebra and
its Applications\/}~\emph{557}, 234--264.

\bibitem[\protect \citeauthoryear{Rose, Tarjan, and Lueker}{Rose
et~al.}{1976}]{Rose1976LexBFS}
Rose, D., R.~Tarjan, and G.~Lueker (1976). Algorithmic aspects
of vertex elimination on graphs. \emph{SIAM J.~Computing\/}~\emph{5},
266--284.

\bibitem[\protect \citeauthoryear{Von~Luxburg}{Von~Luxburg}{2007}]{von2007tutorial}
Von~Luxburg, U. (2007). A tutorial on spectral clustering.
\emph{Statistics and computing\/}~\emph{17\/}(4), 395--416.

\bibitem[\protect \citeauthoryear{Vuokko}{Vuokko}{2010}]{vuokko2010spectralordering}
Vuokko, N. (2010). \emph{Consecutive ones property and spectral
ordering}, in \emph{Proceedings of the SIAM International Conference on Data Mining, SDM 2010}, pp.\ 350--360.
\end{thebibliography}
\end{document}